\documentclass[11 pt, reqno]{amsart}

\usepackage{amsfonts}
\usepackage{amssymb, amsmath, amsthm, enumitem}
\usepackage[colorlinks=true, linkcolor = blue, citecolor= Green]{hyperref}
\usepackage[alphabetic,lite]{amsrefs}
\usepackage{amscd}   
\usepackage{fullpage}
\usepackage[all]{xy} 
\usepackage{faktor}
\usepackage[utf8x]{inputenc}
\usepackage{verbatim}
\usepackage{mathrsfs}
\usepackage{epstopdf}
\usepackage[normalem]{ulem}
\usepackage{setspace}
\DeclareMathAlphabet{\mathpzc}{OT1}{pzc}{m}{it}
\DeclareFontEncoding{OT2}{}{} 
\newcommand{\textcyr}[1]{%
 {\fontencoding{OT2}\fontfamily{wncyr}\fontseries{m}\fontshape{n}\selectfont #1}}

\usepackage[usenames,dvipsnames]{color}

\makeatletter
\def\@tocline#1#2#3#4#5#6#7{\relax
  \ifnum #1>\c@tocdepth 
  \else
    \par \addpenalty\@secpenalty\addvspace{#2}%
    \begingroup \hyphenpenalty\@M
    \@ifempty{#4}{%
      \@tempdima\csname r@tocindent\number#1\endcsname\relax
    }{%
      \@tempdima#4\relax
    }%
    \parindent\z@ \leftskip#3\relax \advance\leftskip\@tempdima\relax
    \rightskip\@pnumwidth plus4em \parfillskip-\@pnumwidth
    #5\leavevmode\hskip-\@tempdima
      \ifcase #1
       \or\or \hskip 1em \or \hskip 2em \else \hskip 3em \fi%
      #6\nobreak\relax
    \hfill\hbox to\@pnumwidth{\@tocpagenum{#7}}\par
    \nobreak
    \endgroup
  \fi}
\makeatother

\newtheorem{lemma}{Lemma}[section]
\newtheorem{theorem}[lemma]{Theorem}
\newtheorem{proposition}[lemma]{Proposition}

\newtheorem{corollary}[lemma]{Corollary}

\newtheorem{conjecture}[lemma]{Conjecture}

\newtheorem{claim*}{Claim}

\newtheorem{definition}[lemma]{Definition}

\newtheorem{notation}[lemma]{Notation}

\theoremstyle{definition}
\newtheorem{remark}[lemma]{Remark}

\newcommand{\PP}{{\mathbb P}}
\newcommand{\C}{{\mathbb C}}
\newcommand{\F}{{\mathbb F}}
\newcommand{\Q}{{\mathbb Q}}
\newcommand{\R}{{\mathbb R}}
\newcommand{\Z}{{\mathbb Z}}

\newcommand{\EE}{{\mathbb E}}

\newcommand{\calC}{{\mathcal C}}

\newcommand{\calF}{{\mathcal F}}
\newcommand{\calG}{{\mathcal G}}

\newcommand{\calK}{{\mathcal K}}

\newcommand{\calM}{{\mathcal M}}

\newcommand{\calO}{{\mathcal O}}
\newcommand{\calP}{{\mathcal P}}
\newcommand{\calQ}{{\mathcal Q}}

\newcommand{\calT}{{\mathcal T}}

\newcommand{\fraki}{{\mathfrak i}}

\newcommand{\frakp}{{\mathfrak p}}

\newcommand{\frakz}{{\mathfrak z}}

\newcommand{\frakP}{{\mathfrak P}}

\usepackage[mathscr]{euscript}

\DeclareMathOperator{\tr}{tr}
\DeclareMathOperator{\Tr}{Tr}

\DeclareMathOperator{\Frob}{Frob}
\DeclareMathOperator{\coker}{coker}

\DeclareMathOperator{\Char}{char}

\DeclareMathOperator{\im}{im}

\DeclareMathOperator{\Hom}{Hom}
\DeclareMathOperator{\Ext}{Ext}

\DeclareMathOperator{\Aut}{Aut}
\DeclareMathOperator{\Gal}{Gal}

\DeclareMathOperator{\ord}{ord}

\DeclareMathOperator{\Pic}{Pic}
\DeclareMathOperator{\Jac}{Jac}

\DeclareMathOperator{\Spec}{Spec}

\DeclareMathOperator{\et}{\acute{e}t}

\DeclareMathOperator{\res}{res}

\DeclareMathOperator{\rank}{rank}

\DeclareMathOperator{\Disc}{Disc}

\DeclareMathOperator{\B}{{\mbox{\textcyr{B}}}}

\DeclareMathOperator{\Sur}{Sur}
\DeclareMathOperator{\coinf}{coinf}

\DeclareMathOperator{\HHur}{\bold{\mathsf{Hur}}}

\DeclareMathOperator{\rDisc}{rDisc}
\DeclareMathOperator{\cl}{cl}
\DeclareMathOperator{\Prob}{Prob}
\DeclareMathOperator{\gr}{gr}
\DeclareMathOperator{\pos}{pos}
\DeclareMathOperator{\Out}{Out}

\newcommand{\hH}{\widehat{H}}
\newcommand{\Ks}{K^{\#}_{\O}}
\newcommand{\Gs}{G^{\#}_{\O}}
\newcommand{\Kp}{K^{\prime}_{\O}}
\newcommand{\Gp}{G^{\prime}_{\O}}
\newcommand{\tG}{\widetilde{G}}
\newcommand{\hZ}{\hat{\mathbb{Z}}}

\numberwithin{equation}{section}
\numberwithin{table}{section}

\usepackage{tikz-cd}

\title{Non-abelian Cohen--Lenstra heuristics in the presence of roots of unity}
\author{Yuan Liu}
\address{Department of Mathematics\\
University of Illinois Urbana-Champaign \\ 1409 W Green St \\ Urbana, IL 61801
 USA}  
\email{yyyliu@illinois.edu}

\begin{document}

	 \begin{abstract}
		For a Galois extension $K/\F_q(t)$ of Galois group $\Gamma$ with $\gcd(q,|\Gamma|)=1$, we define an invariant $\omega_K$, and show that it determines the Weil pairing of the curve corresponding to $K$ and it descends to the prime-to-$|\Gamma|$-torsion part of the lifting invariants of Hurwitz schemes introduced by Ellenberg--Venkatesh--Westerland and Wood. By keeping track of the image of $\omega_K$, we compute, as $K$ varies and $q\to \infty$, the average number of surjections from the Galois group of maximal unramified extension of $K$ to $H$, for any $\Gamma$-group $H$ whose order is prime to $q|\Gamma|$. Motivated by this result, we modify the conjecture of Wood, Zureick-Brown and the author about non-abelian Cohen--Lenstra, for both function fields and number fields, to cover the cases when the base field contains extra roots of unity. We also discuss how to use the invariant $\omega_K$ to construct a random group model, and prove in a special case that the model produces the same moments as our function field result.
	 \end{abstract}

	\maketitle
	
	\hypersetup{linkcolor=black}
	\tableofcontents

\hypersetup{linkcolor=blue}

\section{Introduction}

	Let $\Gamma$ be a finite group. The Cohen--Lenstra \cite{Cohen-Lenstra} and Cohen--Martinet \cite{Cohen-Martinet} heuristics give predictions for the distribution of the class groups of $K$ as $K$ varies over all $\Gamma$-extensions of $\Q$ with prescribed signature. Although only few cases of these heuristics are proven, huge breakthroughs have been made in the function field analogue, which was first formulated by Friedman and Washington \cite{Friedman-Washington}. Ellenberg, Venkatesh, and Westerland \cite{EVW} translated the heuristics into questions about counting the $\F_q$-points of the Hurwitz schemes, the moduli schemes of branched covers of $\PP^1$, and proved a version (as $q\to \infty$) of the function field analogue of the Cohen--Lenstra heuristics.
	
	The class group is the Galois group of the maximal abelian unramified extension. So it is natural to study the distributions of the non-abelian analogues of the class group. In \cite{BBH-imaginary} and \cite{BBH-real}, Boston, Bush, and Hajir extended the Cohen--Lenstra heuristics to a pro-$\ell$ setting (for an odd prime $\ell$) by considering the distribution of the Galois groups of the maximal unramified pro-$\ell$ extensions of quadratic fields. Boston and Wood \cite{Boston-Wood} proved the function field analogue of the Boston--Bush--Hajir heuristics. Wood, Zureick-Brown and the author \cite{LWZB} considered a more generalized setting. Let $Q$ be a global field that is either $\Q$ or $\F_q(t)$. As $K$ varies over all $\Gamma$-extensions of $Q$ that is split completely at $\infty$, \cite{LWZB} gives conjectures of the distribution of the Galois group of the maximal unramified extension of $K$ whose degree is prime to both $|\Gamma|$ and the number of roots of unity in $Q$, and not divisible by $\Char(Q)$; and also proves the function field analogue as $q\to \infty$.
	
 	All of the conjectures above fail when the base field contains extra roots of unity. For example, the Cohen--Lenstra--Martinet heuristics predict that the probability that a quadratic extension of $\Q(\sqrt{-3})$ (which contains the $3$rd roots of unity) has trivial $3$-part of class group is $\prod_{i=2}^\infty (1-3^{-i}) \approx .840$, however, Malle's numerical result \cite{Malle2008} suggests that this probability should be $\frac{4}{3} \prod_{i=1}^{\infty}(1+3^{-i})^{-1} \approx .852$. The reason for this phenomenon is better understood in the function field case: Friedman and Washington in \cite{Friedman-Washington} already discussed the importance of the Weil pairing of the curve corresponding to the function field extension $K/\F_q(t)$. When the $\ell$-th roots of unity $\mu_{\ell}$ is contained in $\F_q(t)$ (that is, when $\ell \mid q-1$), the Weil pairing reduces to some non-trivial structure on the $\ell$-part of the class group of $K$, which makes the distribution of class groups different from the the case that $\mu_\ell \not\subset \F_q(t)$. By keeping track of the image of the Weil pairing, Lipnowski and Tsimerman \cite{Lipnowski-Tsimerman} refined the Cohen--Lenstra heuristics for the distribution of the $\ell$-part of class groups of quadratic extensions over $\F_q(t)$ in the case when $\ell \mid q-1$, and used the results of Ellenberg--Venkatesh--Westerland to prove the refined conjecture as $q \to \infty$; and Lipnowski, Sawin, and Tsimerman \cite{Lipnowski-Sawin-Tsimerman} gave a similar conjecture in the number field case, by making local adjustments at the primes lying over $\ell$ and $\infty$. 	
	
	In this paper, we generalize the non-abelian analogues of Cohen--Lenstra heurisitics to include the roots-of-unity case. Let $\Gamma$ be a finite group, and $Q$ be a number field or $\F_q(t)$. 
	As $K$ varies over $\Gamma$-extensions of $Q$ that is split completely at each archimedean prime when $Q$ is a number field or at $\infty$ when $Q=\F_q(t)$, we study the distribution of the Galois group of the maximal unramified and split-completely-at-$\infty$ extension of $K$ whose degree is prime to $|\Gamma|$ and not divisible by $\Char(Q)$.

\subsection{Main results}
	
	The bulk of this paper deals with the function field case, that is $Q=\F_q(t)$. A $\Gamma$-extension $K/\F_q(t)$ that is split completely at $\infty$ corresponds to a $\Gamma$-cover $X/\PP^1_{\F_q}$.
	By passing this extension to $\overline{\F}_q$, we first define an invariant $\omega_K$ associated to $K$ that, roughly speaking, plays the role of the Weil pairing associated to $X_{\overline{\F}_q}$ in the non-abelian setting. For a field $L$, we let $G_{\O}(L):= L_{\O}/L$ denote the Galois group of the maximal unramified extension $L_{\O}$ of $L$. Consider the exact sequence 
	\[
		1 \longrightarrow G_{\O}(K \overline{\F}_q) \longrightarrow G_{\O}(K) \longrightarrow \Gal(K\overline{\F}_q/ K)= \hZ \cdot \Frob \longrightarrow 1, 
	\] 
	and a splitting $\hZ \cdot \Frob \to G_{\O}(K)$ defined by identifying $\Frob$ with the Frobenius element in the decomposition group of the extension $K_{\O}/K$ at a preselected place of $K_{\O}$ over $\infty$. Then this splitting defines a $\Frob$ action on $G_{\O}(K\overline{\F}_q)$ by conjugation. We let $\Gp(K)$ denote the pro-prime-to-$q|\Gamma|$ completion of $G_{\O}(K\overline{\F}_q)$ and denote the corresponding extension by $\Kp/K\overline{\F}_q(t)$. (Equivalently, $\Gp(K)$ is the pro-prime-to-$q|\Gamma|$ completion of the \'etale fundamental group $\pi_1^{\et}(X_{\overline{\F}_q})$.) Then $\Gp(K)$ naturally obtains an action of $\Frob$ from $G_{\O}(K\overline{\F}_q)$; and similarly, we can define a $\Frob$ action on the Galois group $\Gal(\Kp/\overline{\F}_q(t))$.
	Moreover, $\Gp(K)$ also has an action of $\Gamma$ by conjugation when we fix a homomorphic splitting of $\Gal(\Kp/ \overline{\F}_q(t)) \twoheadrightarrow \Gal(K\overline{\F}_q(t)/\overline{\F}_q(t)) \simeq \Gamma$. 
	We define the invariant $\omega_K$ to be a specific $\Frob$-equivariant isomorphism (see Definition~\ref{def:inv-K})
	\[
		\omega_K: \hZ(1)_{(q|\Gamma|)'} \overset{\sim}{\longrightarrow} H_2(\Gal(\Kp/\overline{\F}_q(t)), \Z)_{(q|\Gamma|)'}
	\]
	where the subscript $(q|\Gamma|)'$ represents the pro-prime-to-$q|\Gamma|$ completion. The invariant $\omega_K$ is first defined group-theoretically by studying the Schur covering group of $\Gal(\Kp/ \overline{\F}_q(t))$. Then we also give a cohomological definition: $\omega_K$ defines a central extension of $\Gp(K)$ by $\hZ(1)_{(q|\Gamma|)'}$, and hence corresponds to a class in $H^2(\Gp(K), \hZ(1)_{(q|\Gamma|)'})$; under the functorial isomorphism between Galois cohomology and \'etale cohomology, the corresponding class in $H^2_{\et}(X_{\overline{\F}_q}, \hZ(1)_{(q|\Gamma|)'})$ has trace $-|\Gamma|$ (see Lemma~\ref{lem:tr-Gamma} and Corollary~\ref{cor:ext-tr}).
	
	Then we pull everything back to $\F_q$ by taking the maximal $\Gamma$-equivariant quotient fixed by the $\Frob$ action. Such quotient of $\Gp(K)$ is naturally the Galois group of the maximal unramified extension of $K$ that is splitting completely over $\infty$ and has degree prime-to-$(q|\Gamma|)$, and we denote this quotient by $\Gs(K)$ and the corresponding extension by $\Ks/K$. For a subextension $L$ of $\Ks/K$ that is Galois over $\F_q(t)$, we define an invariant $\omega_{L/K}$ to be the composition
	\[\begin{tikzcd}
		\omega_{L/K}: \hZ(1)_{(q|\Gamma|)'} \arrow["\omega_K"]{r} & H_2(\Gal(\Kp/\overline{\F}_q(t)),\Z)_{(q|\Gamma|)'} \arrow["\coinf"]{r} & H_2(\Gal(L/\F_q(t)), \Z)_{(q|\Gamma|)'},
	\end{tikzcd}\]
	where $\coinf$ is the coinflation map defined by the quotient map between Galois groups. All the maps are $\Frob$-equivariant and $\Frob$ acts trivially on the last term, so $\omega_{L/K}$ factors through $\mu_{q-1}$, and $\im \omega_{L/K}$ is contained in the $(q-1)$-torsion subgroup of $H_2(\Gal(L/\F_q(t)), \Z)_{(q|\Gamma|)'}$.
		
	The invariant $\omega_K$ is ``abelianized'' to the Weil pairing associated to $X_{\overline{\F}_q}$. Let $\ell$ be a prime number such that $\ell \nmid q |\Gamma|$. The pro-$\ell$ completion $G_{\O}(K\overline{\F}_q)(\ell)$ of $G_{\O}(K\overline{\F}_q)$ is a pro-$\ell$ Demu\v{s}kin group (a special type of one-relator pro-$\ell$ group). Similarly to what we discussed previously, the invariant $\omega_K$ defines to a central extension of $G_{\O}(K\overline{\F}_q)(\ell)$ by $\Z_{\ell}(1)$, from which we obtain information of the relator in a pro-$\ell$ presentation of $G_{\O}(K\overline{\F}_q)(\ell)$. This information, by properties of Demu\v{s}kin groups (namely, the transgression map associated to the presentation), completely determines the cup product
	\[
		H^1(G_{\O}(K\overline{\F}_q), \Z_{\ell}) \times H^1(G_{\O}(K\overline{\F}_q), \Z_{\ell}) \overset{\cup}{\longrightarrow} H^2(G_{\O}(K\overline{\F}_q), \Z_{\ell}).
	\]
	Because $G_{\O}(K\overline{\F}_q)$ acts trivially on $\Z_{\ell}(1)$ and $H^1(G_{\O}(K\overline{\F}_q), \Z_{\ell}(1))$ can be identified as the $\ell$-adic Tate module of $X_{\overline{\F}_q}$, the cup product above followed by the trace map gives the Weil pairing associated to $X_{\overline{\F}_q}$. Similarly, the invariant $\omega_{L/K}$ determines the image of the Weil pairing (see Proposition~\ref{prop:omega-weil}, and also see Definition~\ref{def:image-Weil} for the definition of the image of Weil pairing).
	
	Then, we generalize the results of Lipnowski--Sawin--Tsimerman to the non-abelian setting by keeping track of our invariant $\omega_{L/K}$. In fact, $\omega_{K}$ is defined in the way such that $\omega_{L/K}$ agrees with the prime-to-$|\Gamma|$-torsion part of a component invariant of Hurwitz schemes, which is first introduced by Ellenberg--Venkatesh--Westerland and then carefully explained by Wood in \cite{Wood-lifting} (we refer this invariant by the \emph{EVW-W lifting invariant}). Then by modifying the moment computation using Hurwitz schemes in \cite{LWZB}, we obtain the moment result in the function field case as $q \to \infty$. Let $\rDisc K$ denote the norm of the radical of the ideal $\Disc(K/Q)$. Let $E_{\Gamma}(D,Q)$ be the set of the pairs $(K, \iota)$, where $K$ is a split-completely-at-$\infty$ extension of $Q$ with $\rDisc K = D$, and $\iota$ is an isomorphism $\Gal(K/Q) \to \Gamma$. An \emph{admissible $\Gamma$-group} is a group with a $\Gamma$ action such that it is of order prime to $|\Gamma|$ and is generated, under the action of $\Gamma$, by the elements in the form of $g^{-1}\gamma(g)$ for $g\in G,\, \gamma \in \Gamma$. 
	All $\Gamma$-equivariant quotients of $\Gs(K)$ are admissible $\Gamma$-groups.
		
		\begin{theorem}\label{thm:functionfield}
		Let $\Gamma$ be a finite group, $H$ a finite admissible $\Gamma$-group, and $\delta$ a  group homomorphism $\hat{\Z}(1)_{(|\Gamma|)'} \to H_2(H \rtimes \Gamma, \Z)_{(|\Gamma|)'}$. Then 
		\[
			\lim_{N\to\infty} \lim_{\substack{q \to \infty \\ (q, |\Gamma||H|)=1 \\ \ord(\im \delta) \mid q-1}} \dfrac{\sum\limits_{n\leq N} \sum\limits_{K \in E_{\Gamma}(q^n, \F_q(t))} \#\left\{ \pi\in \Sur_{\Gamma}\left(\Gs(K), H \right) \, \Bigl\rvert\, \pi_{*} \circ \omega_{K^{\#}/K}=\delta \right\}}{\sum\limits_{n \leq N} \# E_{\Gamma}(q^n, \F_q(t))} = \frac{1}{[H: H^{\Gamma}]},
		\]
		where in the limit $q$ is always a prime power, $\ord(\im \delta)$ is the order of the image of $\delta$, and $\pi_*$ is the coinflation map $H_2(\Gal(\Ks/\F_q(t)), \Z)_{(|\Gamma|)'} \to H_2(H \rtimes \Gamma, \Z)_{(|\Gamma|)'}$ induced by $\pi$. 
	\end{theorem}
	
	Inspired by this theorem, we make the following conjecture about the distribution of the prime-to-$\Char(Q)|\Gamma|$ completion of the maximal unramified Galois group of $K$, for a fixed base field $Q$, in both the function field case and the number field case. 
	
	\begin{conjecture}[Moment part of Conjecture~\ref{conj:ff}, and Conjecture~\ref{conj:nf}]\label{conj}
		Let $\Gamma$ be a finite group, $H$ a finite admissible $\Gamma$-group. Let $q$ be a prime power such that $\gcd(q, |\Gamma||H|)=1$, and $\delta$ a group homomorphism $\hZ(1)_{(q|\Gamma|)'} \to H_2(H\rtimes \Gamma, \Z)_{(q|\Gamma|)'}$. If $\ord(\im\delta) \mid q-1$, then
		\[
			\lim_{N\to\infty} \dfrac{\sum\limits_{n\leq N} \sum\limits_{K \in E_{\Gamma}(q^n, \F_q(t))} \#\left\{ \pi\in \Sur_{\Gamma}\left(\Gs(K), H \right) \, \Bigl\rvert\, \pi_{*} \circ \omega_{K^{\#}/K}=\delta \right\}}{\sum\limits_{n \leq N} \# E_{\Gamma}(q^n, \F_q(t))} = \frac{1}{[H: H^{\Gamma}]};
		\]
		otherwise, $\#\{ \pi\in \Sur_{\Gamma}(\Gs(K), H) \mid \pi_{*} \circ \omega_{K^{\#}/K}=\delta\}=0$ for any $K \in E_{\Gamma}(q^n, \F_q(t))$.
		
		Let $Q$ be a number field, and $m$ the maximal integer such that the $m$-th roots of unity is contained in $Q$.	
		For any extension $K/Q$, we define $\Gs(K):=G_{\O}(K)_{(|\Gamma||G_{\O}(Q)|)'}$.
		Assume that $G_{\O}(Q)$ is pro-prime-to $|H|$.
		Then
		\[
			\lim_{N \to \infty} \frac{\sum\limits_{D \leq N} \sum\limits_{K \in E_{\Gamma}(D,Q)} \# \Sur_{\Gamma}\left(G^{\#}_{\O}(K), H\right)}{ \sum\limits_{D \leq N} \# E_{\Gamma}(D,Q)}=\frac{\#H_2(H \rtimes\Gamma, \Z)_{(|\Gamma|)'}[m]}{[H:H^{\Gamma}] \cdot |H|^u},
		\]
		where $u$ is the rank of the group of units in the ring $\calO_Q$ of algebraic integers of $Q$.
	\end{conjecture}
	
	We add the requirement ``pro-prime-to-$|G_{\O}(Q)|$'' in the number field conjecture because, when $G_{\O}(Q)\neq 1$, the group $G_{\O}(K)_{(|\Gamma|)'}$ is not an admissible $\Gamma$-group and its distribution would be very different from the function field case.

	We also conjecture that there exists a unique probability measure that gives the conjectural moment in the function field case (see Conjecture~\ref{conj:prob-exist}). Then the conjectural moment in the number field case is obtained by taking the quotient of a group according to this probability measure modulo $u$ elements according to the Haar measure. This number field conjecture is inspired by \cite{Malle2010}, \cite{Lipnowski-Sawin-Tsimerman}, and two other evidences from algebraic number theory: 
	\begin{enumerate}
		\item Using the previous work \cite{Liu2020} of the author, we show that, in a presentation of the Galois group $G^{\#}_{\O}(K)$, it needs $u$ more relators in the number field case than in the function field case. (See \S\ref{ss:evidence1})
		
		\item When $Q=\Q(\mu_{\ell})$ and $K$ is nice (namely, when $K$ is $\ell$-CM), Schmidt \cite{Schmidt1996} proved that there is a maximal \emph{positively ramified} pro-$\ell$ extension $K_{\pos}(\ell)$ of $K$ such that the Galois group of $K_{\pos}(\ell)$ over the cyclotomic $\Z_\ell$-extension $K_{\infty}$ of $K$ is a very nice analogue of the pro-$\ell$ completion of the \'etale fundamental group of $X_{\overline{\F}_q}$. Using this positively ramified extension, we show that $G_{\O}(K)(\ell)$ is the quotient of $G_{\pos}^{\#}(K)$ modulo $u$ elements, where $G_{\pos}^{\#}(K)$ is the maximal quotient of $\Gal(K_{\pos}(\ell)/K_{\infty})$ fixed by $\Gal(K_{\infty}/K)$ (which is the analogue of $\Gs(K)(\ell)$ from the function field case). (See \S\ref{ss:evidence2})
	\end{enumerate}  
	
	Regarding the function field case, 
	another advantage of studying the invariant $\omega_K$ is that it shows us how to construct a random group model to predict the distribution of interest to us. 
	By the Grothendieck’s theory of specialization, $\Gp(K)$ has a pro-prime-to-$q|\Gamma|$ presentation in the form of $\langle x_1, \ldots, x_{2g} \mid [x_1, x_2]\cdots [x_{2g-1}, x_{2g}] \rangle$, where $g$ is the genus of $X_{\overline{\F}_q}$. Let $F$ be the free pro-prime-to-$q|\Gamma|$ group on generators $x_1, \ldots, x_{2g}$, and let $R$ be the closed normal subgroup of $F$ generated by $[x_1, x_2]\cdots [x_{2g-1}, x_{2g}]$. Then the short exact sequence 
	\begin{equation}\label{eq:intro-stem}
		1 \longrightarrow R/[F,R] \longrightarrow F/[F,R] \longrightarrow \Gp(K) \longrightarrow 1
	\end{equation}
	is a central extension with $R/[F,R]\simeq \hZ_{(q|\Gamma|)'}$; and moreover, the kernel is contained in the commutator of $F/[F,R]$, so the extension is stem. Note that $H_2(G,\Z)$ is the Schur multiplier of $G$, and it is isomorphic to the kernel of a Schur covering (that is, a maximal stem extension) of $G$. Thus, the isomorphism $\omega_K$ implies that there are an action of $\Frob$ and an action of $\Gamma$ on the sequence \eqref{eq:intro-stem} such that they preserve the kernel $R/[F,R]$, they induce the natural $\Frob$ and $\Gamma$ actions on $\Gp(K)$, and moreover, on $R/[F,R]$, $\Gamma$ acts trivially and $\Frob$ acts as multiplication by $q$. Then we expect that by considering the $\Frob$ action as a random automorphism of $F/[F,R]$ satisfying all the requirements, we can obtain a random group model that gives the conjectural moment in the function field case. 
	
	In general, understanding the group structure of such random group models is difficult.
	In this paper, we only study this random group model when we focus on the pro-$\ell$ completions for an odd prime $\ell$ and $\Gamma=\Z/2\Z$. By the classification of Demu\v{s}kin group by Labute \cite{Labute}, a pro-$\ell$ Demu\v{s}kin group whose abelianization is torsion-free  necessarily has a presentation as described in the preceding paragraph, and we denote such Demu\v{s}kin group by $\calG_g$. Then, up to isomorphism, $\calG_n$ admits a unique Schur covering
	\[
		1 \longrightarrow \Z_{\ell} \longrightarrow \widetilde{\calG}_n \overset{\pi_n}{\longrightarrow} \calG_n \longrightarrow 1,
	\]
	and we show that there exists a $\Gamma=\Z/2\Z$ action (unique up to isomorphism) on $\widetilde{\calG}_n$ such that the action preserves and acts trivially on $\ker \pi_n$ and its induced action on the abelianization of $\calG_n$ is taking inversion (the last condition is naturally satisfied by the $\Gamma$ action on $\Gp(K)(\ell)$, since the \'etale fundamental group of $\PP^1_{\overline{\F}_q}$ is trivial). Then we define a random group $Y_{q, n}$ to be the maximal $\Gamma$-equivariant quotient of $\calG_n$ fixed by a random automorphism of $\widetilde{\calG}_n$ that preserves and acts as multiplication by $q$ on $\ker \pi_n$. (See Definition~\ref{def:random-group} for a more careful definition of the random group.) Then we prove that the moment obtained by $Y_{q,n}$ is exactly the moment in Conjecture~\ref{conj}.
	
	\begin{theorem}[see Theorem~\ref{thm:moment}]\label{thm:intro-random-group}
		Let $H$ be a finite $\ell$-group with an action of $\Gamma=\Z/2\Z$ such that $H$ is admissible. Let $n$ be a sufficiently large integer such that there exists a $\Gamma$-equivariant surjection from $\calG_n$ to $H$. Let $\delta$ be a group homomorphism $\Z_{\ell} \to H_2(H \rtimes \Gamma, \Z)(\ell)$ such that $\ord(\im \delta)\mid q-1$. Then 
		\[
			\EE ( \# \{ \pi \in \Sur_{\Gamma} (Y_{q, n}, H) \mid \pi_{\dagger}= \delta \} )= \frac{1}{[H:H^{\Gamma}]},
		\]
		where $\pi_{\dagger}$ is the coinflation map $\Z_{\ell} \simeq H_2(\calG_n \rtimes \Gamma , \Z)(\ell) \to H_2(H \rtimes \Gamma, \Z)(\ell)$ induced by the $\Gamma$-equivariant quotient map $\calG_n \to Y_{q,n} \overset{\pi}{\longrightarrow} H$.
	\end{theorem}
	We conjecture that, as $n \to \infty$, the probability measures defined by $Y_{q,n}$ weakly converge to a probability measure that gives the moments in Conjecture~\ref{conj} (see Conjecture~\ref{conj:limit-measure}). Theorem~\ref{thm:intro-random-group} is proved by studing the orbits of $\Sur_{\Gamma}(\calG_n, H)$ under the action of the automorphisms of $\widetilde{\calG}_n$ that preserves and acts as multiplication by $q$ on $\ker \pi_n$. By using the method of successive approximation in \cite{Labute}, we prove that two surjections in $\Sur_{\Gamma}(\calG_n, H)$ are in the same orbit if and only if they induce the same coinflation map. This random group model can also be used in other scenarios with similar group structures. For example, in Appendix~\ref{sect:Appendix}, we apply the model to study the fundamental groups of random 3-manifolds constructed by random Heegaard splittings, in the notion defined by Dunfield and Thurston  \cite{Dunfield-Thurston}. We show that when $\phi$ is a random automorphism of $\widetilde{\calG}_n$ (forgetting the $\Gamma$ action) that preserves and acts trivially on $\ker \pi_n$, the moment obtained by the random group
	\[
		\lim_{n \to \infty} \faktor{\widetilde{\calG}_n}{[x_1, x_3, \ldots, x_{2n-1}, \phi(x_1), \phi(x_3), \ldots, \phi(x_{2n-1})]}
	\]
	agrees with the average size of $\Sur(\pi_1(M), H)$, for any $\ell$-group $H$, as $M$ varies as a random 3-manifold (see Proposition~\ref{prop:3-manifold}).

\subsection{Prior works and further questions}
	The study of random group models that are relating to the Cohen--Lenstra type of questions started at the very beginning: in the original paper \cite{Cohen-Lenstra} of Cohen and Lenstra, they used the random matrix
	\begin{equation*}
		\lim_{n \to \infty} \coker (M), \quad \text{ for a random $n\times (n+1)$ matrix $M$ with $\Z_\ell$-entries},
	\end{equation*}
	to give the probability measure in the totally real case of their conjecture. Translating it into a random abelian group, this model is equivalent to 
	\[
		 \lim_{n\to \infty} \faktor{\Z_\ell^{\oplus n}}{\text{$n+1$ random elements}}.
	\]
	In all the generalizations of Cohen--Lenstra in the non-roots-of-unity case, the corresponding random group models are constructed using balanced presentations (that is, presentations such that the difference between the numbers of generators and relators is fixed). Boston--Bush--Hajir studied the balanced pro-$\ell$ presentations in \cite{BBH-imaginary, BBH-real}. Wood, Zureick-Brown and the author in \cite{LWZB} gave the conjectural probability measures by using a form of balanced presentations of admissible $\Gamma$-groups, which is a modification of the probability measures of random balanced presentations studied in \cite{Liu-Wood}. In the non-roots-of-unity case, the random balanced presentation models work very well, and the author proved in \cite{Liu2020} that the unramified Galois groups of interest to us can always be obtained by a balanced $\Gamma$-presentation as conjectured in \cite{LWZB}.
	
	However, when considering the roots-of-unity case, the balanced presentation models fail. 
	Similar phenomenon happens to random 3-manifolds: Dunfield and Thurston in \cite{Dunfield-Thurston} tried to use balanced presentations to study random 3-manifolds, but they concluded that the fundamental groups of random 3-manifolds have many more finite quotients than the random balanced presentation model. It was first suggested by Friedman and Washington \cite{Friedman-Washington} that, instead of studying random presentation, we could consider the Frobenius map as a random automorphism acting on Jacobian varieties, and hence study the random matrix 
	\begin{equation}\label{eq:random-matrix-2}
		\lim_{n \to \infty} \coker(I_n-\calF)
	\end{equation}
	where $\calF$ is a random invertible matrix that is compatible with the Weil pairing on Jacobian. The random matrices in this setting and their relationship with function field case have been vastly studied, for example, see \cite{Katz1999, Achter2006, Achter2008, Garton2015, Adam2015, Lipnowski-Tsimerman, Lipnowski-Sawin-Tsimerman}.
	
	The random group model constructed in this paper is a non-abelian generalization of \eqref{eq:random-matrix-2} and sets up a framework for the study of the non-abelian Cohen--Lenstra heuristics in the roots of unity case. There are many important questions in this direct that are left to be solved. We don’t know how to compute the probability measure defined by $Y_{q,n}$, and whether the limit of the measure as $n \to \infty$ has total mass 1. Although, Theorem~\ref{thm:intro-random-group} strongly suggests that our random group model is the right object to study, what we really want is that the moment with respect to the limit measure equals to $1/[H:H^{\Gamma}]$ (that is, to take $\lim_{n\to \infty}$ inside $\EE$). The recent work of Sawin and Wood \cite{Sawin-Wood2022-1, Sawin-Wood2022-2} studies, in a very general setting, how one can get the probability measure directly from the moments, without the help of any random group structures; and we expect that their idea can be applied to our moment results and to prove some of the conjectures in Section~\ref{sect:conjectures}.
	
	Also, the moment result for $Y_{q,n}$ in Theorem~\ref{thm:intro-random-group} cannot be easily generalized to the situation either when $\Gamma\not\simeq \Z/2\Z$, or when studying the distribution of $\Gs(K)$, not just of $\Gs(K)(\ell)$. The main reasons of this difficulty are: 1) when $\Gamma\not\simeq \Z/2\Z$, we don't know how to give a canonical definition of a pro-$\ell$ Demu\v{s}kin $\Gamma$-group as in Lemma~\ref{lem:C2-Demu}; and 2) the proof of Theorem~\ref{thm:intro-random-group} vastly uses the filtration of pro-$\ell$ groups given by the lower central series, which does not exists for a general group.
	
	There is a potential to apply our random group model to obtain some numerical results. For example, for fixed $n$ and $q$, we can take the pro-$\calC$ completion of the random group $Y_{q,n}$ for a finite set of $\ell$-groups $\calC$ (for the definition of pro-$\calC$ completion, see \S\ref{ss:statements}). Since $\widetilde{\calG}_n^{\calC}$ is always finite, it is possible to list all the automorphisms of $\widetilde{\calG}_n^{\calC}$ satisfying the conditions in the definition of $Y_{q,n}$, and then we could see which quotients of $\calG_n^{\calC}$ can occur as $Y_{q,n}^{\calC}$, and compute the probability that they occur. As $n \to \infty$, if the probability measures converge fast, then we may get an approximation of the probability that $\Gs(K)(\ell)^{\calC}\simeq H$ for small $\ell$-groups $H$.
	
	For Cohen--Lenstra in the presence of roots of unity, Lipnowski--Sawin--Tsimerman in \cite{Lipnowski-Tsimerman, Lipnowski-Sawin-Tsimerman} made moment conjectures for the distributions of class groups of quadratic fields. In the non-abelian setting, Wood in \cite{Wood-nonab} gave a conjecture for the average number of unramified $G$-extensions of a quadratic extension of $\Q$ or $\F_q(t)$ for any finite group $G$. The conjectures made in this paper agree with all the previous conjectures about moments. Wood's work also considers some cases that cannot be addressed by the method in this paper -- it also deals with the case that $G$ is of even order and the case that the quadratic fields are imaginary. 
	Malle in \cite{Malle2010} and Boston--Bush in \cite{Boston-Bush} gave conjectures and numerical evidences in some cases for the probability measures of the distributions of class groups and class tower groups respectively.
	
	It is also interesting to understand the difference between the function field case and the number field case, which is related to the very important question: for number fields, what is the best analogue of the \'etale fundamental group from the function field case?
	Conjecture~\ref{conj} and the two evidences \S\ref{ss:evidence1} and \S\ref{ss:evidence2} strongly suggest that, for any $\Gamma$-extension $K/Q$ such that no archimedean prime is ramified,  there is some group associated to $K$ such that its group structure is similar to the prime-to-characteristic part of the \'etale fundamental groups of curves and $G_{\O}(K)_{(|\Gamma|)'}$ is the $\Gamma$-equivariant quotient of this group modulo $\rank \calO_{Q}^{\times}$ elements. When $K$ and $Q$ are nice, this group can be defined by the positively ramified extension by \cite{Schmidt1996} and the discussion in \S\ref{ss:evidence2}. We would like to understand how to generalize the notion of positively ramifications for arbitrary $K/Q$. When we have a better understanding of this mysterious group, we want to construct an invariant associated to $K$ that is analogous to $\omega_K$ in the function field case, and then we want to know whether this invariant agrees with the invariants defined in previous work (the lifting invariant for number fields in \cite{Wood-nonab} and the Artin--Verdier pairing invariant in \cite{Lipnowski-Sawin-Tsimerman}).

\subsection{Outline of the paper}
	In Section~\ref{sect:invariant}, we define the invariants $\omega_K$ and $\omega_{L/K}$: we give the group theoretical definition of $\omega_K$ in \S\ref{ss:group-def-inv}, the cohomological definition of $\omega_K$ in \S\ref{ss:coh-def-inv}, and the definition and a group theoretical description of $\omega_{L/K}$ in \S\ref{ss:inv-L/K}. In Section~\ref{sect:Weil-pairing}, we explain how the invariant $\omega_K$ is related to the Weil pairing in \S\ref{ss:cup-product} and how $\omega_{L/K}$ determines the image of the Weil pairing in \S\ref{ss:weil-pairing}. In Section~\ref{sect:ff-moment}, we first recall the definition of the EVW-W lifting invariants in \S\ref{ss:EVWW}, we show that $\omega_{L/K}$ is the prime-to-$|\Gamma|$-torsion part of the EVW-W lifting invariant in \S\ref{ss:EVW-omega}, and we give the proof Theorem~\ref{thm:functionfield} in \S\ref{ss:proof-main}. In Section~\ref{sect:RandomGroup}, we construct our random group model in \S\ref{ss:construction}, and give the proof of Theorem~\ref{thm:intro-random-group} in \S\ref{ss:moment-random-group} and \S\ref{ss:Labute}. In Section~\ref{sect:conjectures}, we give the statements of conjectures regarding the probability measures and the Cohen--Lenstra moments in \S\ref{ss:statements}, and provide the algebraic number theoretical evidences in \S\ref{ss:evidence1} and \S\ref{ss:evidence2} to support Conjecture~\ref{conj} .

\subsection{Notation}

	For a group $G$ and an element $g\in G$, we let $\ord(G)$ and $\ord(g)$ denote the orders of $G$ and $g$ respectively. For $a,b \in G$, we write $[a,b]=a^{-1}b^{-1}ab$ and $a^b=b^{-1}ab$. Whenever we talk about homomorphism of profinite groups, we always mean continuous homomorphisms. A $\Gamma$-group $H$ is a profinite group $H$ with a continuous action by $\Gamma$, and we write $H \rtimes \Gamma$ for the semidirect product defined by the $\Gamma$ action on $H$. We write $\Hom_G$, $\Sur_G$, and $\Aut_G$ to represent the sets of $G$-equivariant homomorphisms, surjections, and automorphisms respectively. A subgroup of a $\Gamma$-group that is closed under $\Gamma$ action is called a $\Gamma$-subgroup, and a $\Gamma$-quotient is the quotient by a normal $\Gamma$-subgroup. We write $G^{\Gamma}$ for the maximal $\Gamma$-subgroup of $G$ on which $\Gamma$ acts trivially. For a profinite group $G$, we write $G({\ell})$ and $G_{(n)'}$ for the pro-$\ell$ completion and the pro-prime-to-$n$ completion of $G$, and for the latter, we sometimes also call it ``pro-$(n)'$ completion''. A $|\Gamma|'$-$\Gamma$-group is a profinite $\Gamma$-group whose finite quotients all have order prime to $|\Gamma|$.
	
	A stem extension of $G$ is a central extension of $\pi: \widetilde{G}\twoheadrightarrow G$ such that $\ker \pi \subset [\widetilde{G}, \widetilde{G}]$. A \emph{Schur covering} of the group $G$ is a stem extension of $G$ of maximal possible order; or equivalently, a central extension $\pi: \widetilde{G} \to G$ such that under the universal coefficients theorem map $H^2(G,\ker \pi) \to \Hom(H_2(G,\Z), \ker \pi)$ the image of the class of $\pi$ is an isomorphism. 
	A Schur covering always exists, but is not necessarily unique. A $\ell$-Schur covering (resp. a $(n)'$-Schur covering) of $G$ is a stem extension such that the image of the class of $\pi$ maps to a quotient map $H_2(G, \Z) \to \ker \pi$ and $\ker \pi \simeq H_2(G, \Z)(\ell)$ (resp. $\ker \pi \simeq H_2(G, \Z)_{(n)'}$). Two extensions $\pi_i: G_i \twoheadrightarrow G$, $i=1,2$ are called isomorphic if there exists an isomorphism $\phi: G_1\to G_2$ such that $\pi_2 \circ \phi = \pi_1$.
	A splitting of a quotient map $\pi:\widetilde{G} \to G$ is a (not necessarily homomorphic) map $s: G \to \widetilde{G}$ such that $\pi \circ s$ is the identity map.
	For a global field $K$ and a set $S$ of primes of $K$, we write $G_S(K)$ for the Galois group of the maximal unramified-outside-$S$ extension of $K$, and write $K_S/K$ for the corresponding field extension.

\subsection*{Acknowledgement} I would like to thank Nigel Boston, Wei Ho and Melanie Matchett Wood for insightful conversations and encouragements, and Shizhang Li for helpful discussions regarding the proof of Lemma~\ref{lem:tr-Gamma}. I thank Nathan Dunfield, Shelly Harvey and Chris Leininger for discussions regarding 3-manifolds. I also thank Nigel Boston, Nathan Dunfield, Shizhang Li, Will Sawin, Melanie Matchett Wood and Foling Zou for comments on an early draft. I am partially supported by NSF Grant DMS-2200541.

\section{Definition of the invariants}\label{sect:invariant}

\subsection{Group theoretical definition of the invariant $\omega_K$}\label{ss:group-def-inv}

	Let $K$ be a global field of $\Char(K)=p>0$ and let $k \subset K$ be the finite constant field with algebraic closure $\overline{k}$. Assume that $K/k(t)$ is split completely at the place $\infty$. Then considering an exact sequence 
	\begin{equation}\label{eq:etale-es}
		1 \longrightarrow G_{\O}(K\overline{k}) \longrightarrow G_{\O}(K) \longrightarrow \Gal(\overline{k}(t)/k(t))=\hZ \cdot \Frob_k \longrightarrow 1,
	\end{equation}
	where $\Frob_k$ is the Frobenius element of $k$, there is a homomorphic splitting $\hZ \cdot \Frob_k \to G_{\O}(K)$ defined by identifying $\Frob_k$ with the Frobenious element in a decomposition group of $K_{\O}/K$ at a preselected place of $K$ over $\infty$. Then this splitting defines a $\Frob_k$ action on $G_{\O}(K\overline{k})$ by conjugation in $G_{\O}(K)$.

	Let $\Gamma$ be a finite group, and we further assume that $K$ is a Galois extension of $k(x)$ with Galois group isomorphic to $\Gamma$. Then $\Gal(K_{\O}/ \overline{k}(t))$ acts on its normal subgroup $G_{\O}(K\overline{k})$ by conjugation, and therefore $\Gal(K\overline{k}/ \overline{k}(t))\simeq \Gamma$ naturally gives an outer action on $G_{\O}(K\overline{k})$. 
	We let $G_{\O}(K\overline{k})_{(p|\Gamma|)'}$ denote the pro-prime-to-$p|\Gamma|$ completion of $G_{\O}(K\overline{k})$. Then the outer action of $\Gamma$ on $G_{\O}(K\overline{k})$ lifts to a $\Gamma$ action on $G_{\O}(K\overline{k})_{(p|\Gamma|)'}$. This lift is not unique, but by the Schur--Zassenhaus theorem, all lifts are conjugate, and thus we pick one such lift. Note that $G_{\O}(K\overline{k})_{(p|\Gamma|)'}$ is the quotient of $G_{\O}(K\overline{k})$ modulo a characteristic normal subgroup, so it has a $\Frob_k$ action induced from \eqref{eq:etale-es}.  
	
	Throughout this section, we use the following notation.
	
	\begin{notation}\label{not}
	\begin{itemize}
		\item Let $k$ be a finite field of $\Char(k)=p>0$, $\Gamma$ a non-trivial finite group of order not divisible by $p$, and $K/k(t)$ a $\Gamma$-extension that is split completely at $\infty$. 
		\item Let $\Gp(K)$ denote the pro-prime-to-$p|\Gamma|$ completion of $G_{\O}(K\overline{k})$, and $\Kp$ denote the corresponding field extension of $K\overline{k}$. We've seen that $\Gp(K)$ has the actions of $\Frob_k$ and $\Gamma$. 
		\item Let $\Gs(K)$ denote the maximal $\Gamma$-equivariant quotient of $\Gp(K)$ such that the induced $\Frob_k$ action on it is trivial. Then $\Gs(K)$ is naturally a quotient of $G_{\O}(K)$ by \eqref{eq:etale-es}, and we let $\Ks/K$ denote the corresponding extension. Equivalently, $\Ks$ is the maximal unramified extension of $K$, that is split completely at all places of $K$ over $\infty$ and of order prime to $|\Gamma|$ and $p$.
		\item Similarly, for a prime $\ell$ that is prime to $p|\Gamma|$, $G'_{\O}(K)(\ell)$ is the pro-$\ell$ completion of $G_{\O}(K\overline{k})$, and we let $\Kp(\ell)$ denote the corresponding field extension of $K\overline{k}$. The pro-$\ell$ completion $G^{\#}_{\O}(K)(\ell)$ of $G^{\#}_{\O}(K)$ is the maximal $\Gamma$-equivariant quotient of $G'_{\O}(K)(\ell)$ such that the induced $\Frob_k$ action on it is trivial, and we let $\Ks(\ell)/K$ denote the corresponding extension.
	\end{itemize}
	\end{notation}

	Then the quotient maps in Notation~\ref{not} define a $\Gamma$-equivariant surjection $G_{\O}(K \overline{k}) \twoheadrightarrow \Gal(K^{\#}_{\O}/K)$, and hence we obtain the following group surjection
	\begin{equation}\label{eq:rho_K}
		\rho_K: \Gal(K_{\O}/\overline{k}(t)) \twoheadrightarrow \Gal(K^{\#}_{\O}/k(t))\simeq \Gal(K^{\#}_{\O}/K)\rtimes \Gamma.
	\end{equation}

	\begin{lemma}\label{lem:Galois-H2}
		There is a $\Frob_k$-equivariant isomorphism
		\begin{equation}\label{eq:Galois-H2-ell}
			H^2(\Gal(\Kp(\ell)/\overline{k}(t)), \Z_\ell(1)) \simeq \Z_\ell
		\end{equation}
		for any prime $\ell \nmid p|\Gamma|$;
		and there is a $\Frob_k$-equivariant isomorphism 
		 \begin{equation}\label{eq:Galois-H2-whole}
		 	H^2(\Gal(\Kp/\overline{k}(t)), \hZ(1)_{(p|\Gamma|)'}) \simeq \hat{\Z}_{(p|\Gamma|)'}.
		\end{equation}	
		As a consequence, there are $\Frob_k$-equivariant isomorphisms
		\begin{equation}\label{eq:Galois-H_2}
			H_2(\Gal(\Kp(\ell)/\overline{k}(t)), \Z)(\ell)\simeq \Z_\ell(1) \quad \text{and}\quad H_2(\Gal(\Kp/\overline{k}(t)), \Z)_{(p|\Gamma|)'} \simeq \hat{\Z}(1)_{(p|\Gamma|)'}
		\end{equation}
	\end{lemma}
	
	\begin{proof}
		Denote the smooth projective curve over $k$ associated to $K$ by $X$ and let $X_{\overline{k}}$ denote the base change of $X$ to $\overline{k}$. The Galois group $G_{\O}(K\overline{k})$ is exactly the \'etale fundamental groups $\pi_1^{\et}(X_{\overline{k}})$. Because we assume that $\Gamma$ is non-trivial, so $X$ has positive genus, then by \cite[Remark, page 607]{NSW}, the group cohomology of $G_{\O}(K\overline{k})$ agrees with the \'etale cohomology of $X_{\overline{k}}$. Then by the Poincar\'e duality, for a prime $\ell \nmid p |\Gamma|$ and a positive integer $n$, we have a functorial isomorphism
		\begin{equation}\label{eq:H2-full}
			H^2(G_{\O}(K\overline{k}), \mu_{\ell^n}) \simeq \Z/\ell^n\Z,
		\end{equation}
		where both sides have trivial actions of $\Frob_k$ and $\Gamma$.
		
		 We consider the inflation map $\iota_1: H^2(\Gp(K), \mu_{\ell^n})\to H^2(G_{\O}(K\overline{k}), \mu_{\ell^n})$. Because $\Gp(K)$ is the pro-prime-to-$p|\Gamma|$ completion of $G_{\O}(K\overline{k})$, we have $H^1(\ker(G_{\O}(K\overline{k}) \to \Gp(K)), \mu_{\ell^n})=0$, and hence, it follows by the Hochschild--Serre exact sequence that $\iota_1$ is an injection. Similary, the inflation map $\iota_2: H^2(G'_{\O}(K)(\ell), \mu_{\ell^n}) \to H^2(\Gp(K), \mu_{\ell^n})$ is an injection. Note that $G'_{\O}(K)(\ell)$ is a pro-$\ell$ Demu\v{s}kin group with free abelianization \cite[Thm.10.1.2(ii)]{NSW}. Then, because $G'_{\O}(K)(\ell)$ acts trivially on $\mu_\ell$, it follows by the definition of Demu\v{s}kin group that there is an isomorphism
		\[
			H^2(G'_{\O}(K)(\ell), \mu_{\ell^n}) \simeq  \Z/\ell^n\Z.
		\]
		Therefore \eqref{eq:H2-full} implies that the injections $\iota_1$ and $\iota_2$ are isomorphisms. By taking inverse limit over all $n$ and the fact that $\hat{\Z}(1)_{(p|\Gamma|)'}$ is the $\prod_\ell \Z_\ell(1)$ for all primes $\ell\nmid p|\Gamma|$, we have
		\begin{equation*}
			H^2(G'_{\O}(K)(\ell), \Z_{\ell}(1)) \simeq \Z_\ell \quad \text{and} \quad H^2(\Gp(K), \hat{\Z}(1)_{(p|\Gamma|)'}) \simeq \hat{\Z}_{(p|\Gamma|)'},
		\end{equation*}
		and they both have trivial actions of $\Frob_k$ and $\Gamma$.
		Then by the Hochschild--Serre spectral sequence, we have 
		\[
			H^2(\Gal(\Kp/\overline{k}(t)), \hat{\Z}(1)_{(p|\Gamma|)'}) = H^2(\Gp(K), \hat{\Z}(1)_{(p|\Gamma|)'})^{\Gamma} \simeq \hat{\Z}_{(p|\Gamma|)'},
		\]
		\begin{equation*}
			\text{and similarly,} \quad H^2(\Gal(\Kp(\ell) / \overline{k}(t)), \Z_{\ell}(1)) \simeq \Z_{\ell},
		\end{equation*}
		which proves \eqref{eq:Galois-H2-ell} and \eqref{eq:Galois-H2-whole}.
	
	Let $G$ denote $\Gal(\Kp/\overline{k}(t))$. Since the class group of $\overline{k}(t)$ is trivial, $G^{ab}$ is generated by all the inertia subgroups of the maximal abelian subextension of $\Kp/\overline{k}(t)$. Since $\Kp/K\overline{k}$ is unramified, each inertia subgroup of $\Kp/\overline{k}(t)$ is isomorphic to its image in $\Gal(K\overline{k}/\overline{k}(t))\simeq \Gamma$. So the exponent of $G^{ab}$ divides $|\Gamma|$, and hence $\Ext_{\Z}^1(G^{ab}, \hat{\Z}(1)_{(p|\Gamma|)'})=0$. Because $G$ acts trivially on $\hZ(1)_{(p|\Gamma|)'}$,  it follows by the universal coefficient theorem and \eqref{eq:Galois-H2-whole} that, for any $n$ that is prime to $p|\Gamma|$, 
		\[
			\Hom(H_2(G,\Z), \Z/n\Z)\simeq H^2(G, \Z/n\Z)=\mu_n.
		\]
		 So $H_2(G, \Z)/n H_2(G,\Z)\simeq \mu_n$, and by taking the inverse limit over all such $n$, we have
		\[
			H_2(G, \Z)_{(p|\Gamma|)'} \simeq \hat{\Z}(1)_{(p|\Gamma|)'}.
		\]
		By applying the same argument, we can prove $H_2(\Gal(\Kp(\ell)/\overline{k}(t)), \Z)(\ell)\simeq \Z_\ell(1)$.
	\end{proof}

	\begin{corollary}\label{cor:Galois-H2-k}
		There are isomorphisms
		\[
			H^2(\Gal(\Kp/k(t)), \hZ(1)_{(p|\Gamma|)'}) \simeq \hZ_{(p|\Gamma|)'} \quad \text{and}\quad H^2(\Gal(\Kp(\ell)/k(t)),\Z_\ell(1))\simeq \Z_{\ell}.
		\]
	\end{corollary}
	
	\begin{proof}
		Consider the short exact sequence 
		\[
			1 \longrightarrow \Gal(\Kp/\overline{k}(t)) \longrightarrow \Gal(\Kp/k(t)) \longrightarrow \Gal(\overline{k}(t)/k(t)) \longrightarrow 1
		\]
		Since $\Gal(\overline{k}(t)/k(t))$ is a free group generated by $\Frob_k$, it has cohomological dimension 1. Then by the Hochschild--Serre spectral sequence (e.g. by \cite[Chap.II, \S4, Ex.4]{NSW}), there is an exact sequence 
		\begin{eqnarray}
			0 &\longrightarrow& H^1(\Gal(\overline{k}(t)/k(t)), H^1(\Gal(\Kp/\overline{k}(t)), \hZ(1)_{(p|\Gamma|)'})) \label{eq:hs}\\ \nonumber 
			&\longrightarrow& H^2(\Gal(\Kp/k(t)), \hZ(1)_{(p|\Gamma|)'}) \longrightarrow H^2(\Gal(\Kp/\overline{k}(t)), \hZ(1)_{(p|\Gamma|)'})^{\Gal(\overline{k}(t)/k(t))} \longrightarrow 0. 
		\end{eqnarray}
		Recall that $\Gal(\Kp/\overline{k}(t))\simeq \Gal(\Kp/K\overline{k})\rtimes \Gamma$, where $\Gal(\Kp/K\overline{k})$ is pro-prime-to-$p|\Gamma|$ completion of $G_{\O}(K\overline{k})$. 
		We showed in the proof of Lemma~\ref{lem:Galois-H2} that the abelianization of $\Gal(\Kp/\overline{k}(t))$ has exponent dividing $|\Gamma|$, so
		\[
			H^1(\Gal(\Kp/\overline{k}(t)), \hZ(1)_{(p|\Gamma|)'})=\Hom((\Gal(\Kp/\overline{k}(t)), \hZ(1)_{(p|\Gamma|)'}))=0.
		\]
		So, the first isomorphism in the corollary follows by \eqref{eq:hs} and Lemma~\ref{lem:Galois-H2},
		and the second isomorphism in the corollary follows by similar arguments.
	\end{proof}

	We recall the definition of admissible $\Gamma$-groups in \cite{LWZB}: for a finite group $\Gamma$, a profinite $\Gamma$-group is \emph{admissible} if it is of order prime to $|\Gamma|$ and is topologically generated, under the $\Gamma$ action, by the elements in the form of $g^{-1}\gamma(g)$ for $g \in G, \gamma \in \Gamma$. In particular, for an admissible $\Gamma$-group $H$, one can check that $(H \rtimes \Gamma)^{ab}\simeq \Gamma^{ab}$.
	
	\begin{lemma}\label{lem:extension-whole}
		\begin{enumerate}
			\item \label{item:e-w-1} The action of $\Gal(K/\overline{k}(t)) \simeq \Gamma$ on $\Gp(K)$ makes $\Gp(K)$ an admissible $\Gamma$-group.
			\item \label{item:e-w-2} The $(p|\Gamma|)'$-Schur covering of $\Gal(\Kp/\overline{k}(t))$ has kernel isomorphic to $\hZ_{(p|\Gamma|)'}$:
			\begin{equation}\label{eq:stem}
				1 \longrightarrow \hZ_{(p|\Gamma|)'} \longrightarrow \tG \overset{\pi}{\longrightarrow} \Gal(\Kp/\overline{k}(t)) \longrightarrow 1.
			\end{equation}
			Moreover, such $(p|\Gamma|)'$-Schur covering is unique up to isomorphism; that is, if there are two such coverings $\pi_i:\tG_i \to \Gal(\Kp/\overline{k}(t))$ for $i=1,2$, then there is an isomorphism $\phi: G_1 \to G_2$ such that $\pi_1 = \pi_2 \circ \phi$.
			\item \label{item:e-w-3} There exists a unique $\Frob_k$ action on the group $\tG$ in \eqref{eq:stem}, with which \eqref{eq:stem} becomes a short exact sequence of $\Frob_k$-groups
				\begin{equation}\label{eq:stem-whole}
					1 \longrightarrow \hat{\Z}(1)_{(p|\Gamma|)'} \longrightarrow \tG \overset{\pi}{\longrightarrow} \Gal(\Kp/\overline{k}(t)) \longrightarrow 1.
				\end{equation}
		\end{enumerate}
	\end{lemma}

	\begin{proof}
		We let $H$ denote $\Gp(K)$ and $G$ denote $\Gal(\Kp/\overline{k}(t))$, so $G\simeq H \rtimes \Gamma$.
		The subgroup of $H$ generated by elements $\{h^{-1}\gamma(h) \mid h \in H, \gamma \in \Gamma\}$ is the commutator subgroup $[H, \Gamma]$ of $G$, which is a normal subgroup contained in $H$. Then $H/[H,\Gamma]$ is a quotient of $H$ with the trivial action of $\Gamma$, and hence it is a quotient of $G$. 
		However, such quotient corresponds to an unramified extension of $\overline{k}(t)$, which must be trivial. 
		So $H=[H,\Gamma]$, which proves \eqref{item:e-w-1}.

		We denote $A:=\hZ_{(p|\Gamma|)'}$. Note that $H^2(G, A)\simeq \Ext(G, A)$; and to see this correspondence explicitly, we can construct a group extension for every cohomology class $c$. For each $c \in H^2(G, A)$, there is a normalized 2-cocycle $\alpha$ contained in $c$, i.e., a cocycle such that $\alpha(g, 1)=\alpha(1,g)=0$ for any $g\in G$. Then the equivalent class of group extensions corresponding to $c$ can be represented by a group $\widetilde{G}$ with underlying set $A\times G$ and the multiplicative rule 
		\[
			(a, x)(b,y)=(a+x(b)+\alpha(x,y), xy),
		\]
		and the extension
		\begin{equation}\label{eq:ext}
			1 \longrightarrow A \longrightarrow \widetilde{G} \overset{\pi}{\longrightarrow} G \longrightarrow 1,
		\end{equation}
		where the surjection $\pi$ is the projection to the second component. 
		Since $G$ acts trivially on $A$, \eqref{eq:ext} is a central extension, so it is a stem extension if and only if the central extension reduced from \eqref{eq:ext} obtained by modulo $mA$ from $\widetilde{G}$ is nonsplit for any $m\in \hZ_{(p|\Gamma|)'}$, which happens exactly when the image of $c$ in $H^2(G, A/mA)$ is nonzero. From the proof of Lemma~\ref{lem:Galois-H2}, we see that, forgetting the $\Frob_k$-action,
		\[
			H^2(G, A/mA)\simeq \Z/m\Z
		\]
		as groups for any $m \in \hZ_{(p|\Gamma|)'}$.
		Thus, we see that \eqref{eq:ext} is a stem extension if and only if the image of $c$ under the isomorphism \eqref{eq:Galois-H2-whole} is contained in $\hZ^{\times}_{(p|\Gamma|)'}$. We denote the set of these classes corresponding to stem extensions by $C$.
		
		Next, we show that the group extension \eqref{eq:ext} defined by all the cohomological classes in $C$ are isomorphic to each other. We pick a $c$ in $C$, let $\alpha$ be the normalized $2$-cocycle in $c$ and let $\tG$ be the group extension defined by $\alpha$. Then for any $d \in C$, we can write $d$ as $mc$ for some $m \in \hZ^{\times}_{(p|\Gamma|)'}$, and $m\alpha$ is the normalized 2-cocycle contained in $d$. We let $\tG_d$ represent the group extension defined by $m\alpha$. Then there is an isomorphism
		\begin{eqnarray*}
			\tG & \overset{\sim}{\longrightarrow} \tG_d \\
			\text{defined by}\quad (a,1) &\longmapsto& (ma, 1)  \quad \text{for all $a \in A$}\\
			(0, x) &\longmapsto& (0,x) \quad \text{for all $x\in G$,}
		\end{eqnarray*}
		which clearly gives an isomorphism between the extensions $\tG \twoheadrightarrow G$ and $\tG_d \twoheadrightarrow G$. Then by \eqref{eq:Galois-H_2}, these stem extensions are $(p|\Gamma|)'$-Schur coverings, so we finish the proof of the statement~\eqref{item:e-w-2}.

		Denote $A(1):=\hZ(1)_{(p|\Gamma|)'}$. We can apply the argument above about group extensions to $\Gal(\Kp/k(t))$ and see that the each equivalence classes of group extensions 
		\begin{equation}\label{eq:ext-k}
			1 \longrightarrow A(1) \longrightarrow \hat{G} \overset{\hat{\pi}}{\longrightarrow} \Gal(\Kp/k(t)) \longrightarrow 1
		\end{equation}
		is represented by a class in $H^2(\Gal(\Kp/k(t)),A(1))$. By taking restriction, we get a group extension
		\begin{equation}\label{eq:ext-bk}
			1 \longrightarrow A(1) \longrightarrow \hat{\pi}^{-1}(G) \longrightarrow G \longrightarrow 1,
		\end{equation}
		and it is a central extension since $G$ acts trivially on $A(1)$. So \eqref{eq:ext-bk} is a stem extension if and only if the extension obtained by quotienting $\hat{G}$ in \eqref{eq:ext-k} by $mA(1)$ is nonsplit for any $m$, which happens exactly when the image of the corresponding cohomology class $c$ in $H^2(\Gal(\Kp/k(t)), \mu_m)$ is nonzero for any $m \in \hZ_{(p|\Gamma|)'}$.
		From the proof of Lemma~\ref{lem:Galois-H2} and Corollary~\ref{cor:Galois-H2-k}, we see that
		\[
			H^2(\Gal(\Kp/k(t)), \mu_m) \simeq H^2(G, \mu_m)^{\Gal(\overline{k}(t)/k(t))} \simeq \Z/m\Z.
		\]
		So \eqref{eq:ext-bk} is a stem extension if and only if the image of $c$ under the first isomorphism in Corollary~\ref{cor:Galois-H2-k} is contained in $\hZ^{\times}_{(p|\Gamma|)'}$. Then by using the same argument in the proof of the statement~\eqref{item:e-w-2}, we see that all of the extensions \eqref{eq:ext-k} corresponding to such cohomological classes are isomorphic.
		
		We fix an extension \eqref{eq:ext-k} such that \eqref{eq:ext-bk} is a stem extension. Recall that we identify $\Frob_k$ as an element of $\Gal(\Kp/k(t))$ whose image in $\Gal(\overline{k}(t)/k(t))$ is a generator. Then we pick an element $f$ in $\hat{\pi}^{-1}(\Frob_k)$. Since $\Gal(\Kp/k(t))=G \rtimes\langle{\Frob_k}\rangle$ and $f$ generate a free subgroup of $\Gal(\Kp/k(t))$, we have that 
		\[
			\hat{G}=\hat{\pi}^{-1}(G) \rtimes \langle{f}\rangle,
		\]
		which gives an action of $f$ (defined by conjugation) on every term of \eqref{eq:ext-bk}.
		This $f$-action does not depend on the choice of $f$, because two different choices of $f$ differ by an element of $\hZ(1)_{(p|\Gamma|)'}$, which is contained in the center of $\hat{\pi}^{-1}(G)$. So this unique $f$-action on \eqref{eq:ext-bk} gives the exact sequence in the statement \eqref{item:e-w-3}.
	\end{proof}
	
	The following is a group theoretical property of stem extensions that will be useful in this paper.
	
	\begin{lemma}\label{lem:app-stem}
		Let $1 \to A \to \widetilde{G} \to G \to 1$ be a stem extension. If $H$ is a subgroup of $\widetilde{G}$ that maps surjectively to $G$, then $H=\widetilde{G}$. 
	\end{lemma}
	
	\begin{proof}
		Since $\widetilde{G}=AH$, all the conjugates of the subgroup $H$ are in the form of $a^{-1}Ha$ for $ a\in A$. However, $A$ is in the center of $\widetilde{G}$, so $a^{-1}Ha=H$, which implies that  $H$ is a normal subgroup of $\widetilde{G}$. If $H$ is a proper subgroup of $\widetilde{G}$, then $\widetilde{G}/H$ is a non-trivial abelian group, so $A$ is not contained in the commutator subgroup of $\widetilde{G}$, which contradicts to the assumption that the group extension is stem.
	\end{proof}

	The function field extension $K\overline{k}/\overline{k}(t)$ corresponds to a ramified Galois $\Gamma$-cover $C \to \PP^1_{\overline{k}}$, and we let $P_1, \cdots, P_m \in \PP^1_{\overline{k}}$ denote the branch loci of this covering.
	So there is a natural surjection
	\[
		\pi^{\et}_1 (\PP^1_{\overline{k}}- \{P_1, \ldots, P_m\}) \twoheadrightarrow  \Gal(K'_{\O}/\overline{k}(t)) \twoheadrightarrow \Gal(K\overline{k} / \overline{k}(t)).
	\]
	By the Grothendieck's theory of specialization and the description of the fundamental group of a Riemann surface,
	the pro-prime-to-$p$ completion of the \'etale fundamental group above, which we will denote by $\Omega_K$, is a free pro-prime-to-$p$ group on $m-1$ generators and has a presentation 
	\begin{equation}\label{eq:pres-etfund}
		\Omega_K:=\pi^{\et}_1 (\PP^1_{\overline{k}}- \{P_1, \ldots, P_m\})_{(p)'}=\langle g_1, \ldots, g_m \mid g_1g_2\cdots g_m =1\rangle
	\end{equation}
	where the image of $g_i$ in $\Omega_K$ is a generator of an inertia group of $P_i$, and is mapped to a generator of an inertia subgroup at $P_i$ of $\Gal(K'_{\O}/\overline{k}(t))$ (and resp. of $\Gal(K\overline{k}/\overline{k}(t))$). 
	We call the tuple $\underline{g}:=(g_1,\ldots, g_m)$ satisfing the description above for an arbitrary ordering of $P_1, \ldots, P_m$ by \emph{a system of inertia generators for $K$}. The systems of inertia generators are not unique. 
	The Frobenius $\Frob_k$ acts on each Galois group, and hence we obtain a $\Frob_k$-equivariant surjection when we choose compatible lifts of $\Frob_k$
	\begin{equation}\label{eq:surj-omega-K}
		\varpi_K: \Omega_K \longrightarrow \Gal(K'_{\O}/\overline{k}(t)).
	\end{equation}
		
	By \eqref{eq:pres-etfund}, $\Omega_K$ is a quotient of the free pro-prime-to-$p$ group generated by $g_1, \ldots, g_m$. So we obtain a surjection of groups 
	\begin{equation}\label{eq:surj-rho-K}
		\rho_K: \langle g_1, \ldots, g_m \rangle \longrightarrow \Gal(K'_{\O}/\overline{k}(t)),
	\end{equation}
	which factors through $\varpi_K$. The map $\rho_K$ is just a group surjection, because the free group $\langle g_1, \ldots, g_m\rangle$ does not have any natural $\Frob_k$-action. 
	
		Let $z=t-a_i$ for some $a_i \in \overline{k}$ corresponds to the branch locus $P_i$. Then an inertia subgroup $\Gal\left(k((z^{1/\infty}))/k((z))\right)$ at $P_i$ of $\pi^{\et}_1(\PP^1_{\overline{k}} - \{P_1, \ldots, P_m\})_{(p)'}$ is naturally identified with $\hat{\Z}(1)_{(p)'}$ via
		\begin{eqnarray*}
			\Gal\left(k((z^{1/\infty})) / k((z)) \right) &\longrightarrow& \hZ(1)_{(p)'} \\
			\sigma &\longmapsto& \varprojlim_{m} \frac{\sigma(\sqrt[m]{z})}{\sqrt[m]{z}}.
		\end{eqnarray*}
		Under these identifications, by \cite[\S5.1 \& 5.2]{Wood-lifting}, in the presentation~\eqref{eq:pres-etfund}, each inertia generator $g_i$ for $i=1, \ldots, m$ maps to a common generator of $\hat{\Z}(1)_{(p)'}$. For a system of inertia generators $\underline{g}=(g_1, \ldots, g_m)$, we let $I(\underline{g})$ denote the common generator of $\hat{\Z}(1)_{(p)'}$.

	For a central group extension $\pi: G \to H$ with an isomorphism $\ker \pi \simeq A$ for some abelian group $A$, the \emph{differential map $d_{\pi}$ of $\pi$} is the homomorphism 
	\[
		d_{\pi}: H_2(G,\Z) \to A
	\]
	that is the image of the class containing $\pi$ under the quotient map $H^2(G,A) \to \Hom(H_2(G,\Z), A)$ in the universal coefficient theorem.
		
	\begin{proposition}\label{prop:etale-lift}
		Let $\widetilde{G}$, $\pi$ be as described in the group extension \eqref{eq:stem} in Lemma~\ref{lem:extension-whole}, and let $\rho_K$ be the surjection \eqref{eq:surj-rho-K}. 
		\begin{enumerate}
		\item \label{item:etale-lift-1} Then there is a unique surjection 
		\[
			\widetilde{\rho}_K: \langle g_1, \ldots, g_m \rangle \longrightarrow \widetilde{G},
		\]
		such that  $\rho_K=\pi \circ \widetilde{\rho}_K$ and $\ord(\widetilde{\rho}_K(g_i))=\ord(\rho_K(g_i))$ for each $i=1, \ldots, m$. 
		
		\item \label{item:etale-lift-2} The kernel of $\pi$ is generated by $\widetilde{\rho}_K(g_1\cdots g_m)$. So we can define a $\Frob_k$-equivariant isomorphism
		\begin{eqnarray}
			\iota_{\pi}:  \hat{\Z}(1)_{(p|\Gamma|)'} &\overset{\sim}{\longrightarrow}& \ker \pi \label{eq:isom-lift-invariant}\\
			 I(\underline{g}) &\longmapsto& \widetilde{\rho}_K(g_1\cdots g_m). \nonumber
		\end{eqnarray}
		Moreover, this isomorphism is independent of the choice of the system of inertia generators.
		
		\item \label{item:etale-lift-3} The differential map $d_{\pi}: H_2(\Gal(\Kp/\overline{k}(t)),\Z)\to \ker \pi$ of $\pi$ factors through a $\Frob_k$-equivariant isomorphism $H_2(\Gal(\Kp/\overline{k}(t)),\Z)_{(p|\Gamma|)'} \to \ker \pi$, which by abuse of notation will be also denoted by $d_{\pi}$. Then the $\Frob_k$-equivariant isomorphism $\omega_K:=d_{\pi}^{-1} \circ \iota_{\pi}: \hat{\Z}(1)_{(q|\Gamma|)'} \to H_2(\Gal(\Kp/\overline{k}(t)), \Z)_{(p|\Gamma|)'}$ does not depend on the choice of the stem extension $\pi$.
		\end{enumerate}
	\end{proposition}
	
	\begin{remark}\label{rmk:etale-lift}
		The isomorphism $\iota_{\pi}$ identifies $\widetilde{G}$ as a group extension of $\Gal(\Kp/\overline{k}(t))$	with kernel $\hat{\Z}(1)_{(p|\Gamma|)'}$, and hence corresponds to an element $\epsilon$ in $H^2(\Gal(\Kp/\overline{k}(t)), \hat{\Z}(1)_{(p|\Gamma|)'})$. 
		Then the differential map $H_2(\Gal(\Kp/\overline{k}(t)), \Z)_{(p|\Gamma|)'} \to  \hat{\Z}(1)_{(p|\Gamma|)'}$ obtained from $\epsilon$ is the inverse of $\omega_K$.
	\end{remark}

	\begin{proof}
		We let $\overline{g}_i$ be the image of $g_i$ in $\Omega_K$.
		We denote $t_i:=\varpi_K(\overline{g}_i)$ for each $i$ ($\varpi_K$ is the map \eqref{eq:surj-omega-K}), which is a generator of an inertia subgroup at $P_i$ of $\Gal(K'_{\O}/\overline{k}(t))$. Then $\ord(t_i)$ equals the order of its image in $\Gal(K\overline{k}/\overline{k}(t))\simeq \Gamma$ because $K'_{\O}$ is unramified over $K\overline{k}$. By definition $\rho_K$ is a lift of $\varpi_K$, $\rho_K(g_i)=t_i$ and hence $\ord(\rho_K(g_i))=\ord(t_i)$.
		Consider the central group extension \eqref{eq:stem}. The kernel has order prime to $|\Gamma|$, so for each $i$, there exists a unique lift $\widetilde{t}_i \in \pi^{-1}(t_i)$ of $t_i$ in $\widetilde{G}$ such that $\ord(\widetilde{t}_i)=\ord(t_i)$. The elements $\widetilde{t}_i$'s generate a subgroup of $\widetilde{G}$ that maps surjectively to $\Gal(K'_{\O}/\overline{k}(t))$, and therefore by Lemma~\ref{lem:app-stem}, they generate $\Gal(K'_{\O}/\overline{k}(t))$ because \eqref{eq:stem} is a stem extension.
		So the homomorphism $\langle g_1, \ldots, g_m \rangle \to \widetilde{G}$ mapping $g_i$ to $\widetilde{t}_i$ is a surjection, and is the unique surjection satisfies the conditions for $\widetilde{\rho}_K$ given in \eqref{item:etale-lift-1}.

		Note that $\rho_K$ factors through $\widetilde{\rho}_K$, so $r:=\widetilde{\rho}_K(g_1\cdots g_m)$ in contained in the subgroup $\ker \pi \simeq \hat{\Z}_{(p|\Gamma|)'}$ of $\widetilde{G}$. Define $\widehat{G}$ to be the quotient of $\widetilde{G}$ modulo $r$. Then $\widehat{G}$ fits into the following commutative diagram, where each arrow represents a surjection. 
		\begin{equation}\label{eq:diag-lift}
		\begin{tikzcd}
			\langle g_1, \ldots, g_m \rangle \arrow["/g_1\cdots g_m"]{r} \arrow["\widetilde{\rho}_K"]{d} & \Omega_K \arrow["s"]{d} \arrow["\varpi_K"]{dr} & \\
			\widetilde{G} \arrow[bend right=20, "\pi"]{rr}\arrow["/r"]{r} & \widehat{G} \arrow{r} & \Gal(K'_{\O}/\overline{k}(t)) 
		\end{tikzcd}
		\end{equation}
		By the properties of $\widetilde{\rho}_K$, the image of $g_i$ in $\widehat{G}$ has order equal to $\ord(t_i)$. So the surjection $s$ defines a $\widehat{G}$-extension $L/\overline{k}(t)$ that contains $K'_{\O}$ as a subfield, and its inertia subgroup at $P_i$ maps isomorphically to the inertia subgroup at $P_i$ of $K'_{\O}/\overline{k}(t)$. Thus, $L$ is unramified over $K'_{\O}$ and hence is unramified over $K\overline{k}$, and moreover, $\Gal(L/K\overline{k})$ is pro-prime-to-$(p|\Gamma|)$. By definition of $K'_{\O}$, we see that $L$ must be $K'_{\O}$, so $r$ generates $\ker \pi$.
		
		Assume that $\underline{h}:=(h_1,\ldots, h_m)$ is another system of inertia generators. Then there exists $\alpha\in \hat{\Z}(1)^{\times}_{(p)'}$ such that $I(\underline{h})=I(\underline{g})^{\alpha}$, and there is a permutation $\sigma \in S_m$ such that $h_i$ is conjugate to $g_{\sigma(i)}^{\alpha}$ for all $i$. 
		Then $\underline{h}$ defines a surjection $\rho'_K: \langle h_1, \ldots, h_m\rangle \to \Gal(K'_{\O}/\overline{k}(t))$, where $h_i$ maps to the corresponding inertia generator, and $\rho'_K$ lifts to a surjection $\widetilde{\rho}'_K: \langle h_1, \ldots, h_m \rangle \to \widetilde{G}$ such that $\ord(\rho'_K(h_i))=\ord(\widetilde{\rho}'_K(h_i))$ for each $i$. To prove the last claim in the statement \eqref{item:etale-lift-2}, we need to show $\widetilde{\rho}'_K(h_1\cdots h_m)=\widetilde{\rho}_K(g_1\cdots g_m)^{\alpha}$.
		
		Recall that $\Omega_K$ is a free pro-prime-to-$p$ group on $m-1$ generators. By Lemma~\ref{lem:app-stem}, the surjection $\varpi_K$ in \eqref{eq:surj-omega-K} lifts to a surjection $\widetilde{\varpi}_K: \Omega_K \to \widetilde{G}$ such that $\varpi_K=\pi \circ \widetilde{\varpi}_K$. For each $i$, we let $\overline{h}_i$ denote the image of $h_i$ in $\Omega_K$ and then define
		\[
			z_i := \widetilde{\varpi}_K(\overline{h}_i)^{-1} \widetilde{\rho}'_K(h_i).
		\]
		Because both of $\widetilde{\varpi}_K(\overline{h}_i)$ and $\widetilde{\rho}'_K(h_i)$ have image $\rho'_K(h_i)$ in $\Gal(\Kp/\overline{k}(t))$, $z_i$ is contained in $\ker \pi$, so is in the center of $\widetilde{G}$. Thus, 
		\[
			\widetilde{\rho}'_K(h_1 \cdots h_m)=\prod_{i=1}^m \widetilde{\varpi}_K(\overline{h}_i) z_i =\widetilde{\varpi}_K(\overline{h}_1\cdots \overline{h}_m) z_1 \cdots z_m = z_1 \cdots z_m.
		\]
		Similarly, we define $w_i:=\widetilde{\varpi}_K(\overline{g}_i)^{-1} \widetilde{\rho}_K (g_i)$ for each $i$. Then $w_i$ is in the center of $\widetilde{G}$ and $\widetilde{\rho}_K(g_1\cdots g_m)=w_1\cdots w_m$. Now, we fix an index $i$ and suppose $j=\sigma(i)$. 
		There exists $x \in \Omega_K$ such that $\overline{h}_i=x^{-1} \overline{g}_j^\alpha x$.
		We pick the lift $y = \widetilde{\varpi}_K (x)$ of $\varpi_K(x)$. Then $y^{-1} \widetilde{\rho}_K(g_j)^{\alpha} y$ has order equal to $\ord(\widetilde{\rho}_K(g_j))=\ord(\rho_K(g_j))=\ord (\rho'_K(h_i))$ and is a lift of $\rho'_K(h_i)$, so $y^{-1} \widetilde{\rho}_K(g_j)^{\alpha} y = \widetilde{\rho}'_K(h_i)$. Then 
		\begin{eqnarray*}
			z_i &=& \widetilde{\varpi}_K(\overline{h}_i)^{-1} \widetilde{\rho}'_K(h_i) \\
			&=& y^{-1} \widetilde{\varpi}_K(\overline{g}_j)^{-\alpha} \widetilde{\rho}_K(g_j)^{\alpha} y \\
			&=& y^{-1} w_j \left( \widetilde{\rho}_K(g_j)^{-1} w_j \widetilde{\rho}_K(g_j)\right) \left( \widetilde{\rho}_K(g_j)^{-2} w_j \widetilde{\rho}_K(g_j)^2\right) \cdots \left( \widetilde{\rho}_K(g_j)^{1-\alpha} w_j\widetilde{\rho}_K(g_j)^{\alpha-1}\right) y \\
			&=& w_j^{\alpha}.
		\end{eqnarray*}
		So we finish the proof of \eqref{item:etale-lift-2}.

		Since $\pi$ defines a stem extension, its corresponding differential map is surjective, so $d_{\pi}$ defined in \eqref{item:etale-lift-3} is a $\Frob_k$-equivariant isomorphism because of \eqref{eq:Galois-H_2} in Lemma~\ref{lem:Galois-H2}.
		Suppose there are two stem extensions $1 \to A_i \to \widetilde{G}_i \overset{\pi_i}{\to} \Gal(\Kp/\overline{k}(t)) \to 1$ for $i=1,2$. By Lemma~\ref{lem:extension-whole} \eqref{item:e-w-2} and \eqref{item:e-w-3}, there exists an $\Frob_k$-equivariant isomorphism $\phi: \widetilde{G}_1\to \widetilde{G}_2$ such that $\pi_1=\pi_2 \circ \phi$, so $\phi|_{A_1}: A_1 \to A_2$ is an isomorphism. Then the differential map $d_{\pi_i}$ satisfies $d_{\pi_2}=\phi|_{A_1} \circ d_{\pi_1}$.
		For each $i$, let $\iota_{\pi_i}$ denote the map in \eqref{eq:isom-lift-invariant} for $\pi_i$. Then by definition of $\iota_{\pi_i}$, we obtain that $\iota_{\pi_2}=\phi|_{A_1} \circ \iota_{\pi_1}$. Thus $d^{-1}_{\pi_1}\circ \iota_{\pi_1}=d^{-1}_{\pi_2}\circ \iota_{\pi_2}$ and then we proved \eqref{item:etale-lift-3}.
	\end{proof}

	Since the map $\omega_K$ in Proposition~\ref{prop:etale-lift} does not depend on the choice of $\pi$ and the system of inertia generators, it is an invariant associated to $K$. 
	
	\begin{definition}[Invariant $\omega_K$ associated to $K$]\label{def:inv-K}
		We define the \emph{invariant $\omega_K$ associated to $K$} to be the $\Frob_k$-equivariant isomorphism $\omega_K: \hat{\Z}(1)_{(q|\Gamma|)'} \to H_2(\Gal(\Kp/\overline{k}(t)), \Z)_{(p|\Gamma|)'}$ in Proposition~\ref{prop:etale-lift}\eqref{item:etale-lift-3}.
	\end{definition}

\subsection{Cohomological definition of $\omega_K$}\label{ss:coh-def-inv}
	
	As discussed in Remark~\ref{rmk:etale-lift}, the inverse of $\omega_K$ is the differential map of a class $\epsilon \in H^2(\Gal(\Kp/\overline{k}(t)), \hZ(1)_{(p|\Gamma|)'})$. In this subsection, we identify $\epsilon$ as a class in the \'etale cohomology $H^2_{\et}(X_{\overline{k}}, \hZ(1)_{(p|\Gamma|)'})$ for the curve $X_{\overline{k}}$ whose function field is $K\overline{k}$, and we compute the trace of this class in Corollary~\ref{cor:ext-tr}.

	\begin{lemma}\label{lem:another-description-iota}
		Use the notation in Proposition~\ref{prop:etale-lift}. 
		Let $\overline{g}_i\in \Omega_K$ denote the image of $g_i$ under the quotient map defined by the presentation \eqref{eq:pres-etfund}.
		Then there is a surjection $\varrho_K$ that factors through $\rho_K$ and is defined by
		\begin{eqnarray*}
			\varrho_K: \Omega_K &\longrightarrow& \widetilde{G} \\
			\overline{g}_i &\longmapsto& \widetilde{\rho}_K(g_i) \quad \text{ for } i=1, \ldots, m-1.
		\end{eqnarray*}
		In particular, we have the following identity 
		\[
			\widetilde{\rho}_{K}(g_1 \cdots g_m) = \varrho_K(\overline{g}_m)^{-1} \widetilde{\rho}_K(g_m).
		\] 
	\end{lemma}
	
	\begin{proof}
		We consider the injective homomorphism 
		\begin{eqnarray*}
			\varepsilon: \Omega_K &\longrightarrow& \langle{g_1, \ldots, g_m}\rangle \\
			\overline{g}_i &\longmapsto& g_i \quad \text{ for } i=1, \ldots, m-1.
		\end{eqnarray*}
		Then the composition map $\varrho_K:=\widetilde{\rho}_K \circ \varepsilon$ is a homomorphism from $\Omega_K$ to $\widetilde{G}$ such that $\varrho_K (\overline{g}_i) = \widetilde{\rho}_K(g_i)$ for $i=1, \ldots, m-1$. Then, since $\rho_K(g_i)$, $i=1, \ldots, m-1$ generate $\Gal(\Kp/\overline{k}(t))$, it follows by Lemma~\ref{lem:app-stem} that $\varrho_K$ is a surjection.
		Because $\overline{g}_1\cdots \overline{g}_{m}=1$, we see that $\varrho_K(\overline{g}_m)=\varrho_K(\overline{g}_1\cdots \overline{g}_{m-1})^{-1}=\widetilde{\rho}_K(g_1 \cdots g_m)^{-1} \widetilde{\rho}_K(g_m)$.
	\end{proof}

	\begin{lemma}\label{lem:tg-isom}
		Let $\widetilde{G}$ and $\pi$ be as in the group extension~\eqref{eq:stem} in Lemma~\ref{lem:extension-whole}. We define $\widetilde{H}:=\pi^{-1}(\Gp(K))$ and $\varpi:=\pi|_{\widetilde{H}}$, and obtain the following stem extension of $\Gp(K)$
		\begin{equation}\label{eq:ext-H}
			1 \longrightarrow \ker \varpi \longrightarrow \widetilde{H} \overset{\varpi}{\longrightarrow} \Gp(K) \longrightarrow 1.
		\end{equation}
		Then, for any integer $n$ satisfying $\gcd(n, p|\Gamma|)=1$, the transgression map defined by the extension $\varpi$
		\[
			tg_{\varpi,n}: H^1(\ker \varpi, \Z/n\Z) \longrightarrow H^2(\Gp(K), \Z/n\Z)
		\]
		is an isomorphism.
	\end{lemma}
	
	\begin{proof}
		We first show that \eqref{eq:ext-H} is a stem extension. It is central because \eqref{eq:stem} is central. So it suffices to prove that $\ker \varpi$ is contained in $[\widetilde{H},\widetilde{H}]$, or equivalently, that the homomorphism of abelianizations $\widetilde{H}^{ab} \to \Gp(K)^{ab}$ induced by $\varpi$ is an isomorphism. Let $N$ denote the kernel of $\widetilde{H}^{ab} \to \Gp(K)^{ab}$. Then the short exact sequence $1 \to N \to \widetilde{H}^{ab} \to \Gp(K)^{ab} \to 1$ has an action of $\Gamma$ inherited from $\varpi$, and in particular, the $\Gamma$ action on $N$ is trivial. Because $\Gp(K)$ is an admissible $\Gamma$-group (Lemma~\ref{lem:extension-whole}\eqref{item:e-w-1}), its abelianization is a $\hZ_{(|\Gamma|)'}[\Gamma]$-module whose $\Gamma$-invariant is trivial. Then, since $\hZ_{(|\Gamma|)'}[\Gamma]$ is a semisimple ring, we have 
		\begin{equation}\label{eq:2.9}
			\widetilde{H}^{ab} \simeq N \oplus \Gp(K)^{ab}.
		\end{equation}
		On the other hand, by taking quotient of $\widetilde{G}$ in \eqref{eq:stem} modulo $[\widetilde{H}, \widetilde{H}]$, we obtain a short exact sequence 
		\[
			1 \longrightarrow N \longrightarrow \widetilde{H}^{ab}\rtimes \Gamma \longrightarrow \Gp(K)^{ab} \rtimes \Gamma \longrightarrow 1,
		\]
		which is a stem extension because $\pi$ is stem. The isomorphism \eqref{eq:2.9} implies that the abelianization of $\widetilde{H}^{ab} \rtimes \Gamma$ is isomorphic to $N \times \Gamma^{ab}$. Therefore, $N$ must be trivial, so \eqref{eq:ext-H} is stem.

		Consider the transgression map $tg_{\pi,n}: H^1(\ker \pi, \Z/n\Z) \to H^2(\Gal(\Kp/\overline{k}(t)), \Z/n\Z)$ defined by $\pi$. Since $\Gal(\Kp/\overline{k}(t))^{ab}\simeq \Gamma^{ab}$ and we've shown that $\pi$ is the pro-prime-to-$(p|\Gamma|)$ Schur covering of $\Gal(\Kp/\overline{k}(t))$, the prime-to-$(p|\Gamma|)$ part of the Schur multiplier of $\widetilde{G}$ is trivial \cite[Chap.10, Prop.1.12]{GroupRep}. So $H^2(\widetilde{G}, \Z/n\Z)=0$, and then $tg_{\pi,n}$ is surjective. 
		On the other hand, since $\pi$ is a stem extension, $\widetilde{G}^{ab}\simeq \Gal(\Kp/\overline{k}(t))^{ab}$, so the inflation map $H^1(\Gal(\Kp/\overline{k}(t)), \Z/n\Z) \to H^1(\widetilde{G}, \Z/n\Z)$ is an isomorphism, which shows that $tg_{\pi,n}$ is an isomorphism. 
		
		Finally, consider the exact sequence
		\[
			0 \to H^1(\Gp(K), \Z/n\Z) \to H^1(\widetilde{H}, \Z/n\Z) \to H^1(\ker \varpi, \Z/n\Z) \overset{tg_{\varpi,n}}{\longrightarrow} H^2(\Gp(K),\Z/n\Z) \to \cdots .
		\]
		Because $\widetilde{H}^{ab} \simeq \Gp(K)^{ab}$, $tg_{\varpi, n}$ is injective.
		Also, 
		$tg_{\pi,n}$ is the composition map
		\[
			H^1(\ker \pi, \Z/n\Z) \overset{=}{\longrightarrow} H^1(\ker \varpi, \Z/n\Z) \overset{tg_{\varpi, n}}{\longrightarrow} H^2(\Gp(K), \Z/n\Z) \overset{\res}{\longrightarrow} H^2(\Gal(\Kp/\overline{k}(t)), \Z/n\Z).
		\] 
		Recall that we've shown in the proof of Lemma~\ref{lem:Galois-H2} that the map $\res$ is an isomorphism. So we see that $tg_{\pi,n}$ being isomorphic and $tg_{\varpi, n}$ being injective implies that $tg_{\varpi,n}$ is also an isomorphism.
	\end{proof}

	Use the notation in Lemmas~\ref{lem:another-description-iota} and \ref{lem:tg-isom}. Let $X$ be the smooth projective curve over $k$ whose function field is $K$, and let $X_{\overline{k}}$ denote the base change of $X$ to $\overline{k}$. 
	Assume that $n$ is a positive integer relatively prime to $p|\Gamma|$.
	We consider the following composition map
	\begin{equation}\label{eq:map-tr}
	\begin{tikzcd}
		H^1(\ker \varpi, \mu_n) \arrow["tg_{\varpi,n}", "\sim"']{r} & H^2(\Gp(K), \mu_n) \arrow["\sim"']{r} & H^2_{\et}(X_{\overline{k}}, \mu_n) \arrow["\tr_{X,n}", "\sim"']{r} & \Z/n\Z.
	\end{tikzcd}
	\end{equation}
	Here $tg_{\varpi,n}$ is obtained by tensoring $\mu_n$ to the transgression map in Lemma~\ref{lem:tg-isom}. The middle arrow follows by the canonical isomorphism between Galois cohomology and \'etale cohomology and the fact (we proved in Lemma~\ref{lem:Galois-H2}) that the inflation map $H^2(\Gp(K), \mu_n)\to H^2(G_{\O}(K\overline{k}), \mu_n)$ is an isomorphism.
		 The last arrow is the trace obtained by the Poincar\'e duality.
		 
		 Let $\fraki$ denote the homomorphism 
		 \begin{eqnarray}
		 	\ker \varpi &\overset{\fraki}{\longrightarrow}& \hat{\Z}(1)_{(p|\Gamma|)'} \label{eq:fraki}\\
			\widetilde{\rho}_K(g_1\cdots g_m) &\longmapsto& I(\underline{g}), \nonumber
		 \end{eqnarray}
		 and then let $\fraki_n: \ker \varpi \to \mu_n$ denote the composition of $\fraki$ and the natural surjection $\hat{\Z}(1)_{(p|\Gamma|)'} \to \mu_n$.

	\begin{lemma}\label{lem:tr-Gamma}
		 
		The image of $\fraki_n$ under \eqref{eq:map-tr} is $-|\Gamma|$ mod $n$.
	\end{lemma}

	\begin{proof}
		We first recall the definition of $\tr_{X,n}$. Let $C^1(X_{\overline{k}})$ be the free group on the set of prime 1-cycles on $X_{\overline{k}}$. Then the cycle map $\cl_X: C^1(X_{\overline{k}}) \to H^2_{\et}(X_{\overline{k}}, \mu_{n})$ is the composition of the canonical maps $C^1(X_{\overline{k}}) \to \Pic(X_{\overline{k}})$ and $\Pic(X_{\overline{k}}) \to H^2_{\et}(X_{\overline{k}}, \mu_n)$, and $\cl_X$ maps a prime 1-cycle to the element whose $\tr_{X,n}$-trace is 1. Rewriting everything in Galois cohomology, these two maps are two connecting homomorphisms. 
		We let $S$ denote the set of all places of $K\overline{k}$, and for any extension $L$ of $K\overline{k}$, we let $S(L)$ denote the set of all places of $L$.
		Then the first map $C^1(X_{\overline{k}}) \to \Pic(X_{\overline{k}})$ is the connecting homomorphism 
		\[
			\delta_1: \quad \bigoplus_{\frakp \in S} \Z = H^0\left(G_{\O}(K\overline{k}), \bigoplus_{\frakP \in S(K_{\O})} \Z \right) \longrightarrow H^1\left(G_{\O}(K\overline{k}), \calO_{K_{\O}}^{\times}\right)
		\]
		associated to the short exact sequence of $G_{\O}(K\overline{k})$-modules
		\begin{equation}\label{eq:excision}
		\begin{tikzcd}
			1 \arrow{r} &\calO_{K_{\O}}^{\times} \arrow{r} &K_{\O}^{\times} \arrow["\oplus \ord_{\frakP}"]{r} & \bigoplus_{\frakP \in S(K_{\O})} \Z \arrow{r} &0.
		\end{tikzcd}
		\end{equation}
		The second map $\Pic(X_{\overline{k}}) \to H^2_{\et}(X_{\overline{k}}, \mu_n)$ is 
		\[
			\delta_2: \quad H^1 \left(G_{\O}(K\overline{k}), \calO_{K_{\O}}^{\times}\right) \longrightarrow H^2(G_{\O}(K\overline{k}), \mu_n),
		\]
		associated to the Kummer sequence $1 \to \mu_n \to \calO_{K_{\O}}^{\times} \overset{n}{\longrightarrow} \calO_{K_{\O}}^{\times} \to 1$. 
		
		There is a short exact sequence 
		\[
			 1 \longrightarrow N \longrightarrow G'_S(K) \longrightarrow G'_{\O}(K) \longrightarrow 1,
		\]
		where $G'_S(K)$ denotes $G_S(K\overline{k})_{(p|\Gamma|)'}$.
		Applying the Hochschild--Serre exact sequence, we obtain
		\begin{eqnarray}
			0 \longrightarrow H^1(G'_{\O}(K), \mu_{n}) &\longrightarrow& H^1(G'_S(K), \mu_{n}) \longrightarrow H^1(N, \mu_{n})^{G'_{\O}(K)} \nonumber\\
			&\overset{tg}{\longrightarrow}& H^2(G'_{\O}(K), \mu_{n}) \longrightarrow 0 \label{eq:HS-seq},
		\end{eqnarray}
		where the last arrow uses the fact that $G'_S(K)$ is a free pro-prime-to-$p|\Gamma|$ group. 
		By the Riemann existence theorem (a prime-to-$p|\Gamma|$ analogue of \cite[Thm.(10.1.10)]{NSW}),
		$N$ is the free product of the prime-to-$p|\Gamma|$ completions of inertia subgroups at places in $S(\Kp)$, so
		\[
			H^1(N, \mu_{n})^{G'_{\O}(K)} =\left( \bigoplus_{\frakP \in S(\Kp)} H^1(\calT_{\frakP}, \mu_{n}) \right)^ {G'_{\O}(K)}
		= \bigoplus_{\frakp \in S} H^1(\calT_{\frakp}, \mu_{n}),
		\]
		where $\calT_{\frakp}$ (resp. $\calT_{\frakP}$) is the pro-prime-to-$p|\Gamma|$ completion of the local inertia group at $\frakp$ (resp. $\frakP$).
		The last equality above is because $\calT_{\frakP}=\calT_{\frakp}$ for each $\frakP | \frakp$ and $\Gp(K)$ acts transitively and freely on all the primes above $\frakp$, which both follow by the fact that $\Kp/K\overline{k}$ is split completely at all primes. 
		For a place $\frakp \in S$, we let $(K\overline{k})_{\frakp}$ denote the completion of $K\overline{k}$ at $\frakp$.
		Let $a_{\frakp}$ be a uniformizer of the prime $\frakp$ of $(K\overline{k})_{\frakp}$. Then $\calT_{\frakp}$ is the prime-to-$p|\Gamma|$ completion of the $\Gal\left((K\overline{k})_{\frakp}[a_{\frakp}^{1/\infty}] /(K\overline{k})_{\frakp}\right)$, so there is a canonical isomorphism $\hat{\Z}(1)_{(p|\Gamma|)'} \to \calT_{\frakp}$ arising from the action of $\hat{\Z}(1)$ as ring automorphisms of $(K\overline{k})_{\frakp}[a_{\frakp}^{1/\infty}]$ over $(K\overline{k})_{\frakp}$. 
		So we can define a surjection 
		\begin{eqnarray*}
			f_{\frakp}: \Z & \longrightarrow & H^1(\calT_{\frakp}, \mu_n) \\
			1 & \longmapsto & \left( \sigma \mapsto \frac{\sigma(\sqrt[n]{a_{\frakp}})}{\sqrt[n]{a_{\frakp}}} \right) 
		\end{eqnarray*}
		We will prove by explicitly computing the cocycles that the following diagram commutes 
		\begin{equation}\label{eq:diag-com}
		\begin{tikzcd}
			\bigoplus_{\frakp \in S} \Z \arrow["\delta_1"]{r} \arrow["\oplus f_{\frakp}"]{d} & H^1\left(G_{\O}(K\overline{k}), \calO_{K_{\O}^{\times}} \right) \arrow["\delta_2"]{d} \\
			\bigoplus_{\frakp \in S} H^1(\calT_{\frakp}, \mu_n) \arrow["tg"]{r} & H^2(G_{\O}(K\overline{k}), \mu_n),
		\end{tikzcd}
		\end{equation}
		where the lower map $tg$ is the composition of $tg$ in \eqref{eq:HS-seq} and the canonical isomorphism $H^2(\Gp(K), \mu_n) \simeq H^2(G_{\O}(K\overline{k}), \mu_n)$. 
		
		Let $1_{\frakp}$ denote the element $1$ in the component $\Z$ for $\frakp$. Then $f_{\frakp}(1_{\frakp})\in H^1(\calT_{\frakp}, \mu_n) \subset H^1(N, \mu_n)^{\Gp(K)}=\Hom_{G'_S(K)}(N, \mu_n)$ is a surjection from $N \to \mu_n$ whose kernel is closed under the conjugation in $G'_S(K)$. Let $L$ denote the field fixed by $\ker f_{\frakp}(1_{\frakp})$. By the Kummer theory, there exists $b\in \Kp$ such that $L=\Kp(\sqrt[n]{b})$ and $\ord_{\frakP}(b)=1$ for any $\frakP \in S(\Kp)$ above $\frakp$.  
		Then the isomorphism $\Gal(L/\Kp)\simeq \mu_n$ induced by $f_{\frakp}(1_{\frakp})$ is defined by $g \mapsto \frac{g(\sqrt[n]{b})}{\sqrt[n]{b}}$ for $g\in \Gal(L/\Kp)$.
		So we have a central group extension
		\[\begin{tikzcd}
			1 \arrow{r} &\mu_n \arrow{r} &\Gal(L/K\overline{k}) \arrow{r} &\Gp(K) \arrow{r} \arrow["t", swap, bend right=30]{l} &1.
		\end{tikzcd}\]
		and we pick a splitting $t: \Gp(K) \to \Gal(L/K\overline{k})$. Then $tg \circ f_{\frakp}(1_{\frakp})$ is the class corresponding to this central extension, and explicitly, the class containing the 2-cocyle 
		\begin{eqnarray*}
			z: \Gp(K) \times \Gp(K) & \longrightarrow& \mu_n \\
			(\sigma, \tau)   & \longmapsto &\frac{t(\sigma\tau)^{-1}t(\sigma)t(\tau) (\sqrt[n]{b})}{ \sqrt[n]{b}}.
		\end{eqnarray*}
		On the other hand, let's compute $\delta_2 \circ \delta_1 (1_{\frakp})$. The image of $b$ under the quotient map in \eqref{eq:excision} is invariant under the $G_{\O}(K\overline{k})$ action. Then, because we assume $\ord_{\frakP}(b)=1$, the image of $b$ is $1_{\frakp}$ in $\oplus_{\frakp\in S} \Z$. Then by chasing the diagram for the connecting homomorphism $\delta_1$, $\delta_1(1_{\frakp})$ is the class containing the 1-cocycle 
		\begin{eqnarray*}
			x: G_{\O}(K\overline{k}) &\longrightarrow& \calO_{K_{\O}^{\times}} \\
			\sigma &\longmapsto& \frac{\sigma(b)}{b}.
		\end{eqnarray*}  
		We choose a splitting $s: \calO_{K_{\O}}^{\times} \to \calO_{K_{\O}}^{\times}$ of the Kummer sequence such that $s\left( \frac{\sigma(b)}{b}\right) = \frac{t(\sigma) (\sqrt[n]{b})}{\sqrt[n]{b}}$ for any $\sigma \in G_{\O}(K\overline{k})$. There exists such a splitting because $\left( \frac{t(\sigma) (\sqrt[n]{b})}{\sqrt[n]{b}}\right)^n = \frac{\sigma(b)}{b}$. Then by chasing the diagram for $\delta_2$, we see that $\delta_2 \circ \delta_1 (1_{\frakp})$ is the class of the 2-cocycle 
		\begin{eqnarray*}
			y: G_{\O}(K\overline{k}) \times G_{\O}(K \overline{k}) & \longrightarrow & \mu_n \\
			(\sigma, \tau) & \longmapsto & \sigma( s(x(\tau))) s(x(\sigma)) s(x(\tau \sigma))^{-1} = \frac{t(\sigma) t(\tau) (\sqrt[n]{b})}{t(\sigma \tau)(\sqrt[n]{b})}.
		\end{eqnarray*}
		The two 2-cocycles $z$ and $y$ are the same because $y(\sigma, \tau)=t(\sigma\tau)(z(\sigma, \tau))$ equals $z(\sigma, \tau)$ as $z(\sigma, \tau) \subset \mu_n$ is fixed by $t(\sigma, \tau)$. 
		Therefore, we proved that \eqref{eq:diag-com} commutes, which implies that $tg \circ f_{\frakp}(1_{\frakp})$ corresponds to the element of $H^2_{\et}(X_{\overline{k}}, \mu_n)$ of $\tr_{X, n}$-trace 1.

		Recall that $\Omega_K$ denotes the prime-to-$p$ completion of the \'etale fundamental group $\pi_1^{\et}(\PP^1_{\overline{k}} -\{P_1, \ldots, P_m\})$ where $P_1, \ldots, P_m$ correspond to the primes ramified in $K\overline{k}/\overline{k}(t)$. So the prime-to-$p|\Gamma|$ completion of $\ker(\Omega_K \to \Gal(K\overline{k}/\overline{k}(t)))$ is a quotient of $G'_S(K)$. 
		Let $\varrho_K$ denote the surjection defined in Lemma~\ref{lem:another-description-iota}, and let $\alpha: G'_S(K) \to \widetilde{H}$ that factors through the map obtained by restricting $\varrho_K$ to $\ker(\Omega_K \to \Gal(K\overline{k}/\overline{k}(t)))$ and then taking the prime-to-$p|\Gamma|$ completion, we obtain the following commutative diagram
		\[\begin{tikzcd}
			1 \arrow{r} & N  \arrow{r} \arrow[two heads]{d} & G'_S(K) \arrow{r} \arrow["\alpha", two heads]{d} & G'_{\O}(K) \arrow{r} \arrow[equal]{d} & 1 \\
			1 \arrow{r} & \ker \varpi \arrow{r} & \widetilde{H} \arrow["\varpi"]{r} & G'_{\O}(K) \arrow{r} &1,
		\end{tikzcd}\]
		which induces a commutative diagram of the Hochschild--Serre exact sequences 
		\begin{equation}\label{eq:diag-HS}
		\begin{tikzcd}
			\cdots \arrow{r} &H^1(N,\mu_{n})^{G'_{\O}(K)} \arrow[two heads, "tg"]{r} & H^2(G'_{\O}(K), \mu_{n}) \arrow[equal]{d} \arrow{r} & 0 \\
			\cdots \arrow["0"]{r} & H^1(\ker \varpi, \mu_{n})  \arrow["\text{induced by } \alpha"]{u} \arrow["\sim"', "tg_{\varpi, n}"]{r} &H^2(G'_{\O}(K), \mu_{n}) \arrow{r} & 0.
		\end{tikzcd}
		\end{equation}
		Here the isomorphism in the second row follows from Lemma~\ref{lem:tg-isom}.

		We use the notation in Lemma~\ref{lem:another-description-iota} to let $\overline{g}_i$ denote the image of $g_i$ in $\Omega_K$ for each $i=1,\ldots, m$.
		By Lemma~\ref{lem:another-description-iota}, $\varrho_K$ identifies $\ker \varpi$ as the Galois group of an extension of $K'_{\O}$, denoted by $E/\Kp$, that is ramified only at the primes over the prime $v$ of $\overline{k}(t)$ corresponding to $P_m$. 
		Let $\frakp_1, \ldots, \frakp_c$ denote all the primes of $K\overline{k}$ lying above $v$, and let $e$ denote the ramification index of the extension $\frakp_i/v$. 
		Let $h$ denote $\varrho_K(\overline{g}_m^e)$ in $\widetilde{H}$.
		Then $h$ generates the inertia subgroup of $E/\Kp$ at one prime above $v$, and because $h \in \ker \varpi=\ker \pi$ is in the center of $\widetilde{G}$ (where $\pi$ and $\widetilde{G}$ are defined in Lemma~\ref{lem:extension-whole}), we have that $h$ generates the inertia subgroup at each prime above $v$.
				
		Let $I(\underline{g})_n$ denote the image of $I(\underline{g})$ under the surjection $\hat{\Z}(1)_{(p)'} \to \mu_n$. 
		For an element $\zeta \in \mu_n$, a pro-cyclic group $D$, and one generator $d$ of $D$, we write $[d \mapsto \zeta]$ for the element of $\Hom(D, \mu_n)$ that sends $d$ to $\zeta$.
	 	Then $\fraki_n$ is the element $[\widetilde{\rho}_K(g_1\cdots g_m) \mapsto I(\underline{g})_n]$ in $H^1(\ker \varpi, \mu_n)$. 
		Because $\widetilde{\rho}_K(g_1 \cdots g_m)$ is in the center of $\widetilde{G}$, it follows by Lemma~\ref{lem:another-description-iota} and $\ord(\widetilde{\rho}_K(g_m))=e$ that 
		\[
			\varrho_K(\overline{g}_m^e)=\left( \widetilde{\rho}_K(g_m) \widetilde{\rho}_K(g_1\cdots g_m)^{-1} \right)^e = \widetilde{\rho}_K(g_1\cdots g_m)^{-e},
		\]
		and hence the image of $\fraki_n$ in the $\frakp_i$-compoment of $H^1(N,\mu_{n})^{\Gp(K)}=\oplus_{\frakp} H^1(\calT_{\frakp}, \mu_n)$ is $[\overline{g}_m^e \mapsto I(\underline{g})_n^{-e}]$. Recall that $I(\underline{g})_n$ is defined to be $\frac{\overline{g}_m(\sqrt[n]{a_v})}{\sqrt[n]{a_v}}$ for any uniformizer $a_v$ of the ring of integers of $\overline{k}(t)_{v}$. So
		\[
			\left[\overline{g}_m^e \mapsto I(\underline{g})_n^{-e}\right] = \left[\overline{g}_m^e \mapsto \left(\frac{\overline{g}_m^e(\sqrt[n]{a_v})}{\sqrt[n]{a_v}}\right)^{-1}\right]=\left[\overline{g}_m^e \mapsto \left(\frac{\overline{g}_m^e(\sqrt[n]{a_\frakp})}{\sqrt[n]{a_\frakp}}\right)^{-e}\right].
		\]
		Here we choose $a_v$ and $a_{\frakp}$ to be compatible uniformizers such that $a_v=a_{\frakp}^e$, and then the last equality is because $(K\overline{k})_\frakp/\overline{k}(t)_v$ is a degree-$e$ completely ramified extension.
		So we see that the image of $\fraki_n$ under the composition map $H^1(\ker \varpi, \mu_n) \to H^1(\calT_{\frakp_i}, \mu_n) \to H^2(\Gp(K), \mu_n)$ has trace $-e$. By summing over all $\frakp_i$ and considering the diagram \eqref{eq:diag-HS}, we see that the image of $\fraki_n$ in $H^2(\Gp(K), \mu_n)$ has trace $-ec=-[K\overline{k}: \overline{k}(t)]=-|\Gamma|$.		
	\end{proof}
	
	\begin{corollary}\label{cor:ext-tr}
		The map $\fraki$ in \eqref{eq:fraki} identifies \eqref{eq:ext-H} as a central extension of $\Gp(K)$ with kernel $\hat{\Z}(1)_{(p|\Gamma|)'}$. Then the class of $H^2(\Gp(K), \hat{\Z}(1)_{(p|\Gamma|)'})$ corresponding to this central extension is 
		\begin{enumerate}
			\item \label{item:ext-tr-1} the image of $\fraki$ under the inverse limit $tg_{\varpi}:=\varprojlim_{n} tg_{\varpi, n}$ of the transgression map described in \eqref{eq:map-tr}; and
			\item \label{item:ext-tr-2} the image of $\epsilon$ in Remark~\ref{rmk:etale-lift} under the restriction map induced by $\Gp(K) \hookrightarrow \Gal(\Kp/\overline{k}(t))$.
		\end{enumerate}
		In particular, this class corresponds to the class of $H_{\et}^2(X_{\overline{k}}, \hZ(1)_{(p|\Gamma|)'})$ of trace $-|\Gamma|$.
	\end{corollary}
	
	\begin{proof}
		Let $\varphi$ denote a splitting of $\varpi$ such that $\varphi(1)=1$. Consider the following 2-cocycle in $Z^2(\Gp(K), \ker \pi)$
		\[
			\lambda: \Gp(K) \times \Gp(K) \to \ker \varpi
		\]
		defined by $\lambda(x,y)=\varphi(xy)^{-1}\varphi(x)\varphi(y)$, which is the normalized 2-cocycle corresponding to the extension $\varpi$. Then the transgression map $tg_{\varpi}$ maps $\fraki$ to the class of $\fraki \circ \lambda$ \cite[Pages 16-17]{GroupRep}.
		Since the class of $\fraki\circ \lambda$ is the normalized 2-cocycle corresponding to the central extension of $\Gp(K)$ with kernel $\hat{\Z}(1)_{(p|\Gamma|)'}$ described in the corollary, we proved the statement~\eqref{item:ext-tr-1}. The statement \eqref{item:ext-tr-2} is trivial, and the last statement follows by the statement \eqref{item:ext-tr-1} and Lemma~\ref{lem:tr-Gamma}.
	\end{proof}

\subsection{Definition of the invariant $\omega_{L/K}$}\label{ss:inv-L/K}

	\begin{definition}[Invariant $\omega_{L/K}$ associated to $L/K$]\label{def:inv-L/K}
		For a subextension $L/K$ of $K'_{\O}/K$ such that $L$ is Galois over $k(t)$, we define \emph{the $\omega$-invariant $\omega_{L/K}$ associated to $L/K$} to be the $\Frob_k$-equivairant homomorphism $\omega_{L/K}$ that is the composition map as follows 
		\[
			\omega_{L/K}: \hat{\Z}(1)_{(p|\Gamma|)'} \overset{\omega_K}{\longrightarrow} H_2(\Gal(K'_{\O}/\overline{k}(t)), \Z)_{(p|\Gamma|)'} \overset{\coinf}{\longrightarrow} H_2(\Gal(L\overline{k}/\overline{k}(t)), \Z)_{(p|\Gamma|)'}
		\]
		where the map $\coinf$ is the coinflation map defined by the surjection $\alpha:\Gal(K'_{\O}/\overline{k}(t)) \to \Gal(L\overline{k}/\overline{k}(t))$.
		In particular, when $L=\Ks$, we denote $\omega^{\#}_K:=\omega_{\Ks/K}$. 
	\end{definition}

	\begin{lemma}\label{lem:comp-schur}
		Let $n$ be an integer that is divisible by $|\Gamma|$, and $G, H$ are two extensions of $\Gamma$ such that $\ker(G\to \Gamma)$ and $\ker(H \to \Gamma)$ are pro-prime-to-$n$ admissible $\Gamma$-groups. Assume that there is a surjection $\alpha: G \to H$ and a $(n)'$-Schur covering of $H$ 
		\begin{equation}\label{eq:stem-H}
			1 \longrightarrow \ker \pi' \longrightarrow S \overset{\pi'}{\longrightarrow} H \to 1.
		\end{equation} 
		Then there exists a $(n)'$-Schur covering of $G$ such that the following diagram commutes 
		\[\begin{tikzcd}
		1 \arrow{r} & \ker \pi \arrow{r} \arrow{d}& T \arrow["\pi"]{r} \arrow{d} & G \arrow{r} \arrow["\alpha"]{d}&1 \\
		1 \arrow{r} & \ker \pi' \arrow{r} & S \arrow["\pi'"]{r} & H \arrow{r} & 1,
		\end{tikzcd}\]
		and the first vertical arrow is the composition map
		\[\begin{tikzcd}
			\ker \pi \arrow["d_{\pi}^{-1}", "\sim"']{r} &H_2(G, \Z)_{(n)'} \arrow["\coinf"]{r} &H_2(H, \Z)_{(n)'}\arrow["d_{\pi'}", "\sim"']{r} & \ker \pi'
		\end{tikzcd}\]
		where $d_{\pi}$ and $d_{\pi'}$ is the differential maps defined by $\pi$ and $\pi'$ respectively.
	\end{lemma}
	
	\begin{proof}
		There is a commutative diagram induced by $\alpha:G \to H$
		\[\begin{tikzcd}
			H^2(H, \ker \pi') \arrow{r} \arrow{d} & \Hom(H_2(H, \Z), \ker \pi') \arrow{d} \\
			H^2(G, \ker \pi') \arrow{r} & \Hom(H_2(G,\Z), \ker \pi'),
		\end{tikzcd}\]
		where the vertical maps are induced by $\alpha$, and the horizontal maps are from the universal coefficient theorem. Moreover, since $\ker(G \to \Gamma)$ and $\ker(H \to \Gamma)$ are admissible $\Gamma$-groups, we have that $G^{ab}$ and $H^{ab}$ are both isomorphic to $\Gamma^{ab}$, so $\Ext^1_{\Z}(G,\ker \pi')=\Ext^1_{\Z}(H, \ker \pi')=0$ and the two horizontal maps are isomorphisms.
		
		The image of the class of $\pi'$ in the upper-right entry of the above diagram is the differential map $d_{\pi'}$ (we view $d_{\pi'}$ as the quotient map $H_2(H,\Z) \to \ker \pi'$ that factors through $H_2(H,\Z)_{(n)'}\overset{\sim}{\longrightarrow} \ker \pi'$). We let $d$ denote the image of $d_{\pi'}$ in the lower-right entry, so $d$ is the composition $H_2(G, \Z)_{(n)'} \to H_2(H, \Z)_{(n)'} \overset{d_{\pi'}}{\longrightarrow} \ker \pi'$.
		Let $\overline{\pi}: E \to G$ be a group extension of $G$ whose corresponding class in $H^2(G,\ker \pi')$ is the preimage of $d$ . Then $E$ has a subgroup $E'$ such that $E' \to G$ is a stem extension of kernel isomorphic to $\im d$.
		 Because every stem extension is a homomorphic image of a Schur covering group \cite[Chap.11, Cor.2.4]{GroupRep}, $E'\to G$ can be extended to a $(n)'$-Schur covering $\pi: T \to G$. Then one can check by our construction that $\pi$ fits into the diagram in the lemma. 
	\end{proof}
	
	The following is an analogue of Proposition~\ref{prop:etale-lift}, which gives a group theoretical description of the invariant $\omega_{L/K}$. 
	
	\begin{corollary}\label{cor:compatible-cover}
		Use the assumptions and notation in Definition~\ref{def:inv-L/K}. Let $\underline{g}=(g_1, \ldots, g_m)$ be a system of inertia generators for $K$, $\varphi: \langle g_1, \ldots, g_m \rangle \to \Gal(L\overline{k}/\overline{k}(t))$ the surjection defined by $\underline{g}$ and $L$, and 
		\[
			1 \longrightarrow B \longrightarrow S \overset{\pi'}{\longrightarrow} \Gal(L\overline{k}/\overline{k}(t)) \longrightarrow 1
		\]
		be a $(p|\Gamma|)'$-Schur covering of $\Gal(L\overline{k}/\overline{k}(t))$. Then there is a unique surjection $\widetilde{\varphi}: \langle g_1, \ldots, g_m \rangle \to S$ such that $\varphi=\pi' \circ \widetilde{\varphi}$ and $\ord(\widetilde{\varphi}(g_i))=\ord(\varphi(g_i))$ for each $i=1, \ldots, m$. 
		
		The invariant $\omega_{L/K}$ is the composition of the map $\hat{\Z}(1)_{(p|\Gamma|)'} \to B$ defined by $I(\underline{g}) \mapsto \widetilde{\varphi}(g_1\cdots g_m)$ and the inverse of the differential map $H_2(\Gal(L\overline{k}/\overline{k}(t)), \Z)_{(p|\Gamma|)'} \overset{\sim}{\longrightarrow} B$ defined by $\pi'$.
	\end{corollary}	
	
	\begin{proof}
		The existence and uniqueness of $\widetilde{\varphi}$ can be proven by repeating the proof of Proposition~\ref{prop:etale-lift}\eqref{item:etale-lift-1}.
		
		Note that $\Gp(K)$ is an admissible $\Gamma$-group, so is $\Gal(L\overline{k}/K\overline{k})$. By applying Lemma~\ref{lem:comp-schur} to $\alpha: \Gal(\Kp/\overline{k}(t)) \twoheadrightarrow \Gal(L\overline{k}/\overline{k}(t))$, we obtain
		$\pi: \widetilde{G} \to \Gal(\Kp/\overline{k}(t))$ in the diagram in Lemma~\ref{lem:comp-schur}. Then by Proposition~\ref{prop:etale-lift}, for a choice of $\underline{g}$, we have the following commutative diagram 
		\[\begin{tikzcd}
			& \langle g_1, \ldots, g_m \rangle \arrow["\rho_K"]{d} \arrow["\widetilde{\rho_K}", swap]{dl} \arrow[bend left=800, "\varphi"]{dd} \arrow[bend right=70, looseness=2, "\widetilde{\varphi}", swap]{ddl}\\
			\widetilde{G} \arrow["\pi"]{r} \arrow{d} & \Gal(\Kp/\overline{k}(t)) \arrow["\alpha"]{d} \\
			S \arrow["\pi'"]{r} & \Gal(L\overline{k}/\overline{k}(t)).
		\end{tikzcd}\]
		Then the corollary follows by the definitions of $\omega_K$ and $\omega_{L/K}$, Proposition~\ref{prop:etale-lift}\eqref{item:etale-lift-3}, and Lemma~\ref{lem:comp-schur}.
	\end{proof}

		\begin{remark}\label{rmk:inv}
		\begin{enumerate}
			\item\label{item:rmk-inv-1} The invariant $\omega_{L/K}$ is not necessarily an isomorphism, because the coinflation map is not.
			\item\label{item:rmk-inv-2} Assume that $L$ is contained in $\Ks$. Then $\Frob_k$ acts trivially on $\Gal(L\overline{k}/\overline{k}(t))$, and $\Gal(L\overline{k}/\overline{k}(t))$ is naturally isomorphic to $\Gal(L/k(t))$. In this case, the invariant $\omega_{L/K}$ factors through the prime-to-$p|\Gamma|$ part of the roots of unity of $k$, which we denote by $\mu(k)_{(p|\Gamma|)'}$. So by abuse of notation, we write 
			\[
				\omega_{L/K}: \mu(k)_{(p|\Gamma|)'} \longrightarrow  H_2(\Gal(L/k(t)), \Z)_{(p|\Gamma|)'}.
			\]
			The image of $w_{L/K}$ is contained in the $|\mu(k)_{(p|\Gamma|)'}|$-torsion subgroup of $H_2(\Gal(L/k(t)), \Z)_{(p|\Gamma|)'}$.
			\item\label{item:rmk-inv-3} Every argument in this section works when replacing ``pro-prime-to-$(p|\Gamma|)$ completion'' with ``pro-$\ell$ completion'', for any prime $\ell$ that does not divide $p|\Gamma|$. So we can define the pro-$\ell$ analogues of the invariants $\omega_K$ and $\omega_{L/K}$. Explicitly, under the assumptions in Definitions~\ref{def:inv-K} and \ref{def:inv-L/K}, we can define $\Frob_k$-equivariant
			\[
				\omega_{K,\ell}: \Z_\ell(1) \overset{\sim}{\longrightarrow} H_2(\Gal(\Kp/\overline{k}(t)),\Z)(\ell) \quad \text{and}
			\]
			\[
				\omega_{L/K,\ell}: \Z_{\ell}(1) \longrightarrow H_2(\Gal(L\overline{k}/\overline{k}(t)), \Z)(\ell),
			\]
			and give group theoretical descriptions of these invariants.
		\end{enumerate}
	\end{remark}

\section{Relationship between $\omega_K$ and the Weil pairing associated to $K$}\label{sect:Weil-pairing}

	Throughout this section, we assume that $k$, $K$, $p$ and $\Gamma$ be as in Notation~\ref{not}, and $\ell$ is a prime such that $\ell \nmid p|\Gamma|$. In this section, we will discuss how the invariant $\omega_K$ is abelianized to the Weil pairing associated to $K$. In \S\ref{ss:cup-product}, we will first explain how a pro-$\ell$ presentation of $G_{\O}(K\overline{k})(\ell)$ determines its cup product. Then in \S\ref{ss:weil-pairing}, we will show how $\omega_K$ determines the Weil pairing and how $\omega_{L/K}$ determines the image of Weil pairing associated to $L/K$.

\subsection{Presentation of $G_{\O}(K\overline{k})(\ell)$ determines cup product}\label{ss:cup-product}
	
	Because $\Gamma$ is assumed to be non-trivial (see Notation~\ref{not}),
	by \cite[Thm.10.1.2(i)(b) and (iii)]{NSW}, $G_{\O}(K\overline{k})(\ell)$ is a pro-$\ell$ Demu\v{s}kin group generated by $2g$ generators whose abelianization is isomorphic to $\Z_{\ell}^{2g}$, and by a nice choice of the generating set, it admits a pro-$\ell$ presentation in the form of
	\begin{equation*}
		G_{\O}(K\overline{k})(\ell) \simeq  \langle x_1, \ldots, x_{2g} \mid [x_1, x_2][x_3,x_4] \cdots [x_{2g-1}, x_{2g}]=1 \rangle.
	\end{equation*}
	
	 Let $F$ be a free pro-$\ell$ group on $2g$ generators $x_1, \ldots, x_{2g}$ and the following a pro-$\ell$ presentation of $G_{\O}(K\overline{k})(\ell)$ (which is called a \emph{minimal presentation} of $G_{\O}(K\overline{k})(\ell)$ because the minimal number of generators for $F$ and $G_{\O}(K\overline{k})(\ell)$ are equal)
	\begin{equation}\label{eq:pre-H}
		1 \longrightarrow R \longrightarrow F \overset{\phi}{\longrightarrow} G_{\O}(K\overline{k})(\ell) \longrightarrow 1.
	\end{equation}
	Since $F$ is free and $F^{ab}\simeq G_{\O}(K\overline{k})(\ell)^{ab}\simeq \Z_{\ell}^{2g}$, the transgression map in the Hochschild--Serre exact sequence of \eqref{eq:pre-H} is isomorphic, and we denote it by
	\[
		tg: H^1(R, \Z_\ell)^{G_{\O}(K\overline{k})(\ell)} \overset{\sim}{\longrightarrow} H^2(G_{\O}(K\overline{k})(\ell), \Z_\ell).
	\]
	Then for every element $\lambda$ of $R$, there is a trace map $\tr_{\lambda}$ associated to $\lambda$ as follows
	\begin{eqnarray}
		\tr_{\lambda}: H^2(G_{\O}(K\overline{k})(\ell), \Z_\ell) & \longrightarrow& \Z_{\ell}  \nonumber \\
		\varphi & \longmapsto& (tg^{-1} \varphi)(\lambda). \label{eq:trace-lambda}
	\end{eqnarray}
	By definition of Demu\v{s}kin groups and \cite[Prop.(3.9.12)(ii)(iii)]{NSW}, if $\lambda$ generates $R$ as a normal subgroup of $F$, then $\tr_{\lambda}$ is an isomorphism.
	
	For a group $G$, let $\{G^{(i)}\}_{i=1}^{\infty}$ denote the lower central series of $G$, i.e., $G^{(1)}:=G$ and $G^{(i+1)}:=[G^{(i)}, G]$. We assume that $x_1, \ldots, x_{2g}$ are generators of $F$ and $\lambda$ is an element generating $R$ as a normal subgroup. Then by \cite[Prop.(3.9.13)(i)]{NSW}, $\lambda$ can be written in the following form
	\begin{equation}\label{eq:rep-lambda}
		\lambda=\prod_{1\leq i < j \leq 2g} [x_i, x_j]^{a_{ij}} \cdot \lambda', \quad \lambda' \in F^{(3)}, \, a_{ij} \in \Z_\ell.
	\end{equation}
	We let $\chi_1, \ldots, \chi_{2g}$ be the corresponding $\Z_\ell$-basis of $H^1(F, \Z_\ell)=H^1(G_{\O}(K\overline{k})(\ell),\Z_\ell)$ such that $\chi_i(x_j)=1$ if $i=j$ and 0 otherwise. Then, by \cite[Prop.(3.9.13)(ii)]{NSW}, the cup product of $H^1(G_{\O}(K\overline{k})(\ell),\Z_\ell)$ with itself is complelely determined by the $a_{ij}$ in \eqref{eq:rep-lambda}: the bilinear form 
	\begin{equation}\label{eq:cup-product}
		H^1(G_{\O}(K\overline{k})(\ell), \Z_\ell) \times H^1(G_{\O}(K\overline{k})(\ell), \Z_\ell) \overset{\cup}{\longrightarrow} H^2(G_{\O}(K\overline{k})(\ell), \Z_\ell) \overset{\tr_{\lambda}}{\longrightarrow} \Z_\ell
	\end{equation}
	is given by the matrix $M_{\lambda}=[m_{ij}]$ with respect to the basis $\chi_1, \ldots, \chi_{2g}$, where 
	\begin{equation}\label{eq:matrix-M}
		m_{ij}=\tr_{\lambda}(\chi_i \cup \chi_j) = \begin{cases}
			- a_{ij} & \text{ if } i< j \\
			a_{ji} & \text{ if } i>j \\
			0 & \text{ if } i=j.	
		\end{cases}
	\end{equation}

	Since $G_{\O}(K\overline{k})(\ell)$ is the pro-$\ell$ completion of $G_{\O}(K\overline{k})$, it's clear that $H^1(G_{\O}(K\overline{k})(\ell), \Z_{\ell})= H^1(G_{\O}(K\overline{k}), \Z_{\ell})$, and we've shown in the proof of Lemma~\ref{lem:Galois-H2} that the inflation map
	\begin{equation*}
		H^2(G_{\O}(K\overline{k})(\ell), \Z_\ell) \longrightarrow H^2(G_{\O}(K\overline{k}), \Z_\ell)
	\end{equation*}
	is an isomorphism.
	So \eqref{eq:cup-product} implies that the bilinear form 
	\begin{equation}\label{eq:cup-GC}
	\begin{tikzcd}
		H^1(G_{\O}(K\overline{k}), \Z_{\ell}) \times H^1(G_{\O}(K\overline{k}), \Z_{\ell}) \arrow["\cup"]{r} & H^2(G_{\O}(K\overline{k}), \Z_{\ell})\arrow["\tr_{\lambda}"]{r} &\Z_{\ell}
	\end{tikzcd}
	\end{equation}
	is given by the matrix $M_{\lambda}$ in \eqref{eq:matrix-M}.

	We consider a $\ell$-Schur covering of $\Gal(K'_{\O}(\ell)/\overline{k}(t))$ (a pro-$\ell$ analogue of \eqref{eq:stem})
	\begin{equation}\label{eq:Gell}
		1 \longrightarrow \ker \pi_{\ell} \longrightarrow \widetilde{G}_{\ell} \overset{\pi_{\ell}}{\longrightarrow} \Gal(K'_{\O}(\ell)/\overline{k}(t)) \longrightarrow 1,
	\end{equation}
	where $\ker \pi_{\ell} \simeq \Z_{\ell}$.
	Restricting $\pi_{\ell}$ to an extension of $G_{\O}(K\overline{k})(\ell)$, we denote $\widetilde{H}_{\ell}:=\pi_{\ell}^{-1} (G_{\O}(K\overline{k})(\ell))$ and $\varpi_{\ell}:=\pi_{\ell}|_{\widetilde{H}_{\ell}}$, and obtain the following stem extension of $G_{\O}(K\overline{k})(\ell)$
	\begin{equation}\label{eq:ext-ell}
		1 \longrightarrow \ker \varpi_{\ell} \longrightarrow \widetilde{H}_{\ell} \overset{\varpi_{\ell}}{\longrightarrow} G_{\O}(K\overline{k})(\ell) \longrightarrow 1.
	\end{equation}

		\begin{lemma}\label{lem:stem-ell-K}
		Let $1\to R \to F \overset{\phi}{\to} G_{\O}(K\overline{k})(\ell) \to 1$ be a minimal pro-$\ell$ presentation of $G_{\O}(K\overline{k})(\ell)$. Then there is a surjection $\widetilde{\phi}: F \to \widetilde{H}_{\ell}$ such that $\phi= \varpi_{\ell} \circ \widetilde{\phi}$. 
	\begin{enumerate}
	
		\item \label{item:stem-ell-K-1}
		If $\lambda \in R$ generates $R$ as a normal subgroup of $F$, then $\widetilde{\phi}(\lambda)$ is a generator of $\ker \varpi_{\ell}$.
		For two elements $\lambda_1, \lambda_2 \in R$ such that each of them generates $R$ as a normal subgroup, the followings are equivalent
		\begin{enumerate}
			\item \label{item:seK-1} $\widetilde{\phi}(\lambda_1)=\widetilde{\phi}(\lambda_2)$
			\item \label{item:seK-2} $\lambda_1 \equiv \lambda_2 \, \operatorname{mod} \, F^{(3)}$.
		\end{enumerate}
		
		\item \label{item:stem-ell-K-2} The transgression map in the Hochschild--Serre exact sequence of \eqref{eq:ext-ell}
		\[
			tg_{\varpi_{\ell}}: H^1(\ker \varpi_{\ell}, \Z_{\ell}) \longrightarrow H^2( G_{\O}(K\overline{k})(\ell), \Z_{\ell})
		\]
		is an isomorphism. Moreover, for any class $\varphi \in H^2(G_{\O}(K\overline{k})(\ell), \Z_\ell)$ and any $\lambda \in R$ that generates $R$ as a normal subgroup, we have $\tr_{\lambda}(\varphi)= (tg_{\varpi_{\ell}}^{-1} \varphi)(\widetilde{\phi}(\lambda))$.
	\end{enumerate}
	\end{lemma}
	
	\begin{remark}		
		The surjection $\widetilde{\phi}$ is not unique, but the last equality in the statement \eqref{item:stem-ell-K-2} holds for any choice of $\widetilde{\phi}$. Indeed, any two surjections $\widetilde{\phi}_1$ and $\widetilde{\phi}_2$ as described in \eqref{item:stem-ell-K-1} differ by a character $\chi: F \to \ker \varpi_{\ell}$ whose kernel contains $\lambda$ (since $\lambda \in [F,F]$), and hence $\widetilde{\phi}_1(\lambda)=\widetilde{\phi}_2(\lambda)$.
	\end{remark}
	
	\begin{proof}
		First, similarly to Lemma~\ref{lem:tg-isom}, the group extension $\varpi_{\ell}$ is a stem extension of $G_{\O}(K\overline{k})(\ell)$. 
		Since $F^{ab}\simeq G_{\O}(K\overline{k})(\ell)^{ab}$, the following short exact sequence
		\begin{equation}\label{eq:schur-cover-H}
			1 \longrightarrow R/[F,R] \longrightarrow F/[F, R] \longrightarrow G_{\O}(K\overline{k})(\ell) \longrightarrow 1
		\end{equation}
		is a Schur covering group of $H$ (\cite[Chap.11, Thm.2.3(iii) and (iv)(b)]{GroupRep}). Moreover, since $H_2(G_{\O}(K\overline{k})(\ell), \Z) \simeq \Z_{\ell}(1)$, the kernel of the Schur covering $R/[F,R]$ is isomorphic to $\Z_{\ell}$.
		Because $\Ext^1_{\Z}(G_{\O}(K\overline{k})(\ell)^{ab}, \Z_\ell)=\Ext^1_{\Z}(\Z_\ell^{2g}, \Z_{\ell})=0$, $G_{\O}(K\overline{k})(\ell)$ has a unique isomorphism class of Schur covering. So
		$\varpi_{\ell}$ and \eqref{eq:schur-cover-H} are isomorphic as extensions of $G_{\O}(K\overline{k})(\ell)$, i.e., there is an isomorphism $\iota$ such that
		\begin{equation}\label{eq:isom-extensions}
		\begin{tikzcd}
			\widetilde{H}_{\ell}\arrow[two heads, "\varpi_{\ell}"]{dr}  & \\
			F/[F,R] \arrow["\sim"', "\iota"]{u} \arrow[two heads]{r} &G_{\O}(K\overline{k})(\ell)
		\end{tikzcd}
		\end{equation}
		commutes. Then we define $\widetilde{\phi}$ to be the composition of $F \to F/[F,R]$ and $\iota$, and it follows immediately that $\phi=\varpi_{\ell} \circ \widetilde{\phi}$. If $\lambda\in F$ generates $R$ as a normal subgroup, then $\widetilde{\phi}(\lambda)$ generates $\widetilde{\phi}(R)=\ker \varpi_{\ell}$ as a normal subgroup of $\widetilde{H}_{\ell}$, so it is a generator of $\ker \varpi_{\ell}$, which finishes the proof of the first sentence in \eqref{item:stem-ell-K-1} of the lemma.
		
		Because $R$ is contained in $F^{(2)}=[F,F]$, we have $F^{(3)}\supset [F,R]=\ker \widetilde{\phi}$, so \eqref{item:seK-1} implies \eqref{item:seK-2}.
		Taking the quotient of a group $G$ modulo $G^{(i)}$ is a functor mapping surjections to surjections. For a quotient map $\pi: E_1 \twoheadrightarrow E_2$, we denote the induced quotient map by $\pi^{(i)}: E_1/E_1^{(i)} \twoheadrightarrow E_2/E_2^{(i)}$. So we have a commutative diagram of surjections induced by $\phi$, $\widetilde{\phi}$ and $\varpi_{\ell}$
		\[\begin{tikzcd}
			& F/F^{(3)} \arrow[two heads, "\phi^{(3)}"]{d} \arrow[two heads, "{\widetilde{\phi}}^{(3)}", swap]{ld}\\
			\widetilde{H}_{\ell}/\widetilde{H}_{\ell}^{(3)} \arrow[two heads, "\varpi_{\ell}^{(3)}", swap]{r} & G_{\O}(K\overline{k})(\ell)/G_{\O}(K\overline{k})(\ell)^{(3)}.
		\end{tikzcd}\]
		The exact sequence 
		\[
			1 \longrightarrow F^{(2)}/F^{(3)} \longrightarrow F/ F^{(3)} \longrightarrow F/ F^{(2)} \longrightarrow 1
		\]
		is stem, where $F/F^{(2)} \simeq \Z_{\ell}^{2g}$ and $F^{(2)}/F^{(3)}\simeq \Z_{\ell}^{2g^2-2g}$ (because $F^{(2)}/F^{(3)}$ is generated by images of $[x_i, x_j]$ for $i\neq j$). Because the cup product \eqref{eq:cup-product} for the Demu\v{s}kin group $G_{\O}(K\overline{k})(\ell)$ is non-degenerate, it follows by \eqref{eq:rep-lambda} that the image of $R$ in $F/F^{(3)}$ is isomorphic to $\Z_{\ell}$ which is contained in $F^{(2)}/F^{(3)}$. 
		Recall that $\widetilde{\phi}$ factors through the isomorphism $\iota: F/[F,R] \to \widetilde{H}_{\ell}$, the $(F/[F,R])/(F/[F,R])^{(3)}$ is isomorphic to $\widetilde{H}_{\ell}/\widetilde{H}_{\ell}^{(3)}$, where the former is isomorphic to $F/F^{(3)}$ because $[F,R]\subset F^{(3)}$. So we conclude that $\widetilde{\phi}^{(3)}$ is isomorphic, and hence the image of $R$ in $\widetilde{H}_{\ell}/\widetilde{H}_{\ell}^{(3)}$ is isomorphic to $\Z_{\ell}$. Then because the image of $R$ in $\widetilde{H}_{\ell}$ is isomorphic to $\Z_{\ell}$, we see that $\widetilde{\phi}(R)\cap \widetilde{H}_{\ell}^{(3)}=1$, so \eqref{item:seK-2} implies \eqref{item:seK-1}, and we finish the proof of \eqref{item:stem-ell-K-1}.	
		
		Because \eqref{eq:ext-ell} is a stem extension, in the Hochschild--Serre exact sequence 
		\[
			0 \longrightarrow H^1(G_{\O}(K\overline{k})(\ell), \Z_{\ell}) \longrightarrow H^1(\widetilde{H}_{\ell}, \Z_{\ell}) \longrightarrow H^1(\ker \varpi_{\ell}, \Z_{\ell})\overset{tg_{\varpi_{\ell}}}{\longrightarrow} H^2(G_{\O}(K\overline{k})(\ell), \Z_{\ell}),
		\]
		the second arrow is an isomorphism, so $tg_{\varpi_{\ell}}$ is an injection. On the other hand, because we've shown that there is a surjection $\widetilde{\phi}$ such that $\phi=\varpi_{\ell} \circ \widetilde{\phi}$, we have the diagram
		\[\begin{tikzcd}
			H^1(R, \Z_\ell)^{G_{\O}(K\overline{k})(\ell)} \arrow["tg","\sim"']{dr} & \\
			H^1(\ker \varpi_{\ell}, \Z_{\ell}) \arrow{u} \arrow[hook, "tg_{\varpi_{\ell}}"']{r} & H^2(G_{\O}(K\overline{k})(\ell), \Z_{\ell})
		\end{tikzcd}\]
		where the vertical map is induced by $\widetilde{\phi}|_R$. 
		Moreover, the vertical map is an isomorphism, since $H^1(R, \Z_{\ell})^{G_{\O}(K\overline{k})(\ell)}=\Hom_F(R, \Z_{\ell}) = \Hom(R/[F,R], \Z_{\ell})$ and $\iota$ is an isomorphism.
		Then from the diagram, we see that $tg_{\varpi_{\ell}}$ is an isomorphism and the equality $(tg^{-1} \varphi)( \lambda) = (tg_{\varpi_{\ell}}^{-1} \varphi)(\widetilde{\phi}(\lambda))$ holds, which proves the statement \eqref{item:stem-ell-K-2} in the lemma.
	\end{proof}

	Lemma~\ref{lem:stem-ell-K} shows that, after taking the quotient map $\widetilde{\phi}$, the extension \eqref{eq:ext-ell} carries exactly the information of the symplectic form \eqref{eq:cup-product}. So we formulate this symplectic form using only the structure of \eqref{eq:ext-ell} in the following corollary. Because after taking the quotient map $\widetilde{\phi}$ on $F$, the corollary naturally follows by how the minimal presentation \eqref{eq:pre-H} determines the symplectic form \eqref{eq:cup-GC}, so by abuse of notation, we don't distinguish the symbols used for $F$ and for $\widetilde{H}_{\ell}$.

	\begin{corollary}\label{cor:symplectic-stem}
		Consider the stem extension \eqref{eq:ext-ell}, and let $\lambda$ be a generator of $\ker \varpi_{\ell}$. Define a trace map $\tr_{\lambda}$ as
		\begin{eqnarray*}
			\tr_{\lambda}: H^2(G_{\O}(K\overline{k})(\ell), \Z_{\ell}) & \longrightarrow& \Z_{\ell} \\
			\varphi & \longmapsto & (tg_{\varpi_{\ell}}^{-1} \varphi) (\lambda),
		\end{eqnarray*}
		where $tg_{\varpi_{\ell}}$ is the transgression map in Lemma~\ref{lem:stem-ell-K}\eqref{item:stem-ell-K-2}.
		 Assume that $\{x_1, \ldots, x_{2g}\}$ is a minimal set of generators of $\widetilde{H}_{\ell}$. Then $\lambda$ can be uniquely written in the form of
		\[
			\lambda= \prod_{1\leq i < j \leq 2g}[x_i, x_j]^{a_{ij}} \cdot \lambda', \quad \lambda' \in \widetilde{H}_{\ell}^{(3)}, \, a_{ij} \in \Z_{\ell}.
		\]
		Note that $\{ \varpi_{\ell}(x_1), \ldots, \varpi_{\ell}(x_{2g})\}$ is a minimal set of generators of $G_{\O}(K\overline{k})(\ell)$.
		Let $\chi_1, \ldots, \chi_{2g}$ be the corresponding $\Z_\ell$-basis of $H^1(G_{\O}(K\overline{k}), \Z_{\ell})$ such that $\chi_i(\varpi_{\ell}(x_j))=1$ if $i=j$ and 0 otherwise. Then the symplectic form 
		\begin{equation}\label{eq:symp-stem}
		\begin{tikzcd}
			H^1(G_{\O}(K\overline{k}), \Z_{\ell}) \times H^1(G_{\O}(K\overline{k}), \Z_{\ell}) \arrow["\cup"]{r} & H^2(G_{\O}(K\overline{k}), \Z_\ell) \arrow["\tr_{\lambda}", "\sim"']{r} & \Z_{\ell}
		\end{tikzcd}
		\end{equation}
		with respect to the basis $\chi_1, \ldots, \chi_{2g}$ is given by the matrix $M_{\lambda}=[m_{ij}]$, where 
		\[
			m_{ij}= \tr_{\lambda}(\chi_i \cup \chi_j) = \begin{cases}
				-a_{ij} & \text{ if } i< j \\
				a_{ji} & \text{ if } i>j \\
				0 & \text{ if } i=j.
			\end{cases}
		\]
		For a positive integer $n$, the symplectic form above with all coefficients replaced with $\Z/\ell^n\Z$ is given by the matrix $M_{\lambda}$ mod $\ell^n$.
	\end{corollary}
	
	\begin{proof}
		The corollary follows by the argument at the beginning of the section from \eqref{eq:pre-H} to \eqref{eq:matrix-M} and Lemma~\ref{lem:stem-ell-K}.
	\end{proof}

\subsection{The invariant $\omega_{L/K}$ determines the image of Weil pairing} \label{ss:weil-pairing}
	
	Let $X$ be the smooth projective curve over $k$ whose function field is $K$, and let $X_{\overline{k}}$ denote the base change of $X$ to $\overline{k}$. 
	Let $\Jac_X$ denote the Jacobian variety of $X_{\overline{k}}$. Then $\Jac_X$ has an isomorphic principal polarisation, so the Weil pairing associated to $X_{\overline{k}}$ is a bilinear pairing
	\[
		\Jac_X[\ell^n] \times \Jac_X[\ell^n] \longrightarrow \mu_{\ell^n},
	\]
	which is non-degenerate and alternating. By \cite[Chap.V, \S2, Rmk.2.4(f)]{Milne-EC} and \cite[Dualit\'e \S3]{SGA4.5}, the Weil pairing can be identified with the cup product of \'etale cohomology followed by the trace map
	\begin{equation}\label{eq:cup-EC}
	\begin{tikzcd}
		H^1_{\et}(X_{\overline{k}}, \mu_{\ell^n}) \times H^1_{\et}(X_{\overline{k}}, \mu_{\ell^n}) \arrow["\cup"]{r} & H_{\et}^2(X_{\overline{k}}, \mu_{\ell^n}^{\otimes 2}) \arrow["\tr_{X,\ell^n}", "\sim"']{r} &\mu_{\ell^n}.
	\end{tikzcd}
	\end{equation}
	Since $\mu_{\ell^n}$ is a trivial $G_{\O}(K\overline{k})$-module, we fix an isomorphism $\mu_{\ell^n} \to \Z/\ell^n\Z$ that is compatible among all $n$; in other words, we pick a generator $\zeta$ of $\Z_{\ell}(1)$ and then identify $\Z_{\ell}$ with $\Z_{\ell}(1)$ by sending $1$ to $\zeta$. Then it follows by the functorial isomorphism between Galois cohomology and \'{e}tale cohomology that, for a generator $\lambda$ of $\ker \varpi_{\ell}$ in \eqref{eq:ext-ell}, the symplectic form defined by \eqref{eq:cup-EC} is equivalent to the one defined by the mod-$\ell^n$ analogue of \eqref{eq:symp-stem} followed by an automorphism of $\Z/\ell^n\Z$.

	\begin{lemma}\label{lem:Weil-pairing-lambda}
	Identify $\Z_{\ell}$ with $\Z_{\ell}(1)$ by $1 \mapsto \zeta$.
	\begin{enumerate}
		\item \label{item:Weil-lambda-1}
		Let $\lambda$ be a generator of $\ker \varpi_{\ell}$ in \eqref{eq:ext-ell}, and let $M_{\lambda}$ and $\tr_{\lambda}$ be as defined in Corollary~\ref{cor:symplectic-stem}. The Weil pairing on the $\ell$-adic Tate module of $X_{\overline{k}}$ 
		\[
			H^1(\Gp(K\overline{k}), \Z_\ell(1)) \times H^1(\Gp(K\overline{k}), \Z_\ell(1)) \longrightarrow \Z_\ell(1)
		\]
		is equivalent to the composition of the symplectic form determined by the matrix $M_{\lambda}$ and an isomorphism $\tr_X \circ \tr_{\lambda}^{-1}\in \Aut(\Z_{\ell})$, where $\tr_X:=\varprojlim_{n} \tr_{X,\ell^n}$.

		\item \label{item:Weil-lambda-2}
		Let $\underline{g}:=(g_1, \ldots, g_m)$ be a system of inertia generators for $K$ such that $I(\underline{g})$ is the generator $\zeta^{-|\Gamma|}$ of $\Z_{\ell}(1)$. Then let $\lambda$ be the image of $g_1\cdots g_m$ in $\ker \varpi_{\ell}$ as in the pro-$\ell$ analogue of Proposition~\ref{prop:etale-lift}\eqref{item:etale-lift-2}.  Then, $\tr_{\lambda} = \tr_X$ as trace maps from $H^2(G_{\O}(K\overline{k}), \Z_\ell(1)) \to \Z_{\ell}$.
	\end{enumerate}
	\end{lemma}
	
	\begin{proof}
		The statement \eqref{item:Weil-lambda-1} follows immediately by taking the inverse limit of \eqref{eq:cup-EC} over all $n$.
		The statement~\eqref{item:Weil-lambda-2} says that the right block of the following diagram commutes
		\[\begin{tikzcd}
			H^1(\ker\varpi_{\ell}, \Z_{\ell}(1)) \arrow["tg_{\varpi_\ell}", "\sim"']{r} & H^2(G_{\O}(K\overline{k})(\ell), \Z_{\ell}(1)) \arrow["\tr", "\sim"']{r} & \Z_\ell \arrow[equal]{d}\\
			H^1(\ker\varpi_{\ell}, \Z_{\ell}) \arrow["tg_{\varpi_\ell}", "\sim"']{r} \arrow["\cup \zeta", "\sim"']{u} & H^2(G_{\O}(K\overline{k})(\ell), \Z_\ell) \arrow["\tr_{\lambda}", "\sim"']{r} \arrow["\cup \zeta", "\sim"']{u} & \Z_{\ell} ,
		\end{tikzcd}\]
		where the trace map $\tr$ in the first row is obtained by taking the $(-1)$-Tate twist of the composition map 
		\[\begin{tikzcd}
			H^2(G_{\O}(K\overline{k})(\ell), \Z_\ell(2)) \arrow["inf", "\sim"']{r} & H^2(G_{\O}(K\overline{k}), \Z_\ell(2)) \arrow["\sim"']{r} & H^2_{\et}(X_{\overline{k}}, \Z_{\ell}(2)) \arrow["\tr_X", "\sim"']{r} & \Z_\ell(1).
		\end{tikzcd}\]
		By taking the inverse limit of the pro-$\ell$ analogue of Lemma~\ref{lem:tr-Gamma}, the homomorphism $[\lambda \mapsto I(\underline{g})] \in H^1(\ker \varpi_{\ell}, \Z_{\ell}(1))$ (which is defined as the homomorphism $\ker \varpi_{\ell} \to \Z_{\ell}(1)$ that sends $\lambda$ to $I(\underline{g})$) is mapped to $-|\Gamma|$ in the upper-right entry of the above diagram. By our definition of $\lambda$ and $\underline{g}$ in the lemma, the preimage of $[\lambda \mapsto I(\underline{g})]$ in the lower-left entry is the homomorphism $[\lambda \mapsto -|\Gamma|]$, whose image in the lower-middle entry has $\tr_{\lambda}$ exactly $-|\Gamma|$. So we finish the proof.
	\end{proof}

	Next, we will show that the invariant $\omega_{L/K, \ell}$ associated to $L/K$ carries sufficient information to determine the image of Weil pairing.

	\begin{definition}[The image of the Weil pairing]\label{def:image-Weil}
		Assume that $k$ contains $\mu_{\ell^n}$. Let $L/K$ be a subextension of $\Ks(\ell)/K$ such that $L$ is Galois over $k(t)$. 
		Then, by the isomorphism between \'etale cohomology and Galois cohomology, the Weil pairing defines a symplectic pairing 
		\[
			f_{K, \ell^n}: H^1(\Gp(K), \mu_{\ell^n}) \times H^1(\Gp(K), \mu_{\ell^n}) \to \mu_{\ell^n}.
		\]
		Then the restriction of $f_{K,\ell^n}$ gives a pairing 
	\[
		f_{L/K,\ell^n}: H^1(\Gal(L/K), \mu_{\ell^n}) \times H^1(\Gal(L/K), \mu_{\ell^n}) \longrightarrow \mu_{\ell^n},
	\]
	where the Frobenius $\Frob_k$ acts trivially on each term as it acts trivially on $\Gal(L/K)$. We call $f_{L/K,\ell^n}$ \emph{the image of the Weil pairing} associated to $L/K$.
	\end{definition}

	For the rest of this section, we assume that $L$ and $n$ are as described in Definition~\ref{def:image-Weil}.
	Consider compatible $\ell$-Schur coverings of $\Gal(\Kp(\ell)/\overline{k}(t))$ and $\Gal(L\overline{k}/\overline{k}(t))$ (such compatible stem extensions exist by Lemma~\ref{lem:comp-schur}): 
	\begin{equation}\label{eq:compatible-schur-ell}
	\begin{tikzcd}
		1 \arrow{r} & \ker \pi_{\ell} \arrow{r} \arrow{d} & \widetilde{G}_{\ell} \arrow["\pi_{\ell}"]{r} \arrow{d} & \Gal(\Kp(\ell)/\overline{k}(t)) \arrow{r}\arrow[two heads]{d} & 1 \\
		1 \arrow{r} & \ker \pi'_{\ell} \arrow{r} & S_{\ell} \arrow["\pi'_{\ell}"]{r} & \Gal(L\overline{k}/ \overline{k}(t)) \arrow{r} & 1,
	\end{tikzcd}
	\end{equation}
	where the rows are $\ell$-Schur coverings, the last vertical map is the quotient of Galois groups, and the first vertical map is the composition of
	\[\begin{tikzcd}
		\ker \pi_{\ell} \arrow["d^{-1}_{\pi_\ell}", "\sim"']{r} & H_2(\Gal(\Kp(\ell)/\overline{k}(t)), \Z)(\ell) \arrow["\coinf"]{r} & H_2(\Gal(L\overline{k}/\overline{k}(t)), \Z)(\ell) \arrow["d_{\pi'_{\ell}}", "\sim"']{r} & \ker \pi'_{\ell}.
	\end{tikzcd}\]
	By restricting $\pi_{\ell}$ and $\pi'_{\ell}$ to extensions of $G_{\O}(K\overline{k})(\ell)$ and $\Gal(L\overline{k}/K\overline{k})\simeq \Gal(L/K)$ respectively, we obtain the following compatible extensions
	\begin{equation}\label{eq:compatible-schur-L}
	\begin{tikzcd}
		1 \arrow{r} & \ker \varpi_{\ell} \arrow{r} \arrow{d} & \widetilde{H}_{\ell} \arrow["\varpi_{\ell}"]{r} \arrow["\alpha"]{d} & G_{\O}(K\overline{k})(\ell) \arrow{r}\arrow[two heads]{d} & 1 \\
		1 \arrow{r} & \ker \varpi'_{\ell} \arrow{r} & T \arrow["\varpi_{\ell}'"]{r} & \Gal(L/ K) \arrow{r} & 1,
	\end{tikzcd}
	\end{equation}
	where $\widetilde{H}_{\ell}:=\pi_{\ell}^{-1}(G_{\O}(K\overline{k})(\ell))$, $\varpi_{\ell}:=\pi_{\ell}|_{\widetilde{H}_{\ell}}$, $T:={\pi'_{\ell}}^{-1} (\Gal(L\overline{k}/K\overline{k}))$ and $\varpi'_{\ell}:=\pi'_{\ell}|_T$. 
	These rows are $\Gamma$-equivariant when we choose compatible splittings of $\Gal(\Kp(\ell)/\overline{k}(t)) \to \Gal(K\overline{k}/\overline{k}(t))$ and $\Gal(L\overline{k}/\overline{k}(t)) \to \Gal(K\overline{k}/\overline{k}(t))$.

	\begin{proposition}\label{prop:omega-weil}
		Use the notation above. Let $\zeta$ be a generator of $\Z_{\ell}(1)$ and $\zeta_{n}$ the image of $\zeta$ in $\mu_{\ell^n}$.
		Let $\{y_1, \ldots, y_r\}$ be a set of generators of $\Gal(L/K)$ such that $\Gal(L/K)^{ab}$ is the direct products of the images of $\langle y_i \rangle$ for all $i$. Then there exists a corresponding basis $\{ y_1^{\vee}, \ldots, y_r^{\vee}\}$ of $H^1(\Gal(L/K), \mu_{\ell^n})$ such that $y_i^{\vee}(y_j) =0$ for $i\neq j$.
		 Also, $\{y_1, \ldots, y_r\}$ lifts to a minimal set of generators $\{\widetilde{y}_1, \ldots, \widetilde{y}_r\}$ of $T$ in \eqref{eq:compatible-schur-L}.
		
		Let $\lambda_{L}$ denote the image of $\zeta^{-|\Gamma|}$ under the following composition map
		\[\begin{tikzcd}
			\Z_{\ell}(1) \arrow["\omega_{L/K, \ell}"]{r} & H_2(\Gal(L\overline{k}/\overline{k}(t)), \Z)(\ell) \arrow["d_{\pi'_{\ell}}^{-1}"]{r} & \ker \pi'_{\ell} = \ker \varpi'_{\ell},
		\end{tikzcd}\]
		where $\omega_{L/K, \ell}$ is the invariant defined in Remark~\ref{rmk:inv}\eqref{item:rmk-inv-3} and $d_{\pi'_{\ell}}$ is the differential map defined by $\pi'_{\ell}$. Then $\lambda_L$ can be expressed in the form of 
		\[
			\lambda_L = \prod_{1 \leq i < j \leq r} [\widetilde{y}_i, \widetilde{y}_j]^{b_{ij}}  \cdot \lambda'_L, \quad \lambda'_L \in T^{(3)}, \, b_{ij} \in \Z_{\ell}.
		\]
		With respect to the basis $y_1^{\vee}, \ldots, y_r^{\vee}$, the image of the Weil pairing $f_{L/K, \ell^n}$ is given by 
		\[
			f_{L/K, \ell^n} (y_i^{\vee}, y_j^{\vee}) = \begin{cases}
				\left(y_i^{\vee}(y_i) \cdot y_j^{\vee}(y_j)\right)^{-b_{ij}} & \text{ if } i< j \\
				\left(y_i^{\vee}(y_i)  \cdot y_j^{\vee}(y_j)\right)^{b_{ji}} & \text{ if } i>j \\
				0 & \text{ if } i=j.
			\end{cases}
		\]
	\end{proposition}
	
	\begin{proof}
		There exist elements $t_1, \cdots, t_r$ of $\Gal(L/K)^{ab}$ such that $\Gal(L/K)^{ab}=\langle t_1 \rangle \times \cdots \times \langle t_r \rangle$. So there exists a basis $\{ y_1^{\vee}, \ldots, y_r^{\vee}\}$ of $H^1(\Gal(L/K), \mu_{\ell^n})=H^1(\Gal(L/K)^{ab}, \mu_{\ell^n})$ such that $y_i^{\vee}(t_j)=0$ for $i\neq j$. We pick $y_i \in \Gal(L/K)$ to be a lift of $t_i$. By the Burnside basis theorem, $y_1, \ldots, y_r$ form a minimal set of generators of $\Gal(L/K)$. Since $T$ is a stem extension of $\Gal(L/K)$, we have $T^{ab}\simeq \Gal(L/K)^{ab}$ and hence $\{\widetilde{y}_1, \ldots, \widetilde{y}_r\}$ is a minimal set of generators of $T$. 
		
		Let $\lambda$ be as defined in Lemma~\ref{lem:Weil-pairing-lambda}\eqref{item:Weil-lambda-2}. So $\lambda_L$ is $\alpha(\lambda)$, where $\alpha$ is the middle vertical arrow in \eqref{eq:compatible-schur-L}. We let $M_{\lambda}$, $x_i$, $\chi_i$, $m_{ij}$ and $a_{ij}$  be as defined in Corollary~\ref{cor:symplectic-stem} for $1\leq i,j \leq 2g$, then let $\chi_i[n]$ denote the composition of $\chi_i$ and $\Z_{\ell} \to \mu_{\ell^n}$ (defined by $1\mapsto \zeta_n$).
		Let $D=[d_{ij}]$ be a $2g\times r$ $\Z_\ell$-entry matrix such that 
		\[
			\alpha(x_i) \equiv \prod_{j=1}^{r} (\widetilde{y}_j)^{d_{ij}} \mod T^{(2)},
		\]
		and let $E=[e_{ij}]$ be a $2g\times r$ $\Z_{\ell}$-entry matrix such that 
		\[
			y_j^{\vee} = \sum_{i=1}^{2g} e_{ij}\chi_i[n].
		\]
		We compute $y_j^{\vee}(\alpha(x_i))$ by viewing $y_j^{\vee}$ as a homomorphism from $\widetilde{H}_{\ell}$ to $\mu_{\ell^n}$ that factors through $\Gal(L/K)$. On one hand, we have 
		\[
			y_j^{\vee}(\alpha(x_i))=y_j^{\vee}\left(\prod_{s=1}^r (\widetilde{y}_s)^{d_{is}} \right)=y_j^{\vee}(y_j)^{d_{ij}};
		\]
		and on the other hand,
		\[
			y_j^{\vee}(\alpha(x_i))=y_j^{\vee}(x_i)=\prod_{s=1}^{2g} \chi_s[n](x_i)^{e_{sj}}=(\zeta_n)^{e_{ij}}.
		\]
		So we see that $E=D\Delta$, where $\Delta$ is the $r\times r$ diagonal matrix whose $(i,i)$-entry is $\log_{\zeta_n} y_i^{\vee}(y_i)$. By $\lambda_L=\alpha(\lambda)$ and Corollary~\ref{cor:symplectic-stem}, we have
		\begin{eqnarray*}
			\lambda_L &=& \prod_{1 \leq i < j \leq 2g} [\alpha(x_i), \alpha(x_j)]^{a_{ij}} \cdot \lambda' \quad \quad  \text{ for some }\lambda' \in T^{(3)} \\
			&=& \prod_{1 \leq i < j \leq 2g} \left[ \prod_{s=1}^{r} (\widetilde{y}_s)^{d_{is}},  \prod_{t=1}^{r} (\widetilde{y}_t)^{d_{jt}}\right]^{a_{ij}} \cdot \lambda''  \quad \quad  \text{ for some }\lambda'' \in T^{(3)} \\
			&=& \prod_{1 \leq s < t \leq r} [\widetilde{y}_s, \widetilde{y}_t]^{-\sum_{1 \leq i,j \leq 2g} d_{is} d_{jt} m_{ij}} \cdot \lambda''
		\end{eqnarray*}
		So we see that $-b_{st}$ is the $(s,t)$-entry of $D^TM_{\lambda}D$ for $1 \leq s < t \leq r$.
		Then by Corollary~\ref{cor:symplectic-stem} and Lemma~\ref{lem:Weil-pairing-lambda}\eqref{item:Weil-lambda-2}
		\[
			f_{L/K, \ell^n}\left(y_i^{\vee}, y_j^{\vee}\right) = f_{L/K, \ell^n} \left( \sum_{s=1}^{2g} e_{si} \chi_s[n],  \sum_{t=1}^{2g} e_{tj}\chi_t[n]\right)= (\zeta_{n})^{(i,j)\text{-entry of } E^TM_{\lambda}E},
		\]
		which proves the proposition.
	\end{proof}

\section{The function field moments as $q \to \infty$}\label{sect:ff-moment}

	In this section, we study the relationship between the $\omega$-invariant $\omega_{L/K}$ and the EVW-W lifting invariant. Building on the strategy of the work \cite{LWZB} of Wood, Zureick-Brown and the author, we give the proof of Theorem~\ref{thm:functionfield} by counting rational points on Hurwitz schemes with prescribed $\omega$-invaraint.

\subsection{Revisit: Hurwitz schemes and a lifting invariant of Ellenberg--Venkatesh--Westerland and Wood} \label{ss:EVWW}
	In \cite{Wood-lifting}, Wood defined and proved basic properties of a lifting invariant of covers of $\PP^1$ over an algebraically closed field, which was introduced by Ellenberg, Venkatesh, and Westerland. 
	Let $G$ be a finite group and $c$ a subset of $G$ closed under conjugation by elements of $G$ and closed under invertible powering. Let $c/G$ denote the set of conjugacy classes of elements in $c$. 
		
	Let $U(G,c)$ be the group defined by presentation with generators $[g]$ for $g\in c$, and relations $[x][y][x]^{-1}=[xyx^{-1}]$ for $x,y \in c$. Then there is a natural surjection $U(G,c) \to G$ mapping $[g]\mapsto g$ and a natural surjection $U(G,c) \to \Z^{c/G}$ mapping $[g]$ to a generator for the conjugacy class of $g$. These two surjections induce the same map $U(G,c) \to G^{ab}$ to the abelianization of $G$, so we obtain a homomorphism $U(G,c) \to G \times_{G^{ab}} \Z^{c/G}$. This homomorphism is actually a surjection, and \cite[Thm.2.5]{Wood-lifting} proves that $U(G,c)$ is isomorphic to the fiber product $S_c \times_{G^{ab}} \Z^{c/G}$, where $S_c$ is a reduced Schur covering for $G$ and $c$ defined as follows. 
	A \emph{reduced Schur covering} $S_c \to G$ is defined to be the quotient of a Schur covering $\phi: S \to G$ by the normal subgroup generated by the set of commutators
	\[
		\{[\hat{x}, \hat{y}] \mid \hat{x}, \hat{y}\in S, \, \phi(\hat{x})\in c, \, \text{and } [\phi(\hat{x}), \phi(\hat{y})]=1 \}.
	\]
	Via the isomorphic differential map $H_2(G, \Z) \to \ker \phi$, the kernel of a reduced Schur covering is isomorphic to a quotient of $H_2(G,\Z)$, which we denote by $H_2(G,c)$; and one can prove by definition that $H_2(G,c)$ does not depend on the choice of the Schur covering that we start with. Thus, the kernel of the natural quotient map $U(G,c)\to G$ is isomorphic to $H_2(G,c) \times \Z^{c/G}$, and we denote this kernel by $K(G,c)$. 
	
	 Let $p$ be a prime not dividing $|G|$.
	 There is an action of $\hat{\Z}_{(p)'}^{\times}$ on the set of elements of $K(G,c)$. Let $S_c\to G$ be a reduced Schur covering for $G$ and $c$. Then elements of $K(G,c)$ can be written in the form $(g, \underline{m})$ for $g\in \ker (S_c \to G)$ and $\underline{m}\in \Z^{c/G}$. For each conjugacy class $\gamma \in c/G$, we pick an element $x_{\gamma}$ in $\gamma$ and a lift $\widehat{x_{\gamma}}$ of $x_{\gamma}$ in $S_c$. Then for each $\alpha \in  \hat{\Z}_{(p)'}^{\times}$ and $\gamma \in c/G$, we define
	$w_{\alpha}(\gamma)=\widehat{x_{\gamma}}^{-\alpha} \widehat{x_{\gamma}^{\alpha}} \in \ker(S_c\to G)$, where $\widehat{x_{\gamma}^{\alpha}}$ is the picked lift of the element in the conjugacy class of $\gamma^{\alpha}$.
	Then we define a group homomorphism
	$W_{\alpha} : \Z^{c/G} \to \ker(S_c \to G)$
	by sending the generators for the conjugacy class $\gamma$ to $w_{\alpha}(\gamma)$.
	 It was shown in \cite[\S4]{Wood-lifting} that $W_{\alpha}$ does not depend on the choice of the lifts $\widehat{x_{\gamma}}$. Then, for each $\alpha \in  \hat{\Z}_{(p)'}^{\times}$, we define an action of $\alpha$ on the set $K(G,c)$ by
	 \begin{equation}\label{eq:alpha-K}
	 	\alpha \ast (g, \underline{m}) = (g^{\alpha}W_{\alpha}(\underline{m}), \underline{m}^{\alpha}),
	 \end{equation}
	where $g^{\alpha}$ is the $\alpha$th power of $g$ and $\underline{m}^{\alpha}$ is a rearranging of the coordinates of $\underline{m}$ by sending the coordinate of $\gamma$ to the coordinate of $\gamma^{\alpha}$. This action of $\alpha$ is not a group homomorphism.

	Let $\overline{k}$ be an algebraically closed field of characteristic $p$ such that $p\nmid |G|$ if $p\neq 0$ (here we include the case $p=0$). A branched $G$-cover $X$ of $\PP^1$ over $\overline{k}$ that is ramified at $\{P_1, \cdots, P_m\}$ defines a surjection 
	\[
		\varphi: \pi_1^{\et} (\PP^1_{\overline{k}} - \{P_1, \ldots, P_m\})_{(p)'} \longrightarrow G.
	\] 
	Assume that all the inertia groups of the cover are generated by elements in $c$. We let $\underline{g}=(g_1,\ldots, g_m)$ be a system of inertia generators for $\pi_1^{\et} (\PP^1_{\overline{k}} - \{P_1, \ldots, P_m\})_{(p)'}$ and let $I(\underline{g})$ is the corresponding generator of $\widehat{\Z}(1)_{(p)'}$. Then \emph{EVW-W lifting invariant} associated to this cover $X/\PP^1_{\overline{k}}$ is defined to be the $\hat{\Z}^{\times}_{(p)'}$-equivariant map of sets
	\begin{eqnarray*}
		\frakz: \hat{\Z}(1)^{\times}_{(p)'} &\longrightarrow& K(G,c) \\
		I(\underline{g}) &\longmapsto& [\varphi(g_1)]\cdots [\varphi(g_m)].
	\end{eqnarray*}

	There is a Hurwitz scheme $\HHur^n_{G,c}$ defined over $\Z[|G|^{-1}]$, whose points represent tame Galois covers $X/\PP^1$ with $n$ branch points (i.e., the degree of each geometric fiber of $D \to \Spec(\Z[|G|^{-1}])$ is $n$, where $D\subset \PP^1$ is the branch loci of the cover $X/\PP^1$) together with a choice of identification $\Gal(X/\PP^1)\simeq G$ and a choice of point over infinity, such that all inertia subgroups are generated by elements in $c$ and the place at $\infty$ is unramified (see \cite[\S11]{LWZB} for more details about this Hurwitz scheme). For each geometric point $\overline{s}$ on $\HHur^n_{G,c}$ whose residue field has characteristic $p$, we can attach to it the lifting invariant $\frakz_{\overline{s}}$ as defined above. The lifting invariant is actually a component invariant: for any algebraically closed field $k$, the $k$-points of $(\HHur^n_{G,c})_k$ in the same component have the same lifting invariant (\cite[Cor.12.2]{LWZB}). 
	Moreover, when $M$ is sufficiently large, there is a bijection between the set
	$\calK:=\{(g,\underline{m}) \in K(G,c) \mid \text{coordinates of $\underline{m}$ are at least $M$ and sum up to $n$}\}$ and
	the components of $(\HHur^n_{G,c})_{\overline{\F}_q}$ (and similarly, the components of $(\HHur^n_{G,c})_{\C}$) whose lifting invariant has image in $\calK$.
	In particular, by \cite[Thm.12.1]{LWZB}, if $\overline{s}$ is a $\overline{\F}_q$-point of $(\HHur_{G,c}^n)_{\overline{\F}_q}$, then for a $\zeta \in \hat\Z(1)^{\times}_{(q)'}$, we have 
	\begin{equation}\label{eq:Frob-inv}
	\frakz_{\Frob(\overline{s})}(\zeta)=q^{-1} \ast \frakz_{\overline{s}}(\zeta),
	\end{equation}
	where $\Frob$ is the Frobenius map $\Frob_{(\HHur_{G,c}^n)_{\F_q}}$ and the $\ast$ operator is as defined in \eqref{eq:alpha-K}.  
	
	In order to compute the number of $\F_q$-points of $\HHur_{G,c}^n$ acting on $(\HHur_{G,c}^n)_{\overline{\F}_q}$, by the Grothendieck--Lefschetz trace formula, it is crucial to understand the number of $\Frob_{(\HHur_{G,c}^n)_{\F_q}}$-fixed components of $(\HHur_{G,c}^n)_{\overline{\F}_q}$, which we denote by $\pi_{G,c}(q,n)$. 
	For integers $n$ and $M$, we write $\Z^{c/G}_{\equiv q, n, \geq M}$ for the subset of $\Z^{c/G}$ consisting of elements satisfying: 1) the coordinates are constant on each set of conjugacy classes of $c/G$ that can be obtained from one another by taking $q$th powers, 2) the coordinates sum up to $n$, and 3) each coordinate is at least $M$. Then there is a homomorphism $\Z^{c/G}_{\equiv q, n, \geq M} \to G^{ab}$ that is the restriction of the natural map $\Z^{c/G} \to G^{ab}$. 
	Then by \cite[Prop.12.7]{LWZB}, the number of $\Frob_{(\HHur^n_{G,c})_{\F_q}}$-fixed components of $(\HHur^n_{G,c})_{\overline{\F}_q}$ has the following estimation
	\begin{equation}\label{eq:pi-b}
		\pi_{G,c}(q,n)=b(G,c,q,n)+O_G(n^{d_{G,c}(q)-2}),
	\end{equation}
	where $d_{G,c}(q)$ is the number of orbits of $q$th powering on the conjugacy classes in $c/G$ and 
	\[
		b(G,c,q,n):= \sum_{h \in \ker(S_c \to G)} \# \left\{ \underline{m} \in \ker \left( \Z^{c/G}_{\equiv q, n, \geq 0} \to G^{ab}\right) \, \biggr\rvert \, W_{q^{-1}}(\underline{m})^q=h^{q-1}\right\}.
	\] 
	The definition of $b(G,c,q,n)$ above appears differently from but agrees with the one in \cite{LWZB} (see the proof of \cite[Prop.12.7]{LWZB}).

\subsection{Relationship between the EVW-W lifting invariant and the invariant $\omega_{L/K}$} \label{ss:EVW-omega}

	In this subsection, we will prove in Proposition~\ref{prop:EVW-omega} that the $\omega$-invariant $\omega_{L/K}$ equals the prime-to-$|\Gamma|$-torsion part the EVW-W lifting invariant.
	Throughout this subsection, we consider the EVW-W lifting invariants when $G$ is a semidirect product $H \rtimes \Gamma$ with $\gcd(|H|, |\Gamma|)=1$ and $c$ is the set consisting of elements of $G\backslash \{1\}$ that have the same order as their image in $\Gamma$.

	\begin{lemma}\label{lem:reduced-prime}
		Let $\Gamma$ be a finite group and $H$ a $|\Gamma|'$-$\Gamma$-group. Suppose that there is a stem extension of groups
		\[
			1 \longrightarrow A \longrightarrow \widetilde{H} \rtimes \Gamma \overset{\pi}{\longrightarrow} H \rtimes \Gamma \longrightarrow 1,
		\]
		such that $A$ has order prime to $|\Gamma|$. Let $x$ and $y$ be two elements of $H \rtimes \Gamma$ such that $[x,y]=1$ and $x$ has order equal to the order of its image in $\Gamma$. Then $[\widetilde{x}, \widetilde{y}]=1$ for any lifts $\widetilde{x} \in \pi^{-1}(x)$ and $\widetilde{y} \in \pi^{-1}(y)$.
	\end{lemma}
	
	\begin{proof}
		We let $\gamma$ denote the image of $\widetilde{x}$ in $\Gamma$. By the Schur--Zassenhaus theorem, we can pick a homomorphic splitting $s:\Gamma \to \widetilde{H}\rtimes \Gamma$ such that $s(\gamma)=\widetilde{x}$. We can redefine the group $\widetilde{H}\rtimes \Gamma$ as the semidirect product $\widetilde{H} \rtimes s(\Gamma)$, and then using the notation of elements in semidirect product, we write $\widetilde{x}$ as $(1,\gamma)$ and write $\widetilde{y}$ as $(b, \sigma)$ for $b\in \widetilde{H}$ and $\sigma \in \Gamma$. So $[\widetilde{x}, \widetilde{y}]=\widetilde{x}^{-1}\widetilde{y}^{-1}\widetilde{x}\widetilde{y}=(\gamma^{-1} \circ\sigma^{-1}(b^{-1}\gamma(b)), [\gamma, \sigma])$. Since $[x,y]=1$, we have that $[\gamma,\sigma]=1$ and $b^{-1}\gamma(b) \in \ker \pi$, and it's left to show that $b^{-1}\gamma(b)=1$. 
		
		We consider the action of the subgroup $\langle\gamma\rangle$ of $\Gamma$ on the exact sequence 
		\[
			1 \longrightarrow A \longrightarrow \widetilde{H}\longrightarrow H \longrightarrow 1.
		\]
		For a $\langle \gamma \rangle$-group $G$, we denote $Y_\gamma(G):=\left\{g^{-1}\gamma(g) \mid g\in G \right\}$. It follows by \cite[Lem.3.5 and Lem.3.6]{LWZB} that $|A|=|A^{\langle \gamma \rangle}||Y_{\gamma}(A)|$, $|Y_{\gamma}(A)| |Y_{\gamma}(H)| = |Y_{\gamma}(\widetilde{H})|$, and the elements of $Y_{\gamma}(\widetilde{H})$ are equidistributed in $Y_{\gamma}(H)$. Thus, because $\Gamma$ acts trivially on $A$, we see that $A^{\langle \gamma \rangle}=A$, so $Y_{\gamma}(A)=1$ and therefore $Y_{\gamma}(\widetilde{H}) \cap A=1$. Then it follows that if $b^{-1}\gamma(b) \in \ker \pi$ for some $b\in \widetilde{H}$, then $b^{-1}\gamma(b)=1$, which finishes the proof of the lemma.
	\end{proof}

	\begin{lemma}\label{lem:decom-Schur}
		Let $p$ be a prime, $\Gamma$ a finite group, and $H$ a finite admissible $\Gamma$-group such that $p\nmid |H||\Gamma|$. Let $c_{\Gamma}$ be the set $\Gamma\backslash\{1\}$. Let $G:=H\rtimes \Gamma$ and $c_G$ the set of elements of $G\backslash\{1\}$ that have the same order as their image in $\Gamma$. 
		\begin{enumerate}
			\item\label{item:d-S-H} 
				The order of the kernel of the quotient map $H_2(G,\Z) \to H_2(G,c_G)$ has order supported on the set of prime divisors of $|\Gamma|$.
			
			\item\label{item:d-S-Gamma} 
				The quotient map $G \to \Gamma$ induces a surjection $H_2(G, c_G) \to H_2(\Gamma, c_{\Gamma})$, and the kernel of this surjection has order prime to $|\Gamma|$.
		\end{enumerate}
		In particular, combining the two statements above, we have 
		\begin{equation}\label{eq:d-S}
			H_2(G,c_G)\simeq H_2(G,\Z)_{(|\Gamma|)'} \oplus H_2(\Gamma, c_{\Gamma}).
		\end{equation} 
	\end{lemma}
	
	\begin{proof}
		Let $\pi: S \to G$ be a Schur covering of $G$, and $\pi_{c_G}:S_{c_G}\to G$ its reduced Schur covering associated $c_G$. By taking quotient of $S$ modulo the $|\Gamma|$-part of $\ker \pi$, we have a stem extension
		\begin{equation}\label{eq:star}
			1 \longrightarrow (\ker\pi)_{(|\Gamma|)'} \longrightarrow S' \longrightarrow G \longrightarrow 1.
		\end{equation}
		Then by Lemma~\ref{lem:reduced-prime}, the image of $\ker(S\to S_{c_G})$ in $S'$ is trivial, so $\pi_{c_G}$ factors through $S'$, which proves the statement~\eqref{item:d-S-H}.
		
		By the proof of \cite[Lem.12.10]{LWZB}, there are compatible Schur coverings of $G$ and $\Gamma$ 
		\[\begin{tikzcd}
			S \arrow[two heads]{r} \arrow[two heads]{d} & G \arrow[two heads]{d} \\
			S^{\Gamma} \arrow[two heads]{r} & \Gamma,
		\end{tikzcd}\]
		where the left vertical surjection is obtained by taking quotient of $S$ by the prime-to-$|\Gamma|$ part of $\ker (S \to G)$. By definition of $c_G$ and $c_{\Gamma}$, one can check that the diagram above defines the following compatible reduced Schur coverings, where the kernel of the left vertical map has order prime to $|\Gamma|$, which proves the statement \eqref{item:d-S-Gamma} in the proposition.
		\begin{equation*}
		\begin{tikzcd}
			S_{c_{G}} \arrow[two heads]{r} \arrow[two heads]{d} & G \arrow[two heads]{d} \\
			S^{\Gamma}_{c_{\Gamma}}\arrow[two heads]{r} & \Gamma.
		\end{tikzcd}
		\end{equation*}
		Finally, by a basic fact that a Schur covering group of $\Gamma$ (resp. $G$) has order supported on prime divisors of $|\Gamma|$ (resp. $|G|$), we have that $H_2(G,\Z)_{(p|\Gamma|)'}$ is the prime-to-$|\Gamma|$ completion of $H_2(G, c_G)$ and $H_2(\Gamma, c_{\Gamma})$ is the $|\Gamma|$-completion of $H_2(G, c_G)$, so the equality~\ref{eq:d-S} follows.   
	\end{proof}

	\begin{proposition}\label{prop:EVW-omega}
		Let $k$, $K$, $p$, $\Gamma$ and $L$ be as described in Definition~\ref{def:inv-L/K}. Let $G$ denote the Galois group of $L\overline{k}/\overline{k}(t)$, and $c$ the set of elements of $G \backslash \{1\}$ that have the same order as their image in $\Gamma$ (where the quotient map $G \to \Gamma$ is defined by the tower of fields $L\overline{k}/K\overline{k}/\overline{k}(t)$). 
		Then the composition map 		
		\[
			\hat{\Z}(1)^{\times}_{(p)'} \overset{\frakz}{\longrightarrow} K(G,c) \longrightarrow H_2(G,c) \longrightarrow H_2(G,\Z)_{(p|\Gamma|)'}
		\]
		extends to the $\omega$-invariant $\omega_{L/K}:\hZ(1)_{(p|\Gamma|)'} \to H_2(G, \Z)_{(p|\Gamma|)'}$ associated to $L/K$.
		Here $\frakz$ is the EVW-W lifting invariant associated to $L\overline{k}/\overline{k}(t)$, the second map is the projection of from $K(G,c)= H_2(G,c)\times \Z^{c/G}$ to the first coordinate, and the last map is the prime-to-$(p|\Gamma|)$ completion map, which is well-defined by Lemma~\ref{lem:decom-Schur} \eqref{item:d-S-H}.
	\end{proposition}
	
	\begin{proof}	
		Let $\pi:S\to G$ be a Schur covering, and $\pi_c: S_c \to G$ its reduced Schur covering associated to $c$.	
		Let $S' \to G$ be as defined in \eqref{eq:star}.
		The elements in $c$ have order dividing $|\Gamma|$. 
		For each conjugacy class in $c/G$, we pick one element $x$ in this class, and then pick one preimage $\widehat{x}$ of $x$ in $S_c$ such that the image of $\widehat{x}$ in $S'$ is the unique preimage of $x$ whose order equals $\ord(x)$. For any other element $y=g^{-1}xg$ in this conjugacy class, we define $\widehat{y}:=\tilde{g}^{-1}\widehat{x}\tilde{g}$ for some preimage $\tilde{g}$ of $g$ in $S_c$ (by \cite[Lem.2.3]{Wood-lifting}, $\widehat{y}$ is independent of the choice of $\tilde{g}$). 
		By \cite[Thm.2.5]{Wood-lifting}, there is an isomorphism from $U(G,c)$ to $S_c \times_{G^{ab}} \Z^{c/G}$ defined by mapping the generator $[x]$ for each $x\in c$ to $(\widehat{x}, e_x)$, where $e_x$ is the generator corresponding to the conjugacy class of $x$.
		
		Let $\underline{g}=(g_1,\ldots, g_m)$ be a system of inertia generators associated to $K$, and $\varphi$ and $\widetilde{\varphi}$ denote the surjections, defined in Corollary~\ref{cor:compatible-cover}, that map from $\langle g_1,\ldots, g_m \rangle$ to $G$ and $S'$ respectively. Then one can check that the composition map
		\begin{equation}\label{eq:EVW-omega}
			\hat{\Z}(1)^{\times}_{(p)'} \overset{\frakz}{\longrightarrow} K(G,c) \longrightarrow \ker(S_c \to G) \longrightarrow \ker (S' \to G)
		\end{equation}
		sends $I(\underline{g})$ to $\varphi(g_1\cdots g_m)$. We will show in the next paragraph that this composition map extends to   a group homomorphism $\hat{\Z}(1)_{(p)'} \to \ker (S' \to G)$, although $\frakz$ is only a $\hat{\Z}_{(p)'}^{\times}$-equivariant map of sets.
		Then the proposition follows by composing the map \eqref{eq:EVW-omega} with the inverse of the differential map $H_2(G,\Z)_{(p|\Gamma|)'} \overset{\sim}{\to} \ker (S' \to G)$.
		
		For each $x\in c$, we let $\widetilde{x}$ denote the unique preimage of $x$ in $S'$ such that $\ord(\widetilde{x})=\ord(x)$. Recall that  the lift $\widehat{x}$ of $x$ in $S_c$ that we've picked has image $\widetilde{x}$ in $S'$. So for each $\alpha \in \hat{\Z}^{\times}_{(p|\Gamma|)'}$, we have $\widetilde{x^{\alpha}}=\widetilde{x}^{\alpha}$, and hence $\widehat{x}^{-\alpha} \widehat{x^{\alpha}}$ is contained in $\ker(S_c \to S')$ for any $x \in c$. So for any element $(g, \underline{m}) \in K(G,c)$ (using notation in \eqref{eq:alpha-K}), the image of $\alpha \ast (g, \underline{m})$ under the map $K(G,c) \to \ker (S' \to G)$ equals the $\alpha$th power of the image of $(g,\underline{m})$. So we proved that \eqref{eq:EVW-omega} can extend to a group homomorphism $\hat{\Z}(1)_{(p)'} \to \ker (S' \to G)$.
	\end{proof}

\subsection{Proof of Theorem~\ref{thm:functionfield}}\label{ss:proof-main}

	Throughout this subsection, we assume that $\Gamma$ is a finite group and $H$ is a finite admissible $\Gamma$-group. We define $G:=H \rtimes \Gamma$, let $c_2$ denote the set $\Gamma \backslash \{1\}$, and $c_1$ the set of elements of $G \backslash \{1\}$ that have the same order as their image in $\Gamma$. Let $q$ be a power of a prime number $p$ such that $p \nmid |G|$. 
	
	By Lemma~\ref{lem:decom-Schur}, there exists a commutative diagram of surjections
	\begin{equation}\label{eq:diag-Schur}
	\begin{tikzcd}
		&  S' \arrow["\pi'", two heads]{dr}& \\
		\overline{S}^1 \arrow["\phi", two heads]{ur}\arrow["\pi^1", swap, two heads]{rr} \arrow["\widetilde{\rho}",  two heads]{d} & & G \arrow["\rho", two heads]{d} \\
		\overline{S}^2 \arrow["\pi^2", swap, two heads]{rr} && \Gamma
	\end{tikzcd}
	\end{equation}
	where $\pi^1$ is a reduced Schur covering of $G$ and $c_1$, $\pi^2$ is a reduced Schur covering of $\Gamma$ and $c_2$, and $\pi'$ is a $(q|\Gamma|)'$-Schur covering of $G$. 
	Moreover, since $\gcd \left( |\ker \rho \circ \pi'|, |\ker \pi^2| \right)=1$, $\overline{S}^1$ is the fiber product $S' \times_{\Gamma} \overline{S}^2$.

		For a positive integer $n$ and a group homomorphism $\delta: \hat{\Z}(1)_{(q|\Gamma|)'} \to \ker\pi'$, we let $\pi_{G,c_1}^{\delta}(q,n)$ denote the number of $\Frob_{(\HHur_{G,c_1}^n)_{\F_q}}$-fixed components of $(\HHur_{G,c_1}^n)_{\overline{\F}_q}$ whose corresponding $\omega$-invariant is the composition of $\delta$ with the inverse of the differential map of  $\pi'$. We pick a $\zeta \in \hat{\Z}(1)_{(q|\Gamma|)'}^{\times}$ and let $\eta:=\delta(\zeta)$. Define the quantity $b(G,c_1,q,n; \delta)$ as follows, which is independent of the choice of $\zeta$.
		\[
			b(G,c_1,q,n; \delta):= \sum_{\substack{h \in \ker \pi^1 \\ \text{s.t. }\phi(h)=\eta}} \# \left\{ \underline{m} \in \ker \left( \Z^{c_1/G}_{\equiv q, n, \geq 0} \to G^{ab}\right) \, \biggr\rvert \, W_{q^{-1}}(\underline{m})^q=h^{q-1}\right\}
		\]
		
	\begin{lemma}\label{lem:pi-b-marked}
		Given an element $a \in (\Z/|G|^2\Z)^{\times}$, there is a positive integer $M_a$ and a non-empty set $E_a$ of residues modulo $M_a$, and positive numbers $r'_{a,b}$ for $b\in E_a$, such that for a prime power $q\equiv a \, (\operatorname{mod } |G|^2)$ and a non-negative integer $n$ the following holds.
		\begin{enumerate}
			\item \label{item:pi-b-1} If $\ord(\eta) \mid q-1$ and $n \,(\operatorname{mod } M_a) =b \in E_a$, then 
			\[
				\pi^{\delta}_{G,c_1}(q,n)=b(G,c_1,q,n; \delta)+O_G(n^{d_{G,c_1}(q)-2})=r'_{a,b}n^{d_{G,c_1}(q)-1} +O_G(n^{d_{G,c_1}(q)-2}).
			\]
			\item \label{item:pi-b-2} If $\ord(\eta)\nmid q-1$ or $n \,(\operatorname{mod } M_a)\not\in E_a$, then $\pi^{\delta}_{G,c_1}(q,n)=b(G,c_1, q, n ;\delta)=0$.
		\end{enumerate}
	\end{lemma}

	\begin{proof}
		By Proposition~\ref{prop:EVW-omega} and \cite[Cor.12.6]{LWZB}, $\pi^{\delta}_{G,c_1}(q,n)$ can be approximated by the number of $\hat{\Z}^{\times}_{(p)'}$-equivariant maps of sets $\frakz:\hat{\Z}(1)^{\times}_{(p)'} \to K(G,c_1)$ such that $\frakz$ reduces to $\delta$ and all coordinates of the image of $\frakz$ compositiond with the projection $K(G,c_1)\to \Z^{c_1/G}$ is at least $M$ for a sufficiently large $M$. Then the lemma follows from the proof of Proposition~12.7 and Corollary~12.9 of \cite{LWZB}.
	\end{proof}

	\begin{lemma}\label{lem:b=b}
		For any group homomorphism $\delta: \hat{\Z}(1)_{(p|\Gamma|)'} \to \ker  \pi'$ such that $\ord(\im \delta) \mid q-1$, we have 
		\[
			b(G, c_1, q, n; \delta)=b(\Gamma, c_2, q, n).
		\]
	\end{lemma}
	
	\begin{proof}
		By definition of $c_1$ and $c_2$, the quotient map $\rho:G \to \Gamma$ sends each element $c_1$ to an element in $c_2$. Moreover, as $\gcd(|H|, |\Gamma|)=1$, by the Schur--Zassenhaus theorem, for each element of $c_2$, all of its lifts in $c_1$ are conjugate to each other, thus there is a natural bijective correspondence $\Z^{c_1/G} \to \Z^{c_2/\Gamma}$ induced by $\rho$. Since $H$ is an admissible $\Gamma$-group, we have $G^{ab}=\Gamma^{ab}$, so the restriction of this bijection gives a natural isomorphism $\ker(\Z^{c_1/G} \to G^{ab}) \simeq \ker(\Z^{c_2/\Gamma} \to \Gamma^{ab})$.
		
		We write the group homomorphisms $W^1_{\alpha}: \Z^{c_1/G} \to \ker\pi_1$ and $W^2_{\alpha}: \Z^{c_2/\Gamma} \to \ker \pi_2$ for $W_{\alpha}$ (defined above the equation~\eqref{eq:alpha-K}) associated to $(G, c_1)$ and $(\Gamma, c_2)$ respectively. For each $g \in c_1/G$ and its image $\gamma \in c_2/\Gamma$, we can pick $x_g$ in $g$ and $x_{\gamma}$ in $\gamma$ such that $\rho(x_g)=x_{\gamma}$, and then pick their lifts $\widehat{x_{g}} \in \overline{S}^1$ and $\widehat{x_{\gamma}} \in \overline{S}^2$ such that $\widetilde{\rho}(\widehat{x_g})=\widehat{x_{\gamma}}$. Then by definition of $W^1_{\alpha}$ and $W^2_{\alpha}$, we have the following commutative diagram.
		\[\begin{tikzcd}
			\Z^{c_1/G} \arrow["\sim"]{d}\arrow["W^1_{q^{-1}}"]{r} & \ker(\overline{S}^1 \overset{\pi^1}{\to} G) \arrow[two heads, "\widetilde{\rho}"]{d} \\
			\Z^{c_2/\Gamma} \arrow["W^2_{q^{-1}}"]{r} & \ker(\overline{S}^2 \overset{\pi^2}{\to} \Gamma)
		\end{tikzcd}\]
		
		Pick a $\zeta \in \hat{\Z}(1)^{\times}_{(p|\Gamma|)'}$ and let $\eta:= \delta(\zeta)$. Considering the diagram \eqref{eq:diag-Schur} and the isomorphism $\overline{S}^1\simeq S' \times_{\Gamma} \overline{S}^2$, we obtain the following identity of sets 
		\begin{eqnarray*}
			\left\{h \in \ker(\overline{S}^1 \overset{\pi^1}{\to} G) \, \bigr\rvert \, \phi(h)=\eta \right\} &\overset{1-1}{\longleftrightarrow}&  \ker (\overline{S}^2 \overset{\pi^2}{\to} \Gamma)  \\
			h & \longmapsto& \widetilde{\rho}(h). 		
		\end{eqnarray*}
	Therefore, combining everything, we have
		\begin{eqnarray*}
			b(G, c_1, q, n; \delta) &=& \sum_{\substack{h \in \ker(\overline{S}^1 \to G) \\ \text{s.t. } \phi(h) = \eta}} \# \left\{ \underline{m} \in \ker \left( \Z^{c_1/G}_{\equiv q, n, \geq 0 } \to G^{ab} \right) \, \biggr\rvert \, W^1_{q^{-1}}(\underline{m})^q=h^{q-1}\right\} \\
			&=& \sum_{\bar{h} \in \ker(\overline{S}^2 \to \Gamma)} \# \left\{ \underline{m} \in \ker \left( \Z^{c_2/\Gamma}_{\equiv q, n, \geq 0 } \to \Gamma^{ab} \right) \, \biggr\rvert \, W^2_{q^{-1}}(\underline{m})^q=\bar{h}^{q-1}\right\} \\
			&=& b(\Gamma, c_2, q, n).
		\end{eqnarray*}
	\end{proof}

	By the correspondence between smooth projective curves and their function fields, $\HHur^n_{G,c_1}(\F_q)$ parametrizes the pairs $(L, \iota)$, where $L$ is an extension of $\F_q(t)$ that is split completely at $\infty$ and has $\rDisc L=q^n$, and $\iota$ is an isomorphism $\Gal(L/\F_q(t)) \overset{\sim}{\to} G$ such that each inertia subgroup of $L/\F_q(t)$ is generated by an element in $c_1$. We fix a surjection $\rho: G \to \Gamma$ as defined in the diagram~\eqref{eq:diag-Schur}. Then for each pair $(L, \iota)$, consider the subfield $K:=L^{\ker(\rho \circ \iota)}$ of $L$. The surjection $\rho \circ \iota$ defines an isomorphism $\Gal(K/\F_q(t)) \to \Gamma$, and moreover $L/K$ is unramified and  $\rDisc(K)=q^n$.
	Then the restriction of $\iota$ to $\Gal(L/K)$ gives an isomorphism $\Gal(L/K) \to H$.
	
	\begin{lemma}\label{lem:Hur-Sur}
		There is a surjective map of sets
		\begin{eqnarray}
			\HHur_{G,c_1}^n(\F_q) &\longrightarrow& \coprod_{K \in E_{\Gamma}(q^n, \F_q(t))} \Sur_{\Gamma}(\Gs(K), H) \label{eq:Hur-Sur}\\
			(L, \iota) & \longmapsto& \{ \Gs(K) \to \Gal(L/K) \}. \nonumber
		\end{eqnarray}
		Each element in the set on the right has $[H:H^{\Gamma}]$ preimages.
	\end{lemma}
	
	\begin{proof}
		The lemma follows by the proofs of Lemmas~9.3 and 10.2 in \cite{LWZB}.
	\end{proof}

	Let $q$ be a prime power such that $\gcd(q, |G|)=1$ and $\ord(\im \delta) \mid q-1$. Then taking $q$th power on $\im \delta$ is the identity map. So by the Frobenius action on lifting invariants in \eqref{eq:Frob-inv} and Proposition~\ref{prop:EVW-omega}, if there is a $\overline{\F}_q$-point of a connected component of $(\HHur_{G,c_1}^n)_{\F_q}$ has $\omega$-invariant $\delta$, then all the $\overline{\F}_q$-point of this component have $\omega$-invariant $\delta$. 
	So we let $Z^{\delta}_{q,n}$ be the union of components of $(\HHur^n_{G,c_1})_{\F_q}$ whose $\overline{\F}_q$-points has $\omega$-invariant $\delta$. 
	By Lemma~\ref{lem:Hur-Sur}, we have
	\[
		\frac{\#Z^{\delta}_{q,n}(\F_q)}{[H:H^{\Gamma}]}=  \sum_{K \in E_{\Gamma}(q^n, \F_q(t))} \#\left\{ \pi\in \Sur_{\Gamma}\left(\Gs(K), H \right) \, \bigl\rvert\, \pi^{*} \circ \omega^{\#}_K=\delta \right\}.
	\]	 
	By the Grothendieck--Lefschetz trace formula, since $(Z^{\delta}_{q,n})_{\overline{\F}_q}$ is $n$-dimensional (\cite[Lem.10.3]{LWZB}), we have
	\[
		\#Z^{\delta}_{q,n}(\F_q) = \sum_{j=0}^{2n} (-1)^j \Tr(\Frob_{\F_q} |H^j_{c, \et}((Z^{\delta}_{q,n})_{\overline{\F}_q}, \Q_\ell))
	\]
	where $H^j_{c,\et}$ is the $j$th \'etale fundamental group with compact support.
	Here $\Tr(\Frob_{\F_q} |H^{j}_{c, \et}((Z^{\delta}_{q,n})_{\overline{\F}_q}, \Q_\ell))$ is $\pi^{\delta}_{G,c_1}(q,n)q^n$ when $j=2n$, and is at most $q^{j/2}C(n)$ when $0\leq j < 2n$ by the comparison of \'etale and analytic cohomology of components of the Hurwitz scheme \cite[Lem.10.3]{LWZB}. Here   
	\[
		C(n)=\max_{0\leq j < 2n} \dim H^{2n-j} (Z^{\delta}_{\C,n}, \Q),
	\]
	where $Z^{\delta}_{\C,n}$ is the union of components of $(\HHur^n_{G,c_1})_{\C}$ whose prime-to-$|\Gamma|$-torsion part of the EVW-W lifting invariant
	\[
		\hZ(1)^{\times} \overset{\frakz}{\longrightarrow} K(G,c_1) \longrightarrow H_2(G, \Z)_{(|\Gamma|)'} 
	\]
	is given by $\delta$. 
	So we round down $N$ to the nearest integer whose residue mod $M_a$ is in $E_a$, and so we have
	\[
		\sum_{n\leq N} \# Z^{\delta}_{q,n}(\F_q) = \pi^{\delta}_{G, c_1}(q, N) q^N + O_{N,G}(q^{N-1/2}).
	\]
	Similarly, $E_{\Gamma}(q^n, \F_q(t))=\#\HHur_{\Gamma, c_2}^n(\F_q)$, and
	\[
		\sum_{n \leq N} \#\HHur^n_{G,c_1}(\F_q) = \pi_{\Gamma, c_2}(q, N)q^N+O_{N,\Gamma}(q^{N-1/2}).
	\]	
	Therefore, the fraction in the limits of the left-hand side of the equality in Theorem~\ref{thm:functionfield} equals
	\begin{eqnarray*}
		\frac{1}{[H:H^{\Gamma}]}\frac{\sum_{n \leq N} \#Z^{\delta}_{q,n}(\F_q)}{\sum_{n \leq N} \#\HHur^n_{\Gamma, c_2}(\F_q)} &=&\frac{1}{[H:H^{\Gamma}]} \frac{\pi^{\delta}_{G,c_1}(q,N) q^N + O_{N,G}(q^{N-1/2})}{\pi_{\Gamma, c_2}(q, N) q^N+O_{N,\Gamma}(q^{N-1/2})}\\
		&=& \frac{1}{[H:H^{\Gamma}]}\left(1+ \frac{O_G(N^{d_{\Gamma}(q)-2}) + O_{N,G}(q^{-1/2})}{\pi_{\Gamma,c_2}(q,N)+O_{N,\Gamma}(q^{-1/2})} \right),
	\end{eqnarray*}
	where the last equality follows by Lemmas~\ref{lem:pi-b-marked}, \ref{lem:b=b} and \cite[Prop.12.7]{LWZB}. Then the equality in Theorem~\ref{thm:functionfield} follows after taking limits.

\section{A random pro-$\ell$ $\Gamma$-group for $\Gamma=\Z/2\Z$}\label{sect:RandomGroup}

	In this section, we assume that $\Gamma=\Z/2\Z$, and $\ell$ is an odd prime that is not $\Char(k)$. We construct a random group to model the distribution of $\Gs(K)(\ell)$ as $K$ varies among all the quadratic extensions of $k(t)$ that is split completely at $\infty$.
	
	Recall that $G_{\O}(K\overline{k})(\ell)$ is a pro-$\ell$ Demu\v{s}kin group whose abelianization is free of rank $2g$, where $g$ is the genus of the curve over $\overline{k}$ corresponding to $K\overline{k}$. We've shown that $G_{\O}(K\overline{k})(\ell)$ has a $\Gamma$ action and a $\Frob_k$ action, and it admits a $\Gamma$-equivariant and $\Frob_k$-equivariant stem extension \eqref{eq:ext-ell} whose kernel is isomorphic to $\Z_\ell(1)$. Moreover, $\Gs(K)(\ell)$ is the maximal $\Gamma$-equivariant and $\Frob_k$-equivariant quotient of $G_{\O}(K\overline{k})(\ell)$ whose induced $\Frob_k$-action is trivial. So our strategy is to treat the $\Frob_k$ action as a random automorphism of the stem extension. We start with a Demu\v{s}kin group with desired properties, pick a random automorphism that acts on the kernel of the stem extension as taking the $|k|$-th powers. Then define the random group to be the maximal quotient fixed by this random automorphism, as the generating rank of the Demu\v{s}kin group goes to infinity.

\subsection{Construction of the random group}\label{ss:construction}

	By the classification of Demu\v{s}kin groups \cite[Thm.(3.9.11)]{NSW}, if a Demu\v{s}kin group has abelianization isomorphic to $\Z_\ell^{2n}$ , then it is isomorphic to $\langle x_1, \ldots, x_{2n} \mid [x_1, x_2][x_3, x_4]\cdots [x_{2n-1},x_{2n}]\rangle$. In the following lemma, we show that there is a unique (up to isomorphism) $\Z/2\Z$-action on this type of Demu\v{s}kin group such that the Demu\v{s}kin group is an admissible $\Z/2\Z$-group.

	\begin{lemma}\label{lem:p-gp-aut}
		Let $G$ be a pro-$\ell$ group and $\sigma_1, \sigma_2\in \Aut(G)$ have the same order which is prime to $\ell$. Assume that $\pi: G \to H$ is a surjection such that $\ker \pi \subset [G,G]$, and that $\sigma_1, \sigma_2$ preserve $\ker \pi$ and induce the same automorphism on $H$. Then there exists $\tau \in \Aut(G)$ such that $\sigma_1=\tau^{-1}\circ \sigma_2 \circ \tau$ and $\tau$ preserves $\ker \pi$ and induces the trivial action on $H$.
	\end{lemma}
	
	\begin{proof}
		Consider the subgroup $\Aut(G, \pi)$ of $\Aut(G)$ consisting of all the automorphisms of $G$ preserving $\ker \pi$. There is a natural map $\alpha: \Aut(G, \pi) \to \Aut(H)$. Since $\ker \pi \subset [G,G]$, by \cite[Chap.5, Thm.3.5]{Gorenstein}, the kernel of $\alpha$ is a pro-$\ell$ group. The assumptions in the lemma imply that $\alpha(\sigma_1)$ equals $\alpha(\sigma_2)$ and has order prime to $\ell$. Then by the Schur--Zassenhaus theorem, $\sigma_1$ and $\sigma_2$ are conjugate in $\Aut(G, \pi)$ (and hence in $\Aut(G)$) by an element in $\ker \alpha$.
	\end{proof}

	\begin{lemma}\label{lem:C2-Demu}
		Let $n$ be a positive integer and $\calG_n$ a Demu\v{s}kin group whose abelianization is isomorphic to $\Z_{\ell}^{2n}$. Then $\calG_n$ admits a Schur covering (unique up to isomorphism)
		\[
			1 \longrightarrow \Z_{\ell} \longrightarrow \widetilde{\calG}_n \overset{\pi_n}{\longrightarrow} \calG_n \longrightarrow 1.
		\]
		There exists a $\Z/2\Z$ action on $\widetilde{\calG}_n$ such that the action preserves and acts trivially on $\ker \pi_n$ and its induced action on $\calG_n^{ab}$ is taking inversion. Moreover, if $\sigma_1$ and $\sigma_2$ are two such $\Z/2\Z$ actions on $\widetilde{\calG}_n$, then there exists an automorphism $\tau$ of $\widetilde{\calG}_n$ such that $\sigma_1= \tau^{-1} \circ \sigma_2 \circ \tau$, and $\tau$ preserves $[\widetilde{\calG}_n, \widetilde{\calG}_n]$ and induces the trivial action on $\calG_{n}^{ab}$.
	\end{lemma}

	\begin{proof}
		Consider the pro-$\ell$ group defined by the following presentation
		\begin{equation}\label{eq:pres-Demu-Gamma}
			G:=\langle x_1, \cdots, x_{2n} \mid x_1^{-1} x_2^{-1}\cdots x_{2n}^{-1} x_1 x_2 \cdots x_{2n} \rangle.
		\end{equation}
		Let $F$ denote the free pro-$\ell$ generated by $x_1, \cdots, x_{2n}$, and let $R$ denote the closed normal subgroup of $F$ generated by $x_1^{-1}\cdots x_{2n}^{-1}x_1 \cdots x_{2n}$.
		Note that 
		\[
			x_1^{-1}\cdots x_{2n}^{-1} x_1 \cdots x_{2n} = [x_1, x_2][x_1x_2, x_{2n}\cdots x_3] x_3^{-1}\cdots x_{2n}^{-1} x_3 \cdots x_{2n},
		\]
		\[
			\text{and } [x_1x_2, x_{2n}\cdots x_3] \equiv [x_1, x_3] \cdots [x_1, x_{2n}] [x_2, x_3] \cdots [x_2, x_{2n}] \mod F^{(3)}. 
		\]
		By repeating this process, we have $x_1^{-1}\cdots x_{2n}^{-1} x_1 \cdots x_{2n} \equiv \prod_{1\leq i < j \leq 2n} [x_i, x_j] \mod F^{(3)}$. So by \cite[Prop.3.9.13(ii)]{NSW}, the cup product $H^1(G, \Z_\ell) \times H^1(G, \Z_\ell) \to H^2(G, \Z_\ell)$ is non-degenerate. Moreover, $G^{ab}$ is free of rank $2n$. So $G$ is a Demu\v{s}kin group whose abelianization is isomorphic to $\Z_{\ell}^{2n}$ and hence is isomorphic to $\calG_n$, and we will use this presentation \eqref{eq:pres-Demu-Gamma} of $\calG_n$ in the rest of the proof. 
		
		The existence and uniqueness of the Schur covering $\pi_n$ follow immediately from $H_2(G, \Z)\simeq \Z_\ell$ and $\Ext_{\Z}^1(G^{ab}, \Z_{\ell})=0$. We define $\sigma \in \Aut(F)$ by $\sigma(x_i)=x_i^{-1}$. We will show next that $\sigma$ induces a $\Z/2\Z$ action on $\widetilde{\calG}_n$ as described in the lemma. First, because 
		\begin{equation}\label{eq:sigma-rel}
			\sigma(x_1^{-1} \cdots x_{2n}^{-1} x_1 \cdots x_{2n})=x_1\cdots x_{2n} x_1^{-1} \cdots x_{2n}^{-1}=(x_1^{-1}\cdots x_{2n}^{-1} x_1\cdots x_{2n})^{x_{2n}^{-1} \cdots x_{1}^{-1}},
		\end{equation}
		the automorphism $\sigma$ stabilizes $R$. So $\sigma$ induces an automorphism $\overline{\sigma}$ of $G$, and $\overline{\sigma}$ acts on $G^{ab}$ by inversion. The presentation of $G$ defines a stem extension of $G$ with kernel isomorphic to $\Z_{\ell}$
		\[
			1 \longrightarrow R/[F,R] \longrightarrow F/[F,R] \longrightarrow G \longrightarrow 1,
		\]
		where the image of $x_1^{-1}\cdots x_{2n}^{-1}x_1\cdots x_{2n}$ generates the kernel $R/[F,R]$.
		So \eqref{eq:sigma-rel} shows that $\sigma$ preserves $[F,R]$, and hence it induces an automorphism of $F/[F,R]$, and moreover this automorphism acts trivially on $R/[F,R]$. So we proved the existence of the $\Z/2\Z$ action on $\widetilde{\calG}_n$ as described in the lemma.
		
		Suppose that $\sigma_1$ and $\sigma_2$ are two $\Z/2\Z$ actions on $\widetilde{\calG}_n$ as described in the lemma.
		Then since they induce the same automorphism on $\calG_n^{ab}$, the last statement of the lemma follows by Lemma~\ref{lem:p-gp-aut}.
	\end{proof}
	
	\begin{definition}[Definition of the random group]\label{def:random-group}
		Let $n$ be a positive integer, $\calG_n$, $\widetilde{\calG}_n$, $\pi_n$ and the $\Gamma=\Z/2\Z$ action on them be as defined in Lemma~\ref{lem:C2-Demu}, and $q$ be a positive integer such that $\ell \nmid q$. We fix a generator $\xi_n$ of $\ker \pi_n$. 
		Let $\Aut(\widetilde{\calG}_n\rtimes \Gamma, \pi_n; q)$ denote the set of automorphisms of $\widetilde{\calG}_n\rtimes \Gamma$ that preserves and acts as multiplication by $q$ on $\ker \pi_n$. 
		Let $\phi$ be a random element of $\Aut(\widetilde{\calG}_n \rtimes \Gamma, \pi_n ; q)$ with respect to the ``Haar'' measure. We let $1-\phi$ denote the quotient map from $\widetilde{\calG}_n \rtimes \Gamma$ to its maximal $\Gamma$-equivariant quotient on which the induced $\phi$ action is trivial.
		 Then we define the random group $X_{q,n}$ as
		 \begin{equation}
		 	X_{q, n} := (1-\phi)(\widetilde{\calG}_n \rtimes \Gamma) / (1-\phi)(\ker \pi_n).
		 \end{equation}
		 Moreover, we fix a homomorphic splitting of $\widetilde{\calG}_n \rtimes \Gamma \to \Gamma$ and hence an action of $\Gamma$ (by conjugation) on $\widetilde{\calG}_n$. Then $\ker (X_{q,n} \to \Gamma)$ naturally obtains an action of $\Gamma$. We define the following random $\Gamma$-group
		 \[
		 	Y_{q,n} := \ker (X_{q,n} \to \Gamma).
		 \] 
	\end{definition}
	
	\begin{remark}
		\begin{enumerate}
			\item We need to explain what the ``Haar'' measure is in the definition. Since $\widetilde{\calG}_n\rtimes \Gamma$ is a finitely generated profinite group, $\Aut(\widetilde{\calG}_n\rtimes \Gamma)$ is profinite by \cite[Cor.4.4.4]{Ribes-Zalesskii}. The subgroup $\Aut(\widetilde{\calG}_n \rtimes \Gamma, \pi_n; 1)$ consisting of all automorphisms that stabilizes and acts trivially on $\ker \pi_n$ is a closed subgroup of $\Aut(\widetilde{\calG}_n \rtimes \Gamma)$, so it is a profinite group, and hence has a Haar measure. Then since $\Aut(\widetilde{\calG}_n\rtimes \Gamma, \pi_n; q)$ is a coset of $\Aut(\widetilde{\calG}_n\rtimes \Gamma, \pi_n; 1)$ in $\Aut(\widetilde{\calG}_n \rtimes \Gamma)$, it obtains a measure from the measure on $\Aut(\widetilde{\calG}_n \rtimes \Gamma, \pi_n; 1)$, which is the ``Haar'' measure in the definition. 
			\item Since $\ell$ is odd, $\widetilde{\calG}_n$ is a characteristic subgroup of $\widetilde{\calG}_n \rtimes \Gamma$ and then $\phi$ preserves $\widetilde{\calG}_n$. Any action on $\Z/2\Z$ is trivial, so $\ker (1-\phi)$ is contained in $\widetilde{\calG}_n$ for any $\phi$.
		 \end{enumerate}
	\end{remark}

\subsection{Computing the moment of the random group}\label{ss:moment-random-group}

	\begin{lemma}\label{lem:Labute}
		Consider the Schur covering $1 \to \Z_{\ell} \to \widetilde{\calG}_n \overset{\pi_n}{\longrightarrow} \calG_n \to 1$. Let $\xi_n$ be a generator of $\ker \pi_n$. Let $H$ be a pro-$\ell$ group, $\varpi: \widetilde{H} \to H$ a Schur covering of $H$, and $\rho: \calG_n \to H$ a group surjection. 
		\begin{enumerate}
			\item \label{item:Labute-1} Then $\rho$ can be lifted to a surjection $\widetilde{\rho}: \widetilde{\calG}_n \to \widetilde{H}$ such that $\varpi\circ \widetilde{\rho}=\rho \circ \pi_n$. Moreover, the image $\widetilde{\rho}(\xi_n)$ in $\widetilde{H}$ does not depend on the choice of the lift $\widetilde{\rho}$.
			\item \label{item:Labute-2} Given $\rho_1, \rho_2 \in \Sur(\calG_n, H)$, let $\widetilde{\rho}_1, \widetilde{\rho}_2 \in \Sur(\widetilde{\calG}_n, \widetilde{H})$ be their lifts respectively. Assume that $\widetilde{\rho}_1(\xi_n)=\widetilde{\rho}_2(\xi_n)$. Then for any $q\in \Z_{\ell}^{\times}$ with $\ord(\widetilde{\rho}_1(\xi_n))\mid q-1$, there exists $\widetilde{\varphi} \in \Aut(\widetilde{\calG}_n)$ such that 
			\begin{enumerate}
				\item \label{item:L-c-2} $\widetilde{\varphi}$ preserves $\ker \pi_n$ and induces $\varphi \in \Aut(\calG_n)$ such that $\rho_1=\rho_2 \circ \varphi$, and
				\item \label{item:L-c-3} $\widetilde{\varphi}(\xi_n)=\xi_n^q$.
			\end{enumerate}
		\end{enumerate}
	\end{lemma}
	
	We will give the proof of this lemma in \S\ref{ss:Labute}, which heavily uses the method in \cite{Labute} of successive approximation on the filtration of lower $\ell$-central series of pro-$\ell$ groups.

	Let $H$ be a finite $\ell$-group with an action of $\Gamma$, and let $\pi \in\Sur(X_{q,n}, H \rtimes \Gamma)$ (or $\pi \in \Sur_{\Gamma}(Y_{q,n}, H)$ when we fix the action of $\Gamma$ on $\calG_n$). Consider the composition map $\calG_n \rtimes \Gamma \overset{1-\phi}{\longrightarrow} X_{q,n} \overset{\pi}{\longrightarrow} H \rtimes \Gamma$ and let $\varpi: \widetilde{H} \rtimes \Gamma \to H \rtimes \Gamma$ be an $\ell$-Schur covering. By Lemma~\ref{lem:comp-schur}, there is a surjection $\widetilde{\calG}_n\rtimes \Gamma \to \widetilde{H} \rtimes \Gamma$ satisfing the following commutative diagram
	\begin{equation*}
	\begin{tikzcd}
		1 \arrow{r} & \Z_{\ell} \arrow{r} \arrow{d} & \widetilde{\calG}_n \rtimes \Gamma \arrow["\pi_n"]{r} \arrow{d} & \calG_n \rtimes \Gamma \arrow[two heads, "\pi \circ (1-\phi)"]{d} \arrow{r} & 1 \\
		1 \arrow {r} & \ker \varpi \arrow{r} & \widetilde{H} \rtimes \Gamma \arrow["\varpi"]{r} & H \rtimes \Gamma \arrow{r} & 1,
	\end{tikzcd}
	\end{equation*}
	such that the first vertical arrow in the diagram is the composition of
	\[\begin{tikzcd}
		\Z_{\ell} \arrow["d_{\pi_n}^{-1}", "\sim"']{r} & H_2(\calG_n \rtimes \Gamma, \Z)(\ell) \arrow["\coinf"]{r} & H_2(H \rtimes \Gamma, \Z)(\ell) \arrow["d_{\varpi}", "\sim"']{r} & \ker \varpi,
	\end{tikzcd}\]
	where the middle map $\coinf$ is induced by the quotient map $\pi \circ(1-\phi)$. We denote the composition of the first two maps by $\pi_\dagger: \Z_{\ell} \to H_2(H \rtimes \Gamma, \Z)(\ell)$. Because the sequence is $\phi$-equivariant, $\pi_{\dagger}$ factors through $\Z_{\ell} \twoheadrightarrow \Z_{\ell}/(q-1)\Z_{\ell}$.  
	
	\begin{theorem}\label{thm:moment}
		Let $H$ be a finite $\ell$-group with an action of $\Gamma=\Z/2\Z$ that acts on $H^{ab}$ as inversion. 
		Let $n$ be a sufficiently large integer such that there exists a $\Gamma$-equivariant surjection from $\calG_n$ to $H$.
		Let $\delta$ be a group homomorphism $\Z_{\ell} \to H_2(H \rtimes \Gamma, \Z)(\ell)[q-1]$. 
		Then
		\[
			\EE\left( \#\left\{ \pi \in \Sur(X_{q,n}, H \rtimes \Gamma) \mid \pi_{\dagger}=\delta  \right\} \right) =1
		\]
		Also, if we fix a homomorphic splitting of $\widetilde{\calG}_n \rtimes \Gamma \to \Gamma$ (that is, we fix the action of $\Gamma$ on $\widetilde{\calG}_n$), then 
		\[
			\EE \left( \# \left\{ \pi \in \Sur_{\Gamma}(Y_{q, n}, H) \mid \pi_\dagger = \delta \right\} \right) = \frac{1}{[H : H^{\Gamma}]}.
		\]
	\end{theorem}

	\begin{proof}
		For any $\pi \in \Sur(\calG_n \rtimes \Gamma, H \rtimes \Gamma)$, by Lemma~\ref{lem:comp-schur}, there exists $\widetilde{\pi} \in \Sur(\widetilde{\calG}_{n} \rtimes \Gamma, \widetilde{H} \rtimes \Gamma)$ such that $\pi \circ \pi_n = \varpi \circ \widetilde{\pi}$. 
		Let's prove that there is a unique such $\widetilde{\pi}$. 
		Let $S$ be the closed subgroup of $\widetilde{\calG}_{n} \rtimes \Gamma$ generated by all elements of order $2$. Then $S$ is normal because the set of all the order-$2$ elements are closed under conjugation. The quotient $\widetilde{\calG}_{n} \rtimes \Gamma/S$ is a $\Gamma$-equivariant  pro-$\ell$ quotient of $\widetilde{\calG}_{n} \rtimes \Gamma$ with trivial induced $\Gamma$ action. However, $\Gamma$ acts on $\widetilde{\calG}_n^{ab}$ as inversion and any non-trivial pro-$\ell$ group has non-trivial abelianization, so $\widetilde{\calG}_n$ does not have any non-trivial quotient on which the induced $\Gamma$-action is trivial. So we have $S=\widetilde{\calG}_{n} \rtimes \Gamma$, which implies that $\widetilde{\calG}_{n} \rtimes \Gamma$ is generated by order-2 elements. 
		Also, for each order-2 element $x$ of $\widetilde{\calG}_{n} \rtimes \Gamma$, we have that $\widetilde{\calG}_{n} \rtimes \Gamma$ is the product of $\widetilde{\calG}_n$ and $\langle x\rangle$, so $\pi \circ \pi_n(x)$ is non-trivial and hence has order 2.
		Each order-2 element of $H\rtimes \Gamma$ lifts to a unique order-2 element of $\widetilde{H} \rtimes \Gamma$ because $\widetilde{H}\to H$ is an $\ell$-Schur covering. Therefore, $\widetilde{\pi}$ must map each order-2 element $x$ of $\widetilde{\calG}_{n} \rtimes \Gamma$ to the unique lift of $\pi \circ \pi_n(x)$, which finishes the proof of the uniqueness of $\widetilde{\pi}$. So for each $\pi \in \Sur(\calG_n \rtimes \Gamma, H \rtimes \Gamma)$, we let $\widetilde{\pi}$ denote its unique lift $\Sur(\widetilde{\calG}_n \rtimes \Gamma, \widetilde{H} \rtimes \Gamma)$.
		
		Next, we prove the first equality in the theorem. Let $h=d_{\varpi} \circ \delta(\xi_n)$.
		Note that 
		\begin{eqnarray}
			&& \EE \left(\# \left\{ \pi \in \Sur(X_{q,n}, H\rtimes \Gamma) \mid \pi_{\dagger} = \delta \right\} \right) \nonumber\\
			&=& \sum_{\substack{\pi \in \Sur(\calG_n \rtimes \Gamma, H\rtimes \Gamma) \\ \text{s.t. } \widetilde{\pi}(\xi_n)=h}} \Prob ( \phi \in \Aut(\widetilde{\calG}_n \rtimes \Gamma, \pi_n ; q) \mid  \ker (1-\phi) \subseteq \ker \widetilde{\pi} ). \label{eq:EE}
		\end{eqnarray}
		The condition $\ker (1-\phi) \subseteq \ker \widetilde{\pi}$ is equivalent to $\widetilde{\pi} \circ \phi =\widetilde{\pi}$. 
		We let $\Sur(\calG_n\rtimes \Gamma, H\rtimes \Gamma; \xi_n\mapsto h)$ denote the set of $\pi$ as described under the sum in \eqref{eq:EE}, and let $\pi \in \Sur(\calG_n \rtimes \Gamma, H \rtimes \Gamma; \xi_n\mapsto h)$. Each $\phi \in \Aut(\widetilde{\calG}_n \rtimes \Gamma, \pi_n ; q)$ passes $\widetilde{\pi}$ to another surjection $\widetilde{\pi}\circ \phi \in \Sur(\widetilde{\calG}_n \rtimes \Gamma, \widetilde{H}\rtimes \Gamma)$, and the assumption $\ord(h)\mid q-1$ implies that $\widetilde{\pi}\circ \phi(\xi_n)=h$, so $\widetilde{\pi} \circ \phi$ induces a surjection in $\Sur(\calG_n \rtimes \Gamma, H \rtimes \Gamma; \xi_n\mapsto h)$. In other words, $\Aut(\widetilde{\calG}_n \rtimes \Gamma, \pi_n ; q)$ acts on the set $\Sur(\calG_n\rtimes \Gamma, H \rtimes \Gamma; \xi_n \mapsto h)$, and in particular, the stabilizer of $\pi$ by this action is exactly the set of $\phi\in \Aut(\widetilde{\calG}_n \rtimes \Gamma, \pi_n;q)$ such that $\ker(1-\phi)\subseteq \ker \widetilde{\pi}$. Thus, the probability in \eqref{eq:EE} is the inverse of the size of the orbit containing $\pi$. In the rest of the proof, we will prove that the set $\Sur(\calG_n \rtimes \Gamma, H \rtimes \Gamma; \xi_n \mapsto h)$ contains only one orbit. After proving that, it immediately follows that \eqref{eq:EE} equals 1, and hence we finish the proof of the theorem.

		Let $\pi_1, \pi_2 \in \Sur(\calG_n \rtimes \Gamma, H \rtimes \Gamma; \xi_n \mapsto h)$, and let $\widetilde{\pi}_1$ and $\widetilde{\pi}_2$ denote the lifts of $\pi_1$ and $\pi_2$ respectively. 
		Let $\rho_i$ (resp. $\widetilde{\rho}_i$) be the restriction of $\pi_i$ (resp. $\widetilde{\pi}_i$) onto $\calG_n$ (resp. $\widetilde{\calG}_n$) for $i=1, 2$. Then $\widetilde{\rho}_i(\xi_n) = \widetilde{\pi}_i(\xi_n)=h$, so by Lemma~\ref{lem:Labute}\eqref{item:Labute-2}, there exists $\widetilde{\varphi} \in \Aut(\widetilde{\calG}_n)$ satisfying \eqref{item:L-c-2} and \eqref{item:L-c-3}. This automorphism $\widetilde{\varphi}$ is not necessarily $\Gamma$-equivariant, however, we will show below that we can modify $\widetilde{\varphi}$ to make it $\Gamma$-equivariant. Let $\varphi\in\Aut(\calG_n)$ be as described in Lemma~\ref{lem:Labute}\eqref{item:L-c-2}. The automorphism $\varphi$ passes the action of $\gamma$ (the non-trivial element of $\Gamma$) on $\calG_n$ to an action of $^{\varphi}\gamma$ which is defined by
		\[
			^{\varphi} \gamma(g) = \varphi\circ \gamma \circ \varphi^{-1}(g) \quad \text{ for $g \in \calG_n$}.
		\]
		Because of $\rho_2 \circ \varphi= \rho_1$ and that $\rho_1, \rho_2$ extend to $\pi_1, \pi_2$, it follows that $\gamma$ and $^{\varphi}\gamma$ both preserve $\ker \rho_2$ and they induce the same action on $H$ via the quotient map $\rho_2$, where the latter is because for any $g\in \calG_n$
		\[
			\rho_2 \circ \varphi \circ \gamma \circ \varphi^{-1}(g)= \rho_1\circ \gamma \circ \varphi^{-1}(g)=\gamma \circ \rho_1 \circ \varphi^{-1}(x) = \gamma \circ \rho_2(x) = \rho_2 \circ \gamma(x).
		\]
		Also, the actions of $\gamma$ and $^{\varphi}\gamma$ induce the same action on $\calG_n^{ab}$ (which is taking inversion). So, we let $M:=[\calG_n, \calG_n] \cap  \ker \rho_2$, and we see that the actions of $\gamma$ and $^{\varphi}\gamma$ induce the same action on $\calG_n/M$. By Lemma~\ref{lem:p-gp-aut}, there exists $\tau \in \Aut(\calG_n)$ such that $^{\varphi} \gamma = \tau^{-1} \circ \gamma \circ \tau$ and $\tau$ preserves $M$ and induces the trivial action on $\calG_n/M$. Thus, $\tau \circ \varphi$ is a $\Gamma$-equivariant automorphism of $\calG_n$ such that $\rho_2 \circ \tau \circ \varphi = \rho_1$. Because $\pi_n: \widetilde{\calG}_n \to \calG_n$ is a Schur covering, by Lemma~\ref{lem:Labute}\eqref{item:Labute-1}, there exists $\widetilde{\tau} \in \Aut(\widetilde{\calG}_n)$ such that $\pi_n \circ \widetilde{\tau} = \tau \circ \pi_n$. Next, we consider the actions on $\widetilde{\calG}_n$, and similarly to above arguments, $\widetilde{\varphi}$ passes the action of $\gamma$ on $\widetilde{\calG}_n$ to an action of $^{\widetilde{\varphi}}\gamma$. Then because $^{\varphi}\gamma=\tau^{-1}\circ \gamma \circ \tau$ as actions on $\calG_n$, we have that $^{\widetilde{\varphi}} \gamma$ and $\widetilde{\tau}^{-1} \circ \gamma \circ \widetilde{\tau}$ are two order-2 automorphisms of $\widetilde{\calG}_n$ that preserve $\ker \pi_n$ and induce the same automorphism on $\calG_n$. So, by Lemma~\ref{lem:p-gp-aut}, there exists $\sigma \in \Aut(\widetilde{\calG}_n)$ such that $^{\widetilde{\varphi}}\gamma = \sigma^{-1} \circ \widetilde{\tau}^{-1} \circ \gamma \circ \widetilde{\tau} \circ \sigma$ and $\sigma$ preserves $\ker \pi_n$ and induces the trivial action on $\calG_n$. Therefore, we obtain an automorphism $\widetilde{\varphi}' :=\widetilde{\tau}\circ \sigma \circ \widetilde{\varphi} \in \widetilde{\calG}_n$ that respects the $\Gamma$ action on $\widetilde{\calG}_n$, so $\widetilde{\varphi}'$ extends to an automorphism of $\widetilde{\calG}_n \rtimes \Gamma$ sending $\gamma \mapsto \gamma$. It's only left to show that $\widetilde{\varphi}' \in \Aut(\widetilde{\calG}_n \rtimes \Gamma, \pi_n ; q)$. Indeed, $\widetilde{\varphi}'$ perserves $\ker \pi_n$ because all of $\widetilde{\tau}, \sigma$ and $\widetilde{\varphi}$ preserve $\ker \pi_n$; and $\widetilde{\varphi}'$ mapping $\xi_n$ to $\xi_n^q$ follows by applying the following claim to $\kappa:=\widetilde{\tau} \circ \sigma$ and by the fact that $\widetilde{\varphi}(\xi_n)=\xi_n^q$. 
		
		\emph{Claim: Suppose that $\kappa$ is an automorphism of $\widetilde{\calG}_n$ that perserves $\ker \pi_n$ and induces the trivial automorphism on $\widetilde{\calG}_n^{ab}$. Then $\kappa(\xi_n)=\xi_n$.}
		
		We use the usual presentation $\calG_n=\langle x_1, \ldots, x_{2n} \mid [x_1,x_2]\cdots [x_{2n-1},x_{2n}] \rangle$, and consider the short exact sequence defined by this presentation
		\[
			1 \longrightarrow R \longrightarrow F_{2n} \longrightarrow \calG_n \longrightarrow 1,
		\]
		where $F_{2n}$ is the free pro-$\ell$ group on $x_1,\ldots, x_{2n}$ and $R$ is the closed normal subgroup generated by $\lambda:=[x_1,x_2]\cdots [x_{2n-1}, x_{2n}]$. Then it gives the map $\pi_n$ via
		\begin{equation}\label{eq:ses-claim}
			1 \longrightarrow R/[R,F_{2n}] \longrightarrow F_{2n}/[R, F_{2n}]\simeq \widetilde{\calG}_n \overset{\pi_n}{\longrightarrow} \calG_n \longrightarrow 1.
		\end{equation}
		The automorphism $\kappa$ of $\widetilde{\calG}_n$ lifts to an automorphism $\widehat{\kappa}$ of $F_{2n}$. Since $\widetilde{\calG}_n^{ab}=\calG_n^{ab}=F_{2n}^{ab}$, $\widehat{\kappa}$ induces the trivial automorphism on $F_{2n}^{ab}$. So the element $y_i:=x_i^{-1}\widehat{\kappa}(x_i)$ is contained in $F_{2n}^{(2)}$ for each $i$. Then $\widehat{\kappa}(\lambda)=[x_1y_1, x_2y_2]\cdots [x_{2n-1}y_{2n-1}, x_{2n}y_{2n}] \equiv [x_1, x_2]\cdots [x_{2n-1}, x_{2n}] \mod F_{2n}^{(3)}$. As $R \subset F_{2n}^{(2)}$, we have that $[R, F_{2n}]\subset F_{2n}^{(3)}$; and then the quotient map $R/[R, F_{2n}] \to R/(R\cap F_{2n}^{(3)})$ is an isomorphism since both of them are isomorphic to $\Z_\ell$. So $\widehat{\kappa}(\lambda)\equiv \lambda \mod F_{2n}^{(3)}$ implies $\kappa$ acts trivially on $\ker \pi_n$, which proves the claim and hence proves the first equality in the theorem.
		
		To prove the second equality, we consider the following expression of the moment that is similar to \eqref{eq:EE}
		\begin{eqnarray}
			&& \EE \left(\# \left\{ \pi \in \Sur_{\Gamma}( Y_{q,n}, H) \mid \pi_{\dagger}=\delta\right\} \right) \nonumber \\
			&=& \sum_{\substack{\pi \in \Sur(\calG_n\rtimes \Gamma, H \rtimes \Gamma) \\ \text{s.t. } \widetilde{\pi}(\xi_n)=h \\ \text{and } \pi|_{\calG_n} \text{ is $\Gamma$-equivariant}}} \Prob ( \phi \in \Aut(\widetilde{\calG}_n \rtimes \Gamma, \pi_n ; q) \mid  \ker (1-\phi) \subseteq \ker \widetilde{\pi} ). \label{eq:EE'}
		\end{eqnarray}
		By the Schur--Zassenhaus theorem, any two homomorphic splittings of $H \times \Gamma \to \Gamma$ are conjugate by an element of $H$, and therefore, the conjugation gives an automorphism of $H \rtimes \Gamma$ that maps one splitting to the other. Also, because the conjugations of $H \rtimes \Gamma$ acts transitively on the set of all splittings, the number of splittings is exactly $[H: H^{\Gamma}]$. So we have 
		\begin{eqnarray*}
			&& \# \{ \pi \in \Sur(\calG_n \rtimes \Gamma, H \rtimes \Gamma) \mid \widetilde{\pi}(\xi_n) = h \} \\
			&=& [H: H^{\Gamma}]  \# \{ \pi \in \Sur(\calG_n \rtimes \Gamma, H \rtimes \Gamma) \mid \widetilde{\pi}(\xi_n) = h  \text{ and  $\pi|_{\calG_n}$ is $\Gamma$-equivariant} \} .
		\end{eqnarray*}
		Thus, the second equality in the theorem follows by the first equality, \eqref{eq:EE} and \eqref{eq:EE'}.
	\end{proof}

\subsection{Proof of Lemma~\ref{lem:Labute}}\label{ss:Labute}

		Let $F_{2n}$ denote the free pro-$\ell$ group on $2n$ generators $x_1, \ldots, x_{2n}$. Then there is a presentation of $\calG_n$
		\begin{equation}\label{eq:pres-Gn}
			1 \longrightarrow R \longrightarrow F_{2n} \overset{\pi}{\longrightarrow} \calG_n \longrightarrow 1,
		\end{equation}
		and it induces a Schur covering of $\calG_n$
		\begin{equation}\label{eq:ind-ses-stem}
			1 \longrightarrow R/[R,F_{2n}] \longrightarrow F_{2n}/[R, F_{2n}] \longrightarrow \calG_n \longrightarrow 1.
		\end{equation}
		It follows by \cite[Chap.11, Thm.2.3(iii) and (iv)(b)]{GroupRep} that $F_{2n}/[R,F_{2n}]$ can be identified with $\widetilde{\calG}_n$.
		We denote the composition of the quotient map $\pi$ in \eqref{eq:pres-Gn} and $\rho$ by $\kappa: F_{2n} \to H$, and denote $N:=\ker \kappa$. For each $i=1, \ldots, 2n$, we denote $h_i:=\kappa(x_i)$ and we pick an element $\widetilde{h}_i \in \varpi^{-1}(h_i)$. Then the map $\widetilde{\kappa}: F_{2n} \to \widetilde{H}$ defined by $x_i \mapsto \widetilde{h}_i$ is a homomorphism since $F_{2n}$ is free, and moreover it is surjective by Lemma~\ref{lem:app-stem}. By \cite[Chap.11, Thm.2.3(iv)(b)]{GroupRep}, $\ker \widetilde{\kappa}$ contains $[N, F_{2n}]$, so $\widetilde{\kappa}$ factors through $\widetilde{\calG}_n$ since $R \subseteq N$ and $[R,F_{2n}]\subseteq [N, F_{2n}]$. Because $\varpi \circ \widetilde{\kappa}= \rho \circ \pi$, $\widetilde{\kappa}$ induces a surjection $\widetilde{\rho}: \widetilde{\calG}_n \to \widetilde{H}$ such that $\varpi \circ \widetilde{\rho} = \rho \circ \pi_n$,  which proves the existence of $\widetilde{\rho}$ in the statement \eqref{item:Labute-1}. 
		\[\begin{tikzcd}
			F_{2n} \arrow["\pi"]{rr} \arrow{drr} \arrow[bend left=30, "\kappa"]{rrr} \arrow[bend right=70, "\widetilde{\kappa}"']{rrrd} && \calG_n \arrow["\rho"]{r} & H \\
			&& \widetilde{\calG}_n \arrow["\pi_n"']{u} \arrow["\widetilde{\rho}"]{r} & \widetilde{H} \arrow["\phi"]{u}
		\end{tikzcd}\]

		Without loss of generality, we assume $\xi_n$ is the image of $[x_1, x_2]\cdots [x_{2n-1}, x_{2n}]$ in $\widetilde{\calG}_n$. Then for any $\widetilde{\rho}$ satisfying the desired requirments, we have that $\widetilde{\rho}$ maps the image of $x_i$ to $\widetilde{h}_i \delta_i$ for some $\delta_i \in \ker \varpi$. However, $\ker \varpi$ is in the center of $\widetilde{H}$, so $[\widetilde{h}_1,  \widetilde{h}_2]\cdots [\widetilde{h}_{2n-1}, \widetilde{h}_{2n}]=[\widetilde{h}_1 \delta_1,  \widetilde{h}_2 \delta_2]\cdots [\widetilde{h}_{2n-1} \delta_{2n-1}, \widetilde{h}_{2n} \delta_{2n}]$. Thus, $\widetilde{\rho}(\xi_n)$ does not depend on the choice of $\widetilde{\rho}$, which finishes the proof of \eqref{item:Labute-1}.
		
		Let $\rho_i$ and $\widetilde{\rho}_i$, $i=1,2$, be as defined in the statement \eqref{item:Labute-2}. We use the surjections from $F_{2n}$ to $\calG_n$ and $\widetilde{\calG}_n$ defined in \eqref{eq:pres-Gn} and \eqref{eq:ind-ses-stem}. Taking compositions with $\rho_i$ and $\widetilde{\rho}_i$, we obtain surjections $\kappa_i: F_{2n} \to H$ and $\widetilde{\kappa}_i: F_{2n} \to \widetilde{H}$ such that $\varpi \circ \widetilde{\kappa}_i = \kappa_i$ for each $i$. For an ordered $2n$-tuple $\mathbf{z}=(z_1, \ldots, z_{2n}) \in F_{2n}^{2n}$, we denote $\lambda(\mathbf{z}):=[z_1, z_2]\cdots [z_{2n-1}, z_{2n}]$. We write $\mathbf{x}$ for the tuple of generators $(x_1, \ldots, x_{2n})$. We will prove the statement~\eqref{item:Labute-2} in the lemma by proving the following claim.
		
		\emph{ Claim: There exists $\theta \in \Aut(F_{2n})$ such that $\kappa_1 = \kappa_2 \circ \theta$ and $\lambda(\theta(\mathbf{x}))=\lambda(\mathbf{x})^q$. Here $\theta(\mathbf{x})$ denotes the tuple $(\theta(x_1), \ldots, \theta(x_{2n}))$.
		}
		
		We now explain how the claim implies the statement~\eqref{item:Labute-2}. In the presentation \eqref{eq:pres-Gn} of $\calG_n$, the kernel $R$ is the closed normal subgroup generated by $\lambda(\mathbf{x})$. So if $\lambda(\theta(\mathbf{x}))=\lambda(\mathbf{x})^q$, then $R$ is also generated by $\lambda(\theta(\mathbf{x}))$ as $q\in \Z_{\ell}^{\times}$. It follows that $\theta(R)=R$, and hence $\theta$ induces $\varphi \in \Aut(\calG_n)$, and also $\widetilde{\varphi} \in \Aut(\widetilde{\calG}_n)$ by the description of $\widetilde{\calG}_n$ in \eqref{eq:ind-ses-stem}. The automorphisms $\varphi$ and $\widetilde{\varphi}$ satisfy the condition \eqref{item:L-c-2} because $\kappa_1 = \kappa_2 \circ \theta$ and $\kappa_i$ factors through $\widetilde{\rho}_i$ and $\rho_i$ for each $i=1,2$. And the condition~\eqref{item:L-c-3} is satisfied, because $\xi_n$ is assumed to be the image of $\lambda(\mathbf{x})$ and then $\widetilde{\varphi}(\xi_n)$ is the image of $\lambda(\theta(\mathbf{x}))$.
		
		In the remainder of the proof, we will prove the claim by induction on the lower $\ell$-central series $\{F_{2n}^{\{j\}}\}$ of $F_{2n}$, which is defined as
		\[
			F_{2n}^{\{1\}}=F_{2n} \quad \text{and} \quad F_{2n}^{\{n+1\}}:=[F_{2n},  F_{2n}^{\{n\}}](F_{2n}^{\{n\}})^{\ell}.
		\]
		Explicitly, we will show the following.
		\begin{description}
			\item[($\ast$)]\label{item:induction} For each $j \in \Z_{>1}$, there exists a tuple $\mathbf{y}^{\{j\}}=(y_1^{\{j\}}, \ldots, y_{2n}^{\{j\}})$ formed by a set of generators $\{y_i^{\{j\}}\}_{i=1}^{2n}$ of $F_{2n}$ such that $\lambda(\mathbf{y}^{\{j\}})\equiv \lambda(\mathbf{x})^q \mod F_{2n}^{\{j+1\}}$ and $\kappa_1(x_i) = \kappa_2(y_i^{\{j\}})$ for each $i=1, \ldots, 2n$. Moreover, $y_i^{\{j+1\}} \equiv y_i^{\{j\}} \mod F_{2n}^{\{j\}}$.
		\end{description}
		Then the map sending $x_i$ to the inverse limit of $\{y_i^{\{j\}} \mod F_{2n}^{\{j\}}\}$ as $j \to \infty$ defines an automorphism $\theta$ satisfying all conditions in the claim.

		\begin{lemma}
			For $j=2$, the tuple $\mathbf{y}^{\{2\}}$ as in ($\ast$) exists.
		\end{lemma}
			
		\begin{proof}
			Note that $F_{2n}/F_{2n}^{\{2\}}$ is the Frattini quotient of $F_{2n}$, so it is the dual of $H^1(F_{2n}, \F_{\ell})$. The element $\lambda(\mathbf{x})$ defines a symplectic pairing on $H^1(F_{2n}, \F_{\ell})\simeq H^1(\calG_n, \F_{\ell})$
			\[\begin{tikzcd}
				(- ,- ) : H^1(F_{2n}, \F_{\ell}) \times H^1(F_{2n}, \F_{\ell}) \arrow["\cup"]{r} & H^2(\calG_n, \F_{\ell}) \arrow["\tr_{\lambda(\mathbf{x})}"]{r} &\F_{\ell},
			\end{tikzcd}\]
			where $\tr_{\lambda(\mathbf{x})}$ is the trace map induced by $\lambda(\mathbf{x})$. For each $r=1,2$, the surjection $\kappa_r$ induces inflation map $\kappa_r^{*}: H^1(H, \F_{\ell}) \hookrightarrow H^1(F_{2n}, \F_{\ell})$, and then we obtain a bilinear and alternating pairing on $H^1(H, \F_{\ell})$.
			By the assumption that $\widetilde{\kappa}_1(\lambda(\mathbf{x}))= \widetilde{\kappa}_2(\lambda(\mathbf{x}))$ has order dividing $q-1$ and the argument similar to Proposition~\ref{prop:omega-weil}, the two pairings defined by $\kappa_1^{*}$ and $\kappa_2^{*}$ are the same and $(\alpha, \alpha')=q(\alpha, \alpha')$ for any $\alpha, \alpha' \in H^1(H, \F_\ell)$.
			So we can choose a basis $e_1, \ldots, e_{a+b}, f_1, \ldots, f_{a}$ of $H^1(H, \F_{\ell})$ for some $a, b \geq 0$, such that $(e_i, e_j)=(f_i, f_j)=0$ for all $i,j$ and $(e_i, f_j)=1$ if $i=j$ and $0$ otherwise. By the theory of symplectic space, we can extend this basis to a symplectic basis $e_{1,1}, \ldots, e_{n,1}, f_{1,1}, \ldots, f_{n,1}$ of $H^1(F_{2n}, \F_{\ell})$ via the embedding $\kappa_1^{*}$: that is, a basis such that $(e_{i,1}, f_{i,1})=1$ and all the other pairings of elements in the basis are zero, and moreover
			 $e_{i,1}=\kappa_1^*(e_{i})$ and $f_{j,1}=\kappa_1^*(f_j)$ for $1 \leq i \leq a+b$ and $1\leq j \leq a$. Similarly, we extend this basis to a $q$-symplectic basis $e_{1,2}, \ldots, e_{n,2}, f_{1,2}, \ldots, f_{n,2}$ of $H^1(F_{2n}, \F_{\ell})$ via the embedding $\kappa_2^{*}$ such that $(e_{i,1}, f_{i,1})=q$ for all $i$, and $e_{i,2}=\kappa_2^*(e_{i})$ and $f_{j,2}=\kappa_2^*(f_j)$ for $1 \leq i \leq a+b$ and $1\leq j \leq a$.

			Then the automorphism $\beta \in \Aut(H^1(F_{2n}, \F_\ell))$ sending $e_{i,1}\mapsto e_{i,2}$ and $f_{i,1} \mapsto f_{i,2}$ satisfies $\beta \circ \kappa_1^* = \kappa_2^*$ and $(\alpha, \alpha')=q(\beta(\alpha), \beta(\alpha'))$ for any $\alpha, \alpha' \in H^1(F_{2n}, \F_{\ell})$. So by taking the dual, we obtain $\beta^{\vee} \in \Aut( F_{2n}/F_{2n}^{\{2\}})$ making the following diagram commute, where $\kappa_r^{\{2\}}$ is the quotient map of the Frattini quotients induced by $\kappa_r$.
			\[\begin{tikzcd}
				F_{2n}/F_{2n}^{\{2\}} \arrow["\kappa_2^{\{2\}}"]{r} \arrow["\beta^{\vee}", "\sim"', swap]{d} & H/H^{\{2\}} \\
				F_{2n}/F_{2n}^{\{2\}} \arrow["\kappa_1^{\{2\}}", swap]{ru} &
			\end{tikzcd}\]
		Consider the fiber product $P:=H \times_{H/H^{\{2\}}} F_{2n}/F_{2n}^{\{2\}}$ defined by the natural surjection $H \to H/H^{\{2\}}$ and $\kappa_2^{\{2\}}$. For each $i=1, \ldots, 2n$, we let $z_i$ denote the element in $P$ whose first coordinate is $\kappa_1(x_i)$ and the second coordinate is the preimage under $\beta^{\vee}$ of the image of $x_i$ in $F_{2n}/F_{2n}^{\{2\}}$ (these two coordinates have the same image in $H/H^{\{2\}}$ because of the diagram above, so $z_i$ is a well-defined element in $P$). Since $\kappa_2$ factors through a surjection $\kappa_2^{\wedge} : F_{2n} \to P$ obtained by quotienting $F_{2n}$ by $\ker \kappa_2 \cap F_{2n}^{\{2\}}$, by \cite[Lem.2.1]{Lubotzky}, there exist elements $y_1, \ldots, y_{2n}$ in $F_{2n}$ that generate $F_{2n}$ and $\kappa_2^{\wedge}(y_i)=z_i$. It follows that $\kappa_1(x_i)=\kappa_2(y_i)$ and $\beta^{\vee}$ sends the image of $y_i$ to the image of $x_i$. Moreover, letting $\chi_i$ and $\chi'_i$ denote the dual of $x_i$ and $y_i$ respectively in $H^1(F_{2n}, \F_{\ell})$, we have $\beta(\chi_i)=\chi'_i$, and hence $(\chi_i, \chi_j)=q(\chi'_i, \chi'_j)$. Therefore, by \cite[Prop.(3.9.13)(ii)]{NSW}, we have $\lambda(\mathbf{x})^q\equiv [y_1, y_2]\cdots [y_{2n-1}, y_{2n}] \mod F_{2n}^{\{3\}}$, so the tuple $\mathbf{y}^{\{2\}}:=(y_1, \ldots, y_{2n})$ satisfies ($\ast$) for $j=2$.
		\end{proof}

		\begin{lemma}
			Suppose that there exists $\mathbf{y}^{\{j\}}$ in ($\ast$) for some $j>1$. Then $\mathbf{y}^{\{j+1\}}$ exists.
		\end{lemma}
		
		\begin{proof}

			We will prove that there exist $t_1, \ldots, t_{2n} \in F_{2n}^{\{j\}} \cap \ker \kappa_2$ such that the tuple $\mathbf{y}^{\{j+1\}}$ obtained by setting $y_i^{\{j+1\}}$ to be $y_i^{\{j\}}t_i$ satisfies ($\ast$) . First, it is automatic that $y_i^{\{j+1\}} \equiv y_i^{\{j\}}$ mod $F_{2n}^{\{j\}}$ and $\kappa_1(x_i)=\kappa_2(y_i^{\{j+1\}})$ as $t_i \in F_{2n}^{\{j\}} \cap \ker \kappa_2$. So it suffices to consider the requirement $\lambda(\mathbf{y}^{\{j+1\}}) \equiv \lambda(\mathbf{x})^q \mod F_{2n}^{\{j+2\}}$. Let's denote $D:=\lambda(\mathbf{x})^{-q} \lambda(\mathbf{y}^{\{j\}})$. Since $\lambda(\mathbf{x})$ and $\lambda(\mathbf{y}^{\{j\}})$ are both contained in $[F_{2n}, F_{2n}]$, we have $D \in F_{2n}^{\{j+1\}} \cap [F_{2n}, F_{2n}]$.
			Also, assuming that $\xi_n$ is the image of $\lambda(\mathbf{x})$, we have
			\[
				\widetilde{\kappa}_2(\lambda(\mathbf{x})^q)= \widetilde{\rho}_2(\xi_n)^q = \widetilde{\rho}_1(\xi_n) = \widetilde{\kappa}_1(\lambda(\mathbf{x}))=\widetilde{\kappa}_2(\lambda(\mathbf{y}^{\{j\}})),
			\]
			where the first and the third equalities follow by definition of $\widetilde{\kappa}_1$ and $\widetilde{\kappa}_2$, the second one follows by the assumption in the proposition that $\widetilde{\rho}_1(\xi_n)=\widetilde{\rho}_2(\xi_n)$ has order dividing $q-1$, and the last one is because of the statement \eqref{item:Labute-1} in the proposition and the induction hypothesis $\kappa_1(x_i)=\kappa_2(y_i^{\{j\}})$. So we have that $D$ is contained in $F_{2n}^{\{j+1\}} \cap [F_{2n}, F_{2n}] \cap \ker \widetilde{\kappa}_2$.
			
			The lower $\ell$-central series $\{F_{2n}^{\{j\}}\}$ induces a Lie algebra $\gr(F_{2n}):= \sum_i \gr_i(F_{2n})$ over $\F_\ell$, where $\gr_i(F_{2n}):= F_{2n}^{\{j\}} / F_{2n}^{\{j+1\}}$ and the Lie bracket for homogenous elements of $\gr(F_{2n})$ is induced by the commutator: if $a=\overline{x}\in \gr_r(F_{2n})$ and $b=\overline{y} \in \gr_s(F_{2n})$, then $[a,b]$ is the image of the commutator $[x,y]$ in $\gr_{r+s}(F_{2n})$. So
			\begin{eqnarray*}
				\lambda(\mathbf{y}^{\{j+1\}})&=& [y_1^{\{j+1\}}, y_2^{\{j+1\}}]\cdots [y_{2n-1}^{\{j+1\}}, y_{2n}^{\{j+1\}}] \\
				&=& [y_1^{\{j\}}t_1, y_2^{\{j\}}t_2]\cdots [y_{2n-1}^{\{j\}} t_{2n-1}, y_{2n}^{\{j\}} t_{2n}] \\
				&\equiv& \lambda(\mathbf{y}^{\{j\}}) [t_1, y_2^{\{j\}}][y_1^{\{j\}}, t_2] \cdots  [t_{2n-1}, y_{2n}^{\{j\}}] [y_{2n-1}^{\{j\}}, t_{2n}] \mod F_{2n}^{\{j+2\}}.
			\end{eqnarray*}
			We write the group action of $\gr_i(F_{2n})$ additively, and define
			\[
				d_j(t_1, \ldots, t_{2n}):= [\tau_1, \sigma_2] + [\tau_2, \sigma_1] + \cdots + [\tau_{2n-1}, \sigma_{2n}] + [\tau_{2n}, \sigma_{2n-1}] \in \gr_{j+1}(F_{2n}),
			\]
			where $\tau_i$ is the image of $t_i$ in $\gr_{j}(F_{2n})$ and $\sigma_i$ is the image of $y_i^{\{j\}}$ in $\gr_{1}(F_{2n})$.
			
			Considering the short exact sequence $1 \longrightarrow N \longrightarrow F_{2n} \overset{\kappa_2}{\longrightarrow} H \longrightarrow 1$ defined by $\kappa_2$, by \cite[Chap.11, Thm.2.3(iii) and (iv)(b)]{GroupRep}, $[N, F_{2n}]$ is $[F_{2n}, F_{2n}] \cap \ker \widetilde{\kappa}_2$. So, the image of $F_{2n}^{\{j+1\}} \cap [F_{2n}, F_{2n}] \cap \ker \widetilde{\kappa}_2=F_{2n}^{\{j+1\}} \cap [N, F_{2n}]$ in $\gr_{j+1}(F_{2n})$ is generated by $[a,b]$ for some $a\in \gr_r(F_{2n})$ and $b \in \gr_{j+1-r}(F_{2n})$ such that $a$ is the image of some element of $N$. If $b=c^{\ell}$ for some $c \in \gr_{j-r}(F_{2n})$, then $[a,b]=[a^\ell,c]$ and $a^{\ell}$ is the image in $\gr_{r+1}(F_{2n})$ of some element of $N$. If $b=[c,d]$, then $[a,b]=[[d,a], c]+[[a,c],d]$ and each of $[d,a]$ and $[a,c]$ is the image of some element of $N$ and of degree strictly larger than $r$. It follows that the image of $F_{2n}^{\{j+1\}} \cap [N, F_{2n}]$ in $\gr_{j+1}(F_{2n})$ is generated by $[a,b]$ where $a \in \gr_j(F_{2n})$, $b\in \gr_1(F_{2n})$, and $a$ is the image of some element of $N$. Because $y_1^{\{j\}}, \ldots, y_{2n}^{\{j\}}$ form a generating set of $F_{2n}$, we have that $\sigma_1, \ldots, \sigma_{2n}$ generate $\gr_1(F_{2n})$, and hence the argument above shows that there exist $t_1, \ldots, t_{2n} \in F_{2n}^{\{j\}} \cap \ker \kappa_2$ such that $d_j(t_1, \ldots, t_{2n})$ equals the image of $D$ in $\gr_{j+1}(F_{2n})$. Then for such a choice of $t_1, \ldots, t_{2n}$, we have the tuple $\mathbf{y}^{\{j+1\}}$ satisfying $\lambda(\mathbf{y}^{\{j+1\}}) \equiv \lambda(\mathbf{x})^q$, and hence we finish the proof.
		\end{proof}

\section{Conjectures and Evidences}\label{sect:conjectures}
 
 	In \S\ref{ss:statements}, we give the statements of all of the conjectures made in this paper, which include the moment conjectures in both the function field case and the number field case, and also the conjectures regarding the existence and uniqueness of the probability measure with desired moments. Then in \S\ref{ss:evidence1} and \S\ref{ss:evidence2}, we discuss the presentation of $G_{\O}(K)$ and positively ramifications, and how they support the moment conjecture in the number field case.

\subsection{Statements of the conjectures}\label{ss:statements}
	
	We first recall the notion of pro-$\calC$ completion of a $\Gamma$-group defined in \cite{LWZB}.
	Let $\calC$ be a set of isomorphism classes of finite $\Gamma$-groups, and define $\overline{\calC}$ to be the smallest set of isomorphism classes of $\Gamma$-groups containing $\calC$ that is closed under taking finite direct product, $\Gamma$-quotient and $\Gamma$-subgroup. Then for a given profinite $\Gamma$-group $G$, the \emph{pro-$\calC$ completion of $G$} is defined to be 
	\[
		G^{\calC} = \varprojlim_{M} G/M,
	\]
	where the inverse limit runs over all closed normal $\Gamma$-subgroup $M$ of $G$ such that $G/M \in \overline{\calC}$. We say that a group $G$ is \emph{pro-$\calC$} if $G=G^{\calC}$.

	Let $q$ be a positive integer such that $\gcd(q, |\Gamma|)=1$. 
	We consider pairs $(H, h)$ for a profinite admissible $q'$-$\Gamma$-group $H$ and an element $h \in H_2(H \rtimes \Gamma, \Z)_{(|\Gamma|)'}[q-1]$. We say that two such pairs $(G, g)$ and $(H, h)$ are isomorphic if there exists an $\Gamma$-equivariant isomorphism $G \simeq H$ such that its induced coinflation map $H_2(G\rtimes \Gamma, \Z)_{(|\Gamma|)'} \to H_2(H \rtimes \Gamma, \Z)_{(|\Gamma|)'}$ sends $g \mapsto h$.
	We define a topological space $\calP_{\Gamma, q}$, whose underlying space is the set of isomorphism classes of pairs $(H, h)$ as described above and basic opens are the sets 
	\[
		U_{\calC, H, h}:= \left\{(G,g) \in \calP_{\Gamma,q} \,\Bigg| \,\begin{aligned} & \text{there is a $\Gamma$-equivariant surjection $e: G \to H $ s.t. } \\ & \text{$e$ induces an isom.  $G^{\calC} \simeq H$ and the coinf map sending $g \mapsto h$}\end{aligned}\right\},
	\] 
	for each finite set $\calC$ of finite $(q|\Gamma|)'$-$\Gamma$-groups and each pair $(H, h)\in \calP_{\Gamma,q}$ such that $H$ is finite pro-$\calC$.
	For a measure $\mu$ on $\calP_{\Gamma,q}$, we use the following notation to denote the average number of surjections to a pair $(H,h)$
	\[
		\EE_{\mu} \# \Sur_{\Gamma}(*, (H, h)):= \sum_{(G,g) \in \calP_{\Gamma,q}}  \# \left\{ \pi \in \Sur_{\Gamma}(G, H) \, \bigg| \, \begin{aligned} &\text{the $\coinf$ map induced } \\ & \text{by $\pi$ sends $g \mapsto h$} \end{aligned} \right\} {\mu}(G,g).
	\]

	\begin{conjecture}[Existence and uniqueness of the limit probability measure with the desired moments]\label{conj:prob-exist}
		For a finite group $\Gamma$ and a positive integer $q$ satisfying $\gcd(q, |\Gamma|)=1$, there exists a unique probability measure $\mu_{\Gamma, q}$ on the topological space $\calP_{\Gamma,q}$, such that for each $(H,h) \in \calP_{\Gamma,q}$, the following equality holds 
		\[
			\EE_{\mu_{\Gamma,q}} \# \Sur_{\Gamma}(*, (H, h)) = \frac{1}{[H: H^{\Gamma}]}.
		\]
	\end{conjecture}

	When $\Gamma=\Z/2\Z$, we've discussed in Section~\ref{sect:RandomGroup} that there is a random pro-$\ell$ group model $Y_{q,n}$ giving the desired moments. Next, we define a coarser topological space $\calP_{\Gamma,q, \ell}$ on which our pro-$\ell$ version of the probability measure is defined, and then we conjecture that the limit of the random group model $Y_{q,n}$ as $n \to \infty$ gives the desired probability measure.

	\begin{definition}[Definition of measures $\nu_{q,n, \ell}$ on pro-$\ell$ $\Z/2\Z$-groups]\label{def:measure}
		Let $\Gamma=\Z/2\Z$ and $\ell$ an odd prime. 
		We define a coarser topology $\calP_{\Gamma, q, \ell}$ on the underlying space of $\calP_{\Gamma, q}$: where the basic opens are all $U_{\calC, H, h}$ with an additional condition that $H$ is an $\ell$-group. 
		For fixed $q, n$ satisfying $\gcd(2\ell, q)=1$, we use the notation in Theorem~\ref{thm:moment}, fix a generator $\xi_n$ of $H_2(\calG_n, \Z)(\ell)$, and define a probability measure $\nu_{q, \ell, n}$ on the space $\calP_{\Gamma,q, \ell}$ by
	\[
		\nu_{q, \ell, n}(U_{\calC,H, h}):= \Prob \left( \begin{aligned} &\text{there exists $\pi \in \Sur_{\Gamma}(Y_{q,n}, H)$ s.t. }\\ & \text{$\pi$ induces an isom. $Y^{\calC}_{q,n} \simeq H$ and $\pi_\dagger(\xi_n)=h$} \end{aligned} \right),
	\]
	for each finite set $\calC$ of $\ell$-$\Gamma$-groups and pair $(H, h) \in \calP_{\Gamma,q, \ell}$ such that $H$ is finite pro-$\calC$.
	\end{definition}

	\begin{conjecture}\label{conj:limit-measure}
		When $\Gamma=\Z/2\Z$, on the topological space $\calP_{\Gamma, q, \ell}$,
		the probability measures $\nu_{q, \ell, n}$ have weak convergence to the measure $\mu_{\Gamma, q}$ in Conjecture~\ref{conj:prob-exist}, as $n  \to \infty$.
	\end{conjecture}
	
	Next, we give the probability and moment conjecture in the function field case.

	\begin{conjecture}[Function field case]\label{conj:ff}
		Let $\Gamma$ be a finite group and $q$ a prime power such that $\gcd(q, |\Gamma|)=1$. Let $\calC$ be a finite set of finite $(q|\Gamma|)'$-$\Gamma$-groups, and $(H,h)$ a pair in $\calP_{\Gamma, q}$ such that $H$ is finite pro-$\calC$. Fix a generator $\zeta$ of $\hZ(1)_{(q|\Gamma|)'}$, and we let $\delta$ denote the homomorphism $\hZ(1)_{(q|\Gamma|)'} \to H_2(H \rtimes \Gamma, \Z)_{(|\Gamma|)'}$ sending $\zeta \mapsto h$. Then 
		\[
			\lim_{N \to \infty} \frac{\sum\limits_{n \leq N} \#\left\{ K \in E_{\Gamma}(q^n, \F_q(t))\, \bigg|\, \begin{aligned} & \text{ there exists $\pi  \in \Sur_{\Gamma}(\Gs(K),H)$ s.t.} \\& \text{ $\pi$ induces an isom. $\Gs(K)^{\calC} \simeq H$ and $\pi_* \circ \omega_K^{\#}=\delta$} \end{aligned} \right\}}{ \sum\limits_{n\leq N} \# E_{\Gamma}(q^n, \F_q(t))} = \mu_{\Gamma,q}(U_{\calC, H, h}).
		\]
		and
		\[
			\lim_{N \to \infty} \frac{\sum\limits_{n \leq N} \sum\limits_{K \in E_{\Gamma}(q^n, \F_q(t))} \#\left\{ \pi \in \Sur_{\Gamma}\left(\Gs(K), H\right) \,\bigg|\, \pi_{*} \circ \omega_K^{\#} = \delta \right\}}{\sum\limits_{n \leq N} \# E_{\Gamma}(q^n, \F_q(t))} = \frac{1}{[H: H^{\Gamma}]}.
		\]
	\end{conjecture}

	To give the moment conjecture in the number field case, we need to first study how the action of taking quotient modulo $u$ random elements changes the moments. This idea comes from \cite{Lipnowski-Sawin-Tsimerman}, and Lemma~\ref{lem:Qt-moment} below is the analogue of \cite[Lem.8.12]{Lipnowski-Sawin-Tsimerman} for our notion of the probability space.  

	\begin{definition}\label{def:Qt-measure}
		Let $\mu$ be a measure on $\calP_{\Gamma, q}$. For a positive integer $t$, define the measure $\calQ^t\mu$ on $\calP_{\Gamma,q}$ as obtained by taking the $\Gamma$-equivariant quotient of the random pair $(H,h)\in \calP_{\Gamma,q}$ according to $\mu$ modulo $t$ random elements. That is, for any $(H,h)$, define
		\[
			\calQ^t\mu(H, h) := \sum_{(G,g)\in \calP_{\Gamma,q}}  \Prob\left( G/[x_1, \ldots, x_t]  \simeq H \text{ such that } g \mapsto h\right) \mu(G,g),
		\] 
		where $x_1, \ldots, x_t$ are $t$ random elements of $G$ with respect to Haar measure, and the notation $[x_1, \ldots, x_t]$ represents the closed normal $\Gamma$-subgroup of $G$ generated by $x_1, \ldots, x_t$.
	\end{definition}

	\begin{lemma}\label{lem:Qt-moment}
		Let $\mu$ be a measure on $\calP_{\Gamma, q}$. Then for any $(H,h)\in \calP_{\Gamma,q}$, we have
		\[
			\EE_{\calQ^t\mu} \# \Sur_{\Gamma}(*, (H,h)) = \frac{\EE_{\mu} \# \Sur_{\Gamma}(*, (H,h))}{|H|^t}.
		\]
	\end{lemma}
	
	\begin{proof}
		It suffices to prove the lemma for $t=1$. For two pairs $(G_1, g_1), (G_2, g_2) \in \calP_{\Gamma, q}$, we write $\pi: (G_1, g_1) \to (G_2, g_2)$ if $\pi$ is a $\Gamma$-equivariant group homomorphism $G_1 \to G_2$ whose induced coinflation maps $g_1$ to $g_2$, and we write $\Sur_{\Gamma}((G_1, g_1), (G_2, g_2))$ for the set of surjections $(G_1, g_1) \to (G_2, g_2)$. Then we have
	\begin{eqnarray*}
			&& \EE_{\calQ\mu} \# \Sur_{\Gamma}(*, (H, h)) \\
			&=& \sum_{(G,g) \in \calP_{\Gamma,q}} \# \Sur_{\Gamma}\left((G,g), (H, h)\right) \cdot \calQ\mu(G,g) \\
			&=& \sum_{(G,g) \in \calP_{\Gamma,q}} \sum_{(E,e) \in \calP_{\Gamma,q}} \#\Sur_{\Gamma}\left((G,g), (H,h)\right) \cdot \Prob\left(E/[x] \simeq G \text{ s.t. } e\mapsto g\right) \cdot \mu(E,e) \\
			&=& \sum_{(E,e) \in \calP_{\Gamma, q}} \sum_{\pi \in \Sur_{\Gamma}\left( (E,e), (H,h)\right)} \Prob\left(\text{the random $x\in E$ is contained in $\ker \pi$}\right) \cdot \mu(E,e) \\
			&=& \sum_{(E,e) \in \calP_{\Gamma,q}} \# \Sur_{\Gamma}\left((E,e), (H,h) \right)\cdot \frac{1}{|H|} \cdot \mu(E,e) \\
			&=& \frac{\EE_{\mu} \#\Sur_{\Gamma}(*, (H,h))}{|H|}.
		\end{eqnarray*}		
	\end{proof}

	\begin{conjecture}[Number field case]\label{conj:nf}
		Let $\Gamma$ be a finite group, $Q$ a number field, and $m$ the maximal integer such that the $m$-th roots of unity $\mu_m$ is contained in $Q$. Let $G_{\O}^{\#}(K):=G_{\O}(K)_{(|\Gamma||G_{\O}(Q)|)'}$, and $H$ be a finite admissible $\Gamma$-group such that $G_{\O}(Q)$ is pro-prime-to $|H|$. Then, for any extension $K/Q$, 
		\[
			\lim_{B \to \infty} \frac{\sum_{D \leq B} \sum_{K \in E_{\Gamma}(D,Q)} \# \Sur_{\Gamma}(G^{\#}_{\O}(K), H)}{ \sum_{D \leq B} \# E_{\Gamma}(D,Q)}=\frac{\#H_2(H \rtimes\Gamma, \Z)_{(|\Gamma|)'}[m]}{[H:H^{\Gamma}] \cdot |H|^u},
		\]
		where $u$ is the rank of the group of units in the ring $\calO_Q$ of algebraic integers of $Q$. 
	\end{conjecture}
	
	\begin{remark}
	\begin{enumerate}

		\item Although we will see in \S\ref{ss:evidence2} that, for some number field $K$, we can define a lifting invariant that is analogous to the invariant $\omega_K^{\#}$ in the function field case, in general, we don't know how to define such an invariant, so in Conjecture~\ref{conj:nf}, we do not specify the ``lifting invariant'' and study only the distribution of $G_{\O}^{\#}(K)$.
		
		\item We only make the conjecture in the moment version because in the number field case the characteristic is zero and then the distribution of $G^{\#}_{\O}(K)$ is not a measure on $\calP_{\Gamma, q}$ for any $q$. However, one can similarly define a topology on the set of pairs $(H,h)$ where $H$ is an admissible $\Gamma$-group and $h \in H_2(H \rtimes \Gamma, \Z)_{(|\Gamma|)'}[m]$, and conjecture that there exists a unique probability measure on this topological space that gives the desired moments. 
	\end{enumerate}
	\end{remark}

\subsection{Evidence I: Presentations of $G_{\O}(K)$}\label{ss:evidence1} 

	In this section, we fix a finite group $\Gamma$ and let $Q$ denote a global field that is either $\F_q(t)$ or a number field.

	We let $F_n(\Gamma)$ denote \emph{the free $|\Gamma|'$-$\Gamma$-group on $n$ generators} \cite[\S3.1]{LWZB}, that is: $F_n(\Gamma)$ is the free pro-$|\Gamma|'$ group on $\{x_{i,\gamma} \mid i =1 , \ldots n \text{ and } \gamma \in \Gamma\}$ where $\sigma \in \Gamma$ acts on $F_n(\Gamma)$ by $\sigma(x_{i, \gamma})=x_{i,\sigma \gamma}$. Let $G$ be a finitely generated $|\Gamma|'$-$\Gamma$-group. For a sufficiently large $n$, there exists a $\Gamma$-equivariant short exact sequence of $\Gamma$-groups
	\[
		1 \longrightarrow N \longrightarrow F_n(\Gamma) \overset{\pi}{\longrightarrow} G \longrightarrow 1,
	\]
	and we call such a short exact sequence \emph{a $\Gamma$-presentation of $G$}. For a set $\calC$ of isomorphism classes of $|\Gamma|'$-$\Gamma$-groups, if $G$ is pro-$\calC$, then the map $\pi$ factors through the pro-$\calC$ completion map $F_n(\Gamma) \to F_n(\Gamma)^{\calC}$. We call the induced short exact sequence
	\[
		1 \longrightarrow M \longrightarrow F_n(\Gamma)^{\calC} \overset{\pi^{\calC}}{\longrightarrow} G \longrightarrow 1
	\]
	\emph{a pro-$\calC$ $\Gamma$-presentation of $G$}.

	For the rest of this section, we assume $\calC$ is finite. For any $\Gamma$-extension $K$ of $Q$ that is split completely at $\infty$ or at all archimedean primes, by \cite[Thm.6.4]{Liu2020}, the pro-$\calC$ completion of $G_{\O}(K)$ is finite, so $\Gs(K)^{\calC}$ is finite because $\Gs(K)$ is a quotient of $G_{\O}(K)$ in both the number field case and the function field case. Then we can study a pro-$\calC$ $\Gamma$-presentation
	\begin{equation}\label{eq:C-Gur}
		1 \longrightarrow M \longrightarrow F_n(\Gamma)^{\calC} \overset{\pi^{\calC}}{\longrightarrow} \Gs(K)^{\calC} \longrightarrow 1.
	\end{equation}
	The work \cite{Liu2020} studies the minimal number of elements that generates $M$ as a $\Gamma$-closed normal subgroup of $F_n(\Gamma)^{\calC}$ in such a presentation, and gives an upper bound for this minimal number using Galois cohomology groups $H^i(G_{\O}(K), A)$ for $i=1,2$ and $A$ a finite simple $\Gal(K_{\O}/Q)$-module. We will apply that result to number field case and function field case respectively, and show that the number field case needs $u$ more elements than the function field case to generate $M$, where $u$ is the rank of $\calO_Q^{\times}$ as defined in Conjecture~\ref{conj:nf}.
		
	We first recall the definition of multiplicities in \cite[Def.3.1]{Liu2020}. Given a short exact sequence $1 \to N \to F \overset{\omega}{\longrightarrow} H \to 1$ of $\Gamma$-groups, we let $N_0$ be the intersection of all maximal proper $F \rtimes \Gamma$-normal subgroups of $N$, and denote $\overline{N}=N/N_0$ and $\overline{F}=F/N_0$. Then $\overline{N}$ is a direct sum of finite irreducible $\overline{F} \rtimes \Gamma$-groups. For any finite irreducible $\overline{F} \rtimes \Gamma$-group $A$, we define the \emph{multiplicity $m(\omega, \Gamma, H, A)$} to be the multiplicity of $A$ appearing in $\overline{N}$.

	\begin{lemma}\label{lem:multiplicity}
		Let $Q$ be $\F_q(t)$ with $\gcd(q, |\Gamma|)=1$ or a number field. Let $K$ be a $\Gamma$-extension of $Q$ such that each archimedean prime is unramified when $Q$ is a number field, and $\infty$ splits completely when $Q=\F_q(t)$. Let $\calC$ be a finite set of $|\Gamma|'$-$\Gamma$-groups when $Q=\F_q(t)$
		, and a finite set of $(|\Gamma||G_{\O}(Q)|)'$-$\Gamma$-groups when $Q$ is a number field. Let $\ell$ be a prime such that $\ell \neq \Char(Q)$, $\ell \nmid |\Gamma|$, and $\ell \nmid |G_{\O}(Q)|$ when $Q$ is a number field.
		We consider a pro-$\calC$ $\Gamma$-presentation \eqref{eq:C-Gur} of $\Gs(K)^{\calC}$. 
		
		Then for a finite simple $\F_{\ell}[ \Gs(K)^{\calC}\rtimes \Gamma ]$-module $A$, we have
		\[
			 m(\pi^{\calC}, \Gamma, \Gs(K)^{\calC}, A) 
			\leq \begin{cases}
				\dfrac{(n+1)\dim_{\F_\ell} A - \xi(A)-{\bf{1}}_{\F_\ell}(A)}{\dim_{\F_\ell} \Hom_{ \Gs(K)^{\calC}\rtimes \Gamma}(A)} & \text{ if $Q$ is a function field} \\
				\dfrac{(n+u+1)\dim_{\F_\ell} A - \xi(A)-{\bf{1}}_{\F_\ell}(A)}{\dim_{\F_\ell} \Hom_{ \Gs(K)^{\calC}\rtimes \Gamma}(A)} & \text{ if $Q$ is a number field}.
			\end{cases}
		\]
		Here $\xi(A):=\dim_{\F_\ell} A^{\Gamma}/A^{ \Gs(K)^{\calC}\rtimes \Gamma}$; ${\bf{1}}_{\F_\ell}(A)$ is 1 if $A=\F_\ell$ and $\mu_{\ell}\not\subset Q$, and 0 otherwise; and $u$ is the rank of $\calO_{Q}^{\times}$.
	\end{lemma}
	
	\begin{remark}
		We take the pro-$\calC$ completions for a finite $\calC$ only because we want to work with a presentation with finitely many generators. If $\Gs(K)$ is finitely generated, then the pro-$\calC$ completions in the inequality can be removed. Also, we can take an ascending series of finite sets $\{\calC_i\}$ such that $\cup\, \calC_i$ contains all isomorphism classes of $|\Gamma|'$-$\Gamma$-groups, then $\varprojlim \Gs(K)^{\calC_i}=\Gs(K)$. We will see in the discussion after the proof of this lemma that only the difference between the upper bounds in the number field case and the function field case is important, which does not depend on the choice of $\calC$. 
	\end{remark}
	
	\begin{proof}
		We denote $G:=G_{\O}(K)$ and $G^{\#}:=\Gs(K)$.
		By \cite[Prop.3.4 and Cor.5.3]{Liu2020}, 
		\begin{equation}\label{eq:mult}
			m(\pi^{\calC}, \Gamma, (G^{\#})^{\calC}, A) \leq \frac{n \dim_{\F_\ell} A - \xi(A) + \dim_{\F_\ell} H^2(G^{\#}, A)^{\Gamma} - \dim_{\F_\ell} H^1(G^{\#}, A)^{\Gamma}}{\dim_{\F_\ell} \Hom_{ (G^{\#})^{\calC}\rtimes \Gamma}(A)}
		\end{equation}
		where the $\Gamma$ actions on the cohomology groups are induced by the conjugation action of $\Gamma$ on $G^{\#}$.
	
		We first prove the function field case. Note that $G^{\#}$ is the quotient of $G$ modulo the decomposition subgroup at $\infty$. So we can assume that $n$ is sufficiently large such that $\pi^{\calC}$ is the composition map $F_n(\Gamma)^{\calC} \overset{\hat{\pi}^{\calC}}{\longrightarrow} G^{\calC} \to (G^{\#})^{\calC}$. Then we have a formula for $m(\hat\pi^{\calC}, \Gamma, G^{\calC}, A)$ similar to \eqref{eq:mult}. Because $K/Q$ is split completely above $\infty$, the genus of $K$ is positive. By \cite[Prop.9.3(2)]{Liu2020}, 
		\begin{eqnarray}
			&& \dim_{\F_\ell} H^2(G, A)^{\Gamma} - \dim_{\F_\ell} H^1(G, A)^{\Gamma} \nonumber\\
			&=& \dim_{\F_{\ell}} \Hom_{\Gal(K_{\O}/Q)} (A, \mu_{\ell}) - \dim_{\F_\ell} A^{\Gal(K_{\O}/Q)}. \label{eq:delta}
		\end{eqnarray}
		Here $A$ is a $(G^{\#})^{\calC} \rtimes \Gamma$-module, and hence is a $\Gal(\Ks/Q)$-module. Because $\Ks/Q$ is split completely over $\infty$, if $A=\mu_{\ell}$, then $\mu_{\ell}$ must be contained in $Q$, so $A$ is a trivial module and \eqref{eq:delta} is $1-1=0$. If $A=\F_{\ell}$, then \eqref{eq:delta} is $-1$ when $\mu_\ell\not \subset Q$ and is 0 otherwise. If $A$ is neither $\F_{\ell}$ or $\mu_{\ell}$, then both terms in \eqref{eq:delta} are 0. So we have an upper bound: $m(\hat\pi^{\calC}, \Gamma, G^{\calC}, A)\leq (n \dim A - \xi(A)-{\bf1}_{\F_\ell}(A)) \cdot (\dim \Hom_{G^{\calC}\rtimes \Gamma}(A))^{-1}$. Recall that $G^{\#}$ is the $\Gamma$-equivariant quotient of $G$ modulo one element (a generator of the decomposition at $\infty$). By following the proof of \cite[Lem.10.6]{Liu2020}, we have $m(\pi^{\calC}, \Gamma, (G^{\#})^{\calC}, A) = m(\hat{\pi}^{\calC}, \Gamma, G^{\calC}, A)+ \dim A \cdot (\dim \Hom_{G^{\calC}\rtimes \Gamma}(A))^{-1}$, and hence we prove the lemma in the function field case. 
		
		When $Q$ is a number field, $G=G^{\#}$ because of the assumption that $\calC$ contains only groups whose order are prime to $|\Gamma||G_{\O}(Q)|$. Let $T$ denote the set of all archimedean primes and all primes of $K$ above $\ell$. Let $S_\infty(Q)$ denote the set of all archimedean primes of $Q$. By \cite[Thm.7.1]{Liu2020},
		\begin{eqnarray*}
			\log_{\ell} \chi_{K/Q, T}(A)&:=&\sum_{i=0}^2 (-1)^i \dim_{\F_\ell} H^i(G_T(K), A)^{\Gamma}\\
			&=& \sum_{v \in S_{\infty}(Q)}  \left( \dim_{\F_\ell} \hH^0(Q_v, \Hom(A, \mu_{\ell})) - \dim_{\F_\ell} H^0 (Q_v, \Hom(A, \mu_{\ell}))\right),
		\end{eqnarray*}
		where $\hH^0$ is the 0-th Tate cohomology of the group $G_{Q_v}$. 
		We let $r_1$ and $r_2$ denote the number of real embeddings and complex embeddings of $Q$ respectively.
		If $v$ is imaginary, then $\dim H^0(Q_v, \Hom(A, \mu_{\ell}))=\dim \Hom(A, \mu_{\ell}) = \dim A$. 
		When $\ell$ is odd, $G_{Q_v}$ is either trivial or $\Gal(\C/\R)=\Z/2\Z$, so the Tate cohomology terms above are always 0 \cite[Prop.(1.6.2)]{NSW}. 
		Since $K_{\O}/Q$ is unramified at every archimedean prime, if $v$ is real, then $A$ as a $G_{Q_v}$-module is trivial but $\mu_{\ell}$ is not trivial, so $H^0(Q_v, \Hom(A, \mu_{\ell}))=\Hom_{G_{Q_v}}(A, \mu_{\ell})=0$. 
		When $\ell=2$, for each real place $v$ of $Q$, $G_{Q_v}$ acts trivially on $A$ and so on $\Hom(A, \mu_2)$. Since $\Hom(A, \mu_2)$ has exponent 2, its $G_{Q_v}$-norm map is zero, which implies that $\hH^0(Q_v, \Hom(A, \mu_2))=H^0(Q_v, \Hom(A, \mu_2))$. When $\ell=2$ and $v$ is imaginary, $G_{Q_v}=1$, so the Tate cohomology term is trivial. Therefore, in every case, we have $\log_{\ell}\chi_{K/Q, T}(A) = -r_2 \dim_{\F_\ell} A$. 
		
		Then by \cite[Prop.9.4]{Liu2020}, 
		\begin{eqnarray}
			&& \dim_{\F_{\ell}} H^2(G, A)^{\Gamma} - \dim_{\F_\ell} H^1(G, A)^{\Gamma} \nonumber \\
			&\leq& -r_2\dim_{\F_\ell} A+ \dim_{\F_{\ell}} \Hom_{\Gal(K_T/Q)}(A, \mu_{\ell}) - \dim_{\F_{\ell}}A^{\Gal(K_{\O}/Q)} + [K:\Q] \dim_{\F_\ell} A. \label{eq:delta-nf}
		\end{eqnarray}
		Similarly to the function field case, if $A=\mu_\ell$, then $\mu_{\ell} \subset Q$ and $\Delta:=\dim \Hom_{\Gal(K_T/Q)}(A, \mu_{\ell})-\dim A^{\Gal(K_{\O}/Q)} =0$. If $A$ is neither $\F_{\ell}$ or $\mu_{\ell}$, then $\Delta=0$. When $A=\F_{\ell}$, $\Delta$ is $0$ if $\mu_{\ell}\subset Q$ and is $-1$ when $\mu_{\ell} \not\subset Q$. Finally, because $[K:\Q]=r_1+2r_2$ and $u=r_1+r_2-1$, the lemma in the number field case follows by \eqref{eq:delta-nf}. 
	\end{proof}
	
	The multiplicities computed in Lemma~\ref{lem:multiplicity} are closely related to the minimal number of relators in the presentation~\eqref{eq:C-Gur} (which is the minimal number of elements that generate $M$ as a $\Gamma$-closed normal subgroup of $F_n(\Gamma)^{\calC}$). Let $R$ and $F$ be the quotients of $M$ and $F_n(\Gamma)^{\calC}$ respectively modulo the intersection of all maximal proper $F_n(\Gamma)^{\calC} \rtimes \Gamma$-normal subgroups of $M$. Then 
	\[
		R\simeq \bigoplus_{A} A^{m(\pi^{\calC}, \Gamma, \Gs(K)^{\calC}, A)},
	\]
	where the direct sum runs over all finite simple $F\rtimes \Gamma$-groups $A$. Then a set of elements generates $M$ as a $\Gamma$-closed normal subgroup of $F_n(\Gamma)^{\calC}$ if and only if the image of the set in $R$ generates $R$ as a $\Gamma$-closed normal subgroup of $F$. By \cite[Thm.5.1 and Rmk.5.2]{Liu-Wood}, for a non-abelian $A$, the direct product of arbitrarily many copies of $A$ can always be generated by 1 element; and for an abelian $A$, the maximal $m$ such that $A^m$ can be generated by 1 element is $\dim_{\F_\ell} A \cdot (\dim_{\F_\ell} \Hom_{\Gs(K)^{\calC} \rtimes \Gamma}(A))^{-1}$. Thus, the upper bounds for the multiplicities given in Lemma~\ref{lem:multiplicity} show that in the pro-$\calC$ $\Gamma$-presentation \eqref{eq:C-Gur}, the number field case needs $u$ more relators than the function field case. If these $u$ extra relators are random, then Conjecture~\ref{conj:nf} follows by Conjecture~\ref{conj:ff} and Lemma~\ref{lem:Qt-moment} (regardless of the ``prime-to-$q$'' condition in the function field case).
	
	Of course, in Lemma~\ref{lem:multiplicity}, we give only upper bounds of the multiplicities, and if there exist better upper bounds, then one may wonder whether Conjecture~\ref{conj:nf} is still reasonable. \cite{Liu2020} shows that the equality in Lemma~\ref{lem:multiplicity} holds if and only if, in the number field case, $H^2(G_{\O}(K), A)$ and $\B_{\O}(K,A)$ are in the same class of the Grothendieck group of $\F_\ell[\Gal(K/Q)]$-modules (see the proof of \cite[Prop.9.4]{Liu2020}). Here $\B_{\O}(K,A)$ is an abelian group defined in \cite[Def.8.1]{Liu2020}.
	We don't know to what extent this condition holds, so it's unclear whether there are upper bounds that are generically better than the ones given in Lemma~\ref{lem:multiplicity}.

\subsection{Evidence II: Positively ramified extensions of number fields}\label{ss:evidence2}
	
	If Conjectures~\ref{conj:ff} and \ref{conj:nf} are true, then one may expect that, for a number field $K$, there exists some group associated to $K$ such that: 1) it has group structure similar to the prime-to-$\Char(k)$ part of the \'etale fundamental group of a smooth projective curve $X$ over a finite field $k$; and 2) $G^{\#}_{\O}(K)$ is a $\Gamma$-equivariant quotient of that group modulo $u$ elements. In the function field case, for $K$ satisfying Notation~\ref{not}, we have seen that $G^{\#}_{\O}(K)$ is naturally a quotient of $\pi_1^{\et}(X_{\overline{k}})=G_{\O}(K\overline{k})$. When we focus on the pro-$\ell$ completions of Galois groups in the number field case, then the passage to the algebraic closure of constant field is analogous to taking the composition of $K$ with the cyclotomic $\Z_\ell$-extension $Q_{\infty}$ of $Q$. However, the \'etale fundamental group $G_{\O}(K)$ is too small to be an analogue of $\pi_1^{\et}(X)$ because $Q_\infty$ is already ramified. Let $S$ be a set of primes of $K$ that contains all the archimedean primes and primes above $\ell$, and let $G_S(K)(\ell)$ denote the Galois group of the maximal pro-$\ell$ extension of $K$ that is unramified outside primes above $S$. It is well known that there exists a strong analogy between $G_S(K)(\ell)$ and $\pi_1^{\et}(X - \{P_1, \ldots, P_m\})(\ell)$ when $\ell\neq \Char(k)$: both of them are Poincar\'e group at $\ell$ of dimension 2 and have the Poitou--Tate duality. Because the Riemann existence theorem shows that there is a natural surjection 
	\[\begin{tikzcd}
		\pi_1^{\et}(X-\{P_1, \ldots, P_m\})(\ell) \arrow[two heads]{r} &\pi_1^{\et}(X)(\ell)
	\end{tikzcd}\]
	whose kernel is the free product of the inertia subgroups at $P_1, \ldots, P_m$. So if the analogue of $\pi_1^{\et}(X)(\ell)$ exists, then it should be some group between $G_S(K)(\ell) \twoheadrightarrow G_{\O}(K)(\ell)$. 
	
	In the sequence of works by Wingberg and Schmidt \cite{Wingberg1984, Wingberg1985, Schmidt1995, Schmidt1996}, they introduced the notion of positively ramified extensions of number fields, and in some nice cases, they prove that there is a good analogy between the Galois group of the maximal positively ramified extension and $\pi_1^{\et}(X)$. We briefly explain their ideas below.
	Let $\ell$ be an odd prime number, and $\Q_\ell^{pre}:= \Q_\ell^{nr}(\zeta_\ell+\zeta_\ell^{-1})(odd)(\zeta_\ell)$, where $\Q_\ell^{nr}$ is the maximal unramified extension of $\Q_\ell$ and the notation $\Q_\ell^{nr}(\zeta_\ell+\zeta_\ell^{-1})(odd)$ means the maximal pro-odd extension of $\Q_\ell^{nr}(\zeta_\ell+\zeta_\ell^{-1})$. So $\Q_\ell^{pre}$ is a $\Z/2\Z$-extension of $\Q_\ell^{nr}(\zeta_\ell+\zeta_\ell^{-1})(odd)$, and we let $\rho$ denote the non-trivial element of its Galois group, which can be viewed as the $\ell$-adic version of the complex conjugation. A local field $k/\Q_\ell$ is \emph{orientable} if $k \subset \Q_\ell^{pre}$ and if either $k=k^{+}:=k^{\langle{\rho}\rangle}$ or $\zeta_\ell \in k^{nr}$. In other words, at each local place, Wingberg and Schmidt construct the structure that is similar to a CM-field. So the positively ramified extension is the analogue of the totally real field: a local extension $l/k$ is \emph{positively ramified (p.r.)} if $l \subset \Q_\ell^{pre}k$ (in $\Q_2$ if $p=2$) and the Galois closure of $l/k$ has at most pure wild ramification (i.e., the ramification index is a power of $\ell$). Then a number field extension $L/K$ is called \emph{positively ramified (p.r.)} if its local completion is p.r. at each prime.

	For an odd prime number $\ell$, a number field $K$ is called \emph{$\ell$-CM} if, at every prime $\frakp$ of $K$ lying above $\ell$, $K_{\frakp}$ is orientable, $[K_{\frakp}:K_{\frakp}^+]=2$ and $K_{\frakp}$ contains an index-2 subfield $k$ such that $k=k^{+}$. We let $K_{\pos}$ denote the maximal p.r. extension of $K$. Then by \cite[Thm.7.3]{Schmidt1996}, if $K$ is $\ell$-CM and contains $\mu_{\ell}$, then the followings hold (for analogous statements for $\pi_1^{\et}(X)$, see \cite[Cor.(10.1.3)(i)(ii)]{NSW}):
	\begin{enumerate}
		\item If $g_{\ell}(K)>0$, then $\Gal(K_{\pos}/K)(\ell)$ is a Poincar\'e group of dimension 3 with dualizing module $\mu_{\ell^\infty}$.
		\item If $g_{\ell}(K)=0$, then $\Gal(K_{\pos}/K)(\ell)$ is a free pro-$\ell$ group.
	\end{enumerate}
	Here $g_{\ell}(K)$ is the \emph{$\ell$-genus} which can be computed by the rank of the positive ramified Picard group \cite[Prop.6.7]{Schmidt1996}. Moreover, Schmidt \cite[Thm.8.3]{Schmidt1996} proved the analogue of the Riemann existence theorem: kernel of $G_S(K)(\ell) \twoheadrightarrow \Gal(K_{\pos}/K)(\ell)$ is the free product of inertia subgroups.
	
	Assume that $\ell$ is an odd prime number, $Q=\Q(\mu_{\ell})$ such that $Q_{\O}(\ell)=Q$ (equivalently, the $\ell$-part of the class group of $Q$ is trivial), and $\Gamma$ is a finite group with $\ell\nmid |\Gamma|$. Then $Q$ is $\ell$-CM. Assume that $K$ is a $\Gamma$-extension of $Q$ that is $\ell$-CM (for example, this holds when $K/Q$ is unramified at all primes lying above $\ell$). We let $K_{\pos}(\ell)$ denote the maximal pro-$\ell$ p.r. extension of $K$. Then $KQ_{\infty}$ is a p.r. extension of $K$, and we have the following short exact sequence that is the analogue of \eqref{eq:etale-es}
	\begin{equation}\label{eq:positive-es}
		1 \longrightarrow \Gal(K_{\pos}(\ell) / KQ_{\infty}) \longrightarrow \Gal(K_{\pos}(\ell)/K) \longrightarrow \Gal(KQ_{\infty}/K) \longrightarrow 1,
	\end{equation}
	where every term naturally has an action of $\Gamma$ (when we fix a splitting of $\Gal(K_{\pos}(\ell)/Q) \to \Gal(K/Q)$).
	The Galois group $\Gal(KQ_{\infty}/K)\simeq \Z_{\ell}$, which is a Poincar\'e group of dimension 1 with dualizing module $\Q_{\ell}/\Z_{\ell}$. So by \cite[Thm.(3.7.4)]{NSW}, from the exact sequence above, we see that $\Gal(K_{\pos}(\ell)/KQ_{\infty})$ is a pro-$\ell$ Poincar\'e group of dimension 2 with dualizing module $\mu_{\ell^{\infty}}$, and hence we have the analogue of Poincar\'e duality. 
	Moreover, since $\mu_{\ell^{\infty}} \subset K Q_{\infty}$, by Kummer theory, the abelianization of $\Gal(K_{\pos}(\ell)/KQ_{\infty})$ is free over $\Z_{\ell}$, so $\Gal(K_{\pos}(\ell)/KQ_{\infty})$ is a pro-$\ell$ Demu\v{s}kin group in the form of $\langle x_1, \ldots, x_{2n} \mid [x_1, x_2]\cdots [x_{2n-1}, x_{2n}] \rangle$.
	So one can repeat the argument in the proof of Lemma~\ref{lem:Galois-H2} to show that there is a $\Gal(Q_{\infty}/Q)$-equivariant isomorphism
	\begin{equation}\label{eq:pos-H2-mu}
		H_2(\Gal(K_{\pos}(\ell)/Q_{\infty}), \Z)(\ell)\simeq \Z_{\ell}(1).
	\end{equation}

	We let $v$ denote the unique prime of $Q$ lying above $\ell$, let $\frakp$ be a prime of $K$ lying above $v$, and pick a prime $\frakP$ of $K_{\pos}(\ell)$ above $\frakp$. By \cite[Thm.(7.4.4)]{NSW} and the assumption $\ell \nmid |\Gamma|$, the group of principal units of $K_{\frakp}$ is isomorphic to 
	\[
		U^1_{K_{\frakp}} \simeq \Z_\ell[\Gal(K_{\frakp}/Q_v)]^{\ell-1}
	\]
	as $\Gal(K_{\frakp}/Q_v)$-modules. Let $a$ denote a generator of $\Gal(KQ_{\infty}/K)$. The extension $KQ_{\infty}/K$ is completely widely ramified at $\frakp$. Then by local class field theory, we can pick an element $a_{\frakp}$ of the inertia subgroup $\Gal(K_{\pos}(\ell)/K)$ at $\frakP$, such that it maps to $a$ under the quotient map in \eqref{eq:positive-es} and the $\Gal(K_{\frakp}/Q_v)$-closed subgroup of the abelianization of $\Gal(K_{\pos}(\ell)_{\frakP}/K_{\frakp})$ generated by the image of $a_{\frakp}$ is isomorphic to $\Z_\ell[\Gal(K_{\frakp}/Q_v)]$. Then
for other prime $\frakp'$ of $K$ above $v$, we choose an element $\gamma\in \Gal(K/Q)$ such that $\gamma(\frakp)=\frakp'$, and define $a_{\frakp'}$ to be $\gamma(a_{\frakp})$. 
	We define $G_{\pos}^{\#}(K)$ to be the quotient of $\Gal(K_{\pos}(\ell)/KQ_{\infty})$ modulo the $\Gamma$-closed normal subgroup generated by these elements $a_{\frakp}$'s for all $\frakp \mid v$, and we denote the field extension corresponding to $G_{\pos}^{\#}(K)$ by $K^{\#}_{\pos}/KQ_{\infty}$. The group $G_{\pos}^{\#}(K)$ is designed in the way such that it is a reasonable analogue of $\Gs(K)$ in Notation~\ref{not}. Indeed, we have to take the above approach to quotient out all the relevant lifts of $\Gal(KQ_{\infty}/K)$ (which is the analogue of the Frobenius), because in the function field case the place $\infty$ of $X_{\overline{k}}$ cannot further ramify or be further inert. 
	For example, if $v$ splits completely in $K/Q$, then the elements $a_{\frakp}$'s are lifts of $a$ that are compatible with the $\Gal(K/Q)$ action on all the primes above $v$, which is exactly the same situation as the function field case. 
	So the quotient map $\Gal(K_{\pos}(\ell)/Q_{\infty}) \twoheadrightarrow \Gal(K_{\pos}^{\#}/Q_{\infty})$ induces a coinflation map, which gives the analogue of the invariant $\omega_K^{\#}$ in Definition~\ref{def:inv-L/K}:
	\[
		\Z_{\ell}(1) \simeq H_2(\Gal(K_{\pos}(\ell)/Q_{\infty}), \Z)(\ell) \longrightarrow H_2(\Gal(K_{\pos}^{\#}/Q_{\infty}), \Z)(\ell)
	\]
	and it's clear by our construction that this map factors through $\mu_{\ell}$.
	We cannot recover the group theoretical description of the invariant as described in Proposition~\ref{prop:etale-lift} and Corollary~\ref{cor:compatible-cover} because we don't have the analogue of the system of inertia generators for $\Gamma$-extensions of $Q_{\infty}$. However, we can generalize the cohomological definition of the invariant in \S\ref{ss:coh-def-inv}: we can define a trace map $H^2(\Gal(K_{\pos}(\ell)/KQ_{\infty}), \Z_{\ell}(1) )\to \Z_{\ell}$ and choose the isomorphism \eqref{eq:pos-H2-mu} such that it corresponds to the class of trace $-|\Gamma|$.

	By the discussion above, it's reasonable to consider $G_{\pos}^{\#}(K)$ for a number field $K$ satisfying our assumptions as a good analogue of $\Gs(K)$ for a function field $K$. Then, for the number field case, we consider the natural quotient map
	\begin{equation}\label{eq:Pos-Nr}
		G^{\#}_{\pos}(K) \longrightarrow G_{\O}(K)(\ell),
	\end{equation}
	whose kernel is the normal subgroup generated by the inertia subgroups of $K_{\pos}^{\#}/K$ at primes above $v$. Let $\frakp$ be a prime of $K$ above $v$, and $\frakP$ a prime of $K_{\pos}^{\#}$ above $\frakp$. By definition of p.r., the local completion $(K_{\pos}^{\#})_{\frakP}/K_{\frakp}$ is a sub-$\ell$-extension of $\Q_{\ell}^{pre}/K_{\frakp}$. As we assumed that $K$ is $\ell$-CM,  $K_{\frakp}^{+}$ has index 2 in $K_{\frakp}$ and $\Gal(K_{\frakp}^+/Q_v^+) \cong \Gal(K_{\frakp}/Q_v)$. Recall that $\rho$ is the $\ell$-complex conjucation, we let $(U_{K_{\frakp}}^1)^+$ denote the $\rho$-invariant subgroup of the principal unit group $U_{K_{\frakp}}^1$. Then since $U_{K_{\frakp}}^1 \simeq \Z_\ell[\Gal(K_{\frakp}/\Q_{\ell})]$ as $\Gal(K_{\frakp}/\Q_{\ell})$-modules, we have the following isomorphism of $\Gal(K_{\frakp}/Q_v)$-modules
	\[
		(U_{K_{\frakp}}^1)^+ \simeq \Z_\ell[\Gal(K_{\frakp}^+/\Q_{\ell})] \simeq \Z_\ell[\Gal(K_{\frakp}^+/Q_v^+)]^{\frac{\ell-1}{2}}.
	\]
	So by local class field theory, the maximal pro-$\ell$ abelian subextension of $K_{\frakp}$ inside $\Q_{\ell}^{pre}$ has Galois group isomorphic to $\Z_{\ell}[\Gal(K_{\frakp}/Q_v)]^{\frac{\ell-1}{2}}\oplus \Z_{\ell}$, where the extra term $\Z_{\ell}$ comes from the unramified part. Recall that when we define $G_{\pos}^{\#}(K)$, we quotient out a copy of $\Z_\ell[\Gal(K_{\frakp}/Q_v)]$ that is generated by $a_{\frakp}$. So we conclude that the Galois group of the maximal abelian subextension of $(K_{\pos}^{\#})_{\frakP}/K_{\frakp}$ is a quotient of $\Z_{\ell}[\Gal(K_{\frakp}/Q_v)]^{\frac{\ell-3}{2}} \oplus \Z_{\ell}$. By the Burnside basis theorem for pro-$\ell$ groups, the inertia subgroup of $(K_{\pos}^{\#})_{\frakP}/K_{\frakp}$ is the $\Gal(K_{\frakp}/Q_v)$-closed normal subgroup generated by $\frac{\ell-3}{2}$ well-chosen elements.
	Finally, we look at all the primes above $v$, and obtain that the kernel of the quotient map \eqref{eq:Pos-Nr} is the $\Gal(K/Q)$-closed normal subgroup generated by $\frac{\ell-3}{2}$ elements. The rank of $\calO_Q^{\times}$ is exactly $\frac{\ell-3}{2}$.
	
	The construction above can be applied to other cases, for example, to the case that $Q=\Q(\mu_{\ell^m})$ and $K$ is $\ell$-CM. However, the situations that can be covered are rare because of the $\ell$-CM condition. It would be interesting to understand how to generalize the definition of $\ell$-CM to cover more cases.

\appendix

\section{A random group model for random 3-manifolds}\label{sect:Appendix}

	In the work \cite{Dunfield-Thurston}, Dunfield and Thurston studied random orientable 3-manifolds constructed by a random Heegaard splitting, which means gluing two handlebodies by a random walk in the mapping class group of a surface. They compared the probabilities that the random 3-manifolds have finite covers of particular kinds to the analogous probabilities coming from random balanced presentations, and concluded that the fundamental groups of random 3-manifolds have many more finite quotients than the random group model defined by random balanced presentations. In this section, we use the results from Section~\ref{sect:RandomGroup} to construct a random group model to simulate the fundamental groups of 3-manifolds.
	
	Let $\Sigma_g$ be a closed surface of genus $g$, and let $\calM_g$ be its mapping class group. Then the fundamental group $\pi_1(\Sigma_g)$ of $\Sigma_g$ has a presentation $\langle a_1, \ldots, a_g, b_1, \ldots, b_g \mid [a_1,b_1]\cdots[a_g, b_g] \rangle$, and $\calM_g$ is naturally isomorphic to an index-2 subgroup of $\Out(\pi_1(\Sigma_g))$. For an element $\varphi \in \calM_g$, the 3-manifold $M$ obtained by gluing two copies of $\Sigma_g$ by $\varphi$ has fundamental group 
	\[
		\pi_1(M)=\faktor{\pi_1(\Sigma_g)}{[a_1, \ldots, a_g, \phi(a_1), \ldots, \phi(a_g)]},
	\]
	where $\phi \in \Aut( \pi_1(\Sigma_g))$ is a lift of $\varphi$, and $[a_1, \ldots, a_g, \phi(a_1), \ldots, \phi(a_g)]$ means the normal subgroup of $\pi_1(\Sigma_g)$ generated by those elements. Also, the map from $\Sigma_g$ to the classifying space $B\pi_1(\Sigma_g)$ defines a homology class in $H_2(\pi_1(\Sigma_g), \Z)$, which is invariant under the action of $\calM_g$. 
	For a finite group $H$, Dunfield and Thurston showed that, as $g\to \infty$, the expected number of covers with covering group $H$ of the random 3-manifold $M$ goes to $|[H,H]| \cdot |H_2(H,\Z)| / |\Aut(H)|$ (\cite[Thm.6.21]{Dunfield-Thurston}). Therefore, when we compute the number of surjections from $\pi_1(M)$ to $H$, we have
	\[
		\lim_{g\to \infty}\EE(\# \Sur(\pi_1(M), H)) = |[H,H]| \cdot |H_2(H,\Z)|.
	\]
	Recall that, $H_2(\pi_1(\Sigma_g), \Z)$ can be identified with the kernel of a Schur covering of $\pi_1(\Sigma_g)$. We will explicitly give a pro-$\ell$ random group model that have the above moments for all $\ell$-groups $H$.
		
	Let $\ell$ be a prime number. We use the notation in Section~\ref{sect:RandomGroup} to let $\calG_n$ denote the pro-$\ell$ Demu\v{s}kin group whose abelianization is $\Z_{\ell}^{2n}$ and let 
	\[
		1 \longrightarrow \Z_{\ell} \longrightarrow \widetilde{\calG}_n  \overset{\pi_n}{\longrightarrow} \calG_n \longrightarrow 1
	\]
	denote the Schur covering of $\calG_n$. 
	We use the presentation $\calG_n=\langle a_1, \ldots, a_{n}, b_1, \ldots, b_n \mid [a_1, b_1]\cdots [a_{n}, b_{n}]\rangle$, and let $\xi_n \in H_2(\calG_n, \Z) \simeq \ker \pi_n$ denote the image of $[a_1,b_1]\cdots [a_n, b_n]$. 
	
	\begin{proposition}\label{prop:3-manifold}
		Consider the random group 
		\[
			Z_{n}:= \faktor{\widetilde{\calG}_n}{[a_1, \ldots, a_n, \phi(a_1), \ldots, \phi(a_n)]}
		\]
		where $\phi$ is a random automorphism of $\widetilde{\calG}_n$ that preserves and acts trivially on $\ker \pi_n$. Then for any finite $\ell$-group $H$, we have
		\[
			\lim_{n \to \infty} \EE \left( \# \Sur(Z_n, H)\right) = |[H,H]|\cdot |H_2(H, \Z)|.
		\]
	\end{proposition}
	
	\begin{proof}
	
		For an automorphism $\phi$ of $\widetilde{\calG}_n$, a surjection from $Z_n \to H$ defines a surjection $\widetilde{\calG}_n \to Z_n \to H$ whose kernel contains $a_i$ and $\phi(a_i)$ for all $i$. So we have 
		\[
			\EE \left( \# \Sur(Z_n, H)\right) 
			= \sum_{\substack{\pi \in \Sur(\widetilde{\calG}_n, H) \\ \text{s.t. } \pi(a_i)=1 \, \forall i}} \Prob \left( \phi \in \Aut(\widetilde{\calG}_n, \pi_n ; 1) \,\Big| \, \phi(a_i) \in \ker \pi, \,  \forall i \right),
		\]
		where $\Aut(\widetilde{\calG}_n, \pi_n; 1)$ denote the closed subgroup of $\Aut(\widetilde{\calG}_n)$ consisting of all automorphisms that preserve and act trivially on $\ker \pi_n$.
		Each $\pi$ as described in the sum factors through a surjection $\varpi \in \Sur(\calG_n, H)$ because $\pi([a_1, b_1]\cdots[a_n,b_n])=1$.
		By Lemma~\ref{lem:Labute}, $\Aut(\widetilde{\calG}_n, \pi_n; 1)$ acts transitively on the set $\left\{ \varpi \in \Sur(\calG_n, H) \mid \varpi_*(\xi_n)=1 \right\}$. 
		Since $\phi(a_i)\in \ker \pi$ if and only if $\pi \circ \phi (a_i)=1$, we have
		\[
			\Prob  \left( \phi \in \Aut(\widetilde{\calG}_n, \pi_n ; 1) \,\Big| \, \phi(a_i) \in \ker \pi, \,  \forall i \right) 
			= \frac{\#\{ \rho \in \Sur(\widetilde{\calG}_n ,H) \mid \rho(a_i)=1, \, \forall i\}}{ \# \{ \rho \in \Sur(\calG_n, H) \mid \rho_{*} (\xi_n)=1 \}}.
		\]
		The numerator equals the number of surjections from a free group on generators $b_i$'s to $H$, so it is approximately $|H|^n$ for large $n$. Let $F_{2n}$ denote the free pro-$\ell$ group on the generators $a_1, \ldots, b_n$, and $\widetilde{H} \to H$ is a Schur covering of $H$. Homomorphisms from $F_{2n}$ to $H$ lift to homomorphisms from $F_{2n}$ to $\widetilde{H}$. As $n \to \infty$, considering all homomorphisms from $F_{2n}$ to $H$, the image of $[a_1,b_1] \cdots [a_n, b_n]$ is nearly uniformly distributed in $[\widetilde{H}, \widetilde{H}]$. So the denominator above is approximately $|H|^{2n}/ |[\widetilde{H}, \widetilde{H}]| = |H|^{2n} / \left( |[H,H]| \cdot |H_2(H, \Z)|\right)$. Therefore, we obtain
		\[
			\EE(\#\Sur(Z_n, H)) \sim |H|^n \cdot  \frac{|H|^n}{|H|^{2n} / (|[H,H]| \cdot |H_2(H,Z)|) },
		\]
		and the proposition follows after taking $n \to \infty$. 
	\end{proof}

\def\cprime{$'$}

\end{document}